\pdfoutput=1
\documentclass[12pt]{amsart}
\usepackage{lmodern}
\usepackage{ifthen}
\usepackage{amsfonts}
\usepackage{amsxtra}
\usepackage{amssymb}
\usepackage{mathdots}
\usepackage{array}
\usepackage[margin=1.0in]{geometry}
\usepackage{xcolor}
\definecolor{cite}{rgb}{0.30,0.60,1.00}
\definecolor{url}{rgb}{0.00,0.00,0.80}
\definecolor{link}{rgb}{0.40,0.10,0.20}
\usepackage[colorlinks,linkcolor=link,urlcolor=url,citecolor=cite,pagebackref,breaklinks]{hyperref}
\usepackage{bbm}
\usepackage{mathtools}
\usepackage{mathrsfs}
\usepackage{appendix}
\usepackage[all]{xy}
\usepackage[lite,abbrev,msc-links,alphabetic]{amsrefs}

\usepackage{graphicx}
\usepackage{multirow}
\usepackage{pstricks}
\usepackage{pst-pdf}
\usepackage{enumitem}
\usepackage{pifont}
\usepackage{marvosym}
\usepackage{tensor}

\RequirePackage{xspace}
\RequirePackage{tensor}

\usepackage[OT2,T1]{fontenc}
\DeclareSymbolFont{cyrletters}{OT2}{wncyr}{m}{n}
\DeclareMathSymbol{\Sha}{\mathalpha}{cyrletters}{"58}

\usepackage{graphicx}

\makeatletter
\providecommand*{\Dashv}{%
  \mathrel{%
    \mathpalette\@Dashv\vDash
  }%
}
\newcommand*{\@Dashv}[2]{%
  \reflectbox{$\m@th#1#2$}%
}
\makeatother

\setitemize[0]{leftmargin=20pt}
\setenumerate[0]{leftmargin=30pt}

\numberwithin{equation}{section}

\theoremstyle{plain}
\newtheorem{proposition}{Proposition}[section]

\newtheorem{corollary}[proposition]{Corollary}
\newtheorem{lemma}[proposition]{Lemma}
\newtheorem{theorem}[proposition]{Theorem}

\theoremstyle{definition}
\newtheorem{definition}[proposition]{Definition}

\theoremstyle{remark}
\newtheorem{remark}[proposition]{Remark}
\newtheorem{example}[proposition]{Example}

\renewcommand{\c}[1]{\mathcal{#1}}
\renewcommand{\d}[1]{\mathbb{#1}}
\newcommand{\f}[1]{\mathfrak{#1}}
\renewcommand{\r}[1]{\mathrm{#1}}

\renewcommand{\(}{\left(}
\renewcommand{\)}{\right)}
\newcommand{\res}{\mathbin{|}}

\renewcommand{\leq}{\leqslant}
\renewcommand{\geq}{\geqslant}

\newcommand{\cC}{\c C}

\newcommand{\cE}{\c E}
\newcommand{\cF}{\c F}

\newcommand{\cN}{\c N}
\newcommand{\cO}{\c O}
\newcommand{\cP}{\c P}

\newcommand{\cS}{\c S}

\newcommand{\cY}{\c Y}
\newcommand{\cZ}{\c Z}

\newcommand{\dD}{\d D}
\newcommand{\dE}{\d E}

\newcommand{\dL}{\d L}
\newcommand{\dM}{\d M}
\newcommand{\dN}{\d N}

\newcommand{\dQ}{\d Q}

\newcommand{\dV}{\d V}

\newcommand{\dX}{\d X}

\newcommand{\dZ}{\d Z}

\newcommand{\fm}{\f m}

\newcommand{\rd}{\,\r d}

\newcommand{\CF}{\mathbbm{1}}

\newcommand{\Den}{\r{Den}}

\newcommand{\Herm}{\r{Herm}}

\newcommand{\Int}{\r{Int}}

\newcommand{\xra}{\xrightarrow}

\newcommand{\lrarr}{\longrightarrow}
\newcommand{\ovS}{\overline{S}}

\newcommand{\la}{\langle}

\newcommand{\OFb}{{O_{\breve F}}}
\newcommand{\Fb}{\breve{F}}
\newcommand{\q}{q}
\newcommand{\ubE}{\overline{\mathbb{E}}}
\newcommand{\ucE}{\overline{\mathcal{E}}}

\newcommand{\aform}{\ensuremath{\langle\text{~,~}\rangle}\xspace}
\newcommand{\sform}{\ensuremath{(\text{~,~})}\xspace}

\newcommand{\OGr}{\mathrm{OGr}}

\newcommand{\isoarrow}{%
	\ifbool{@display}{\overset{\sim}{\longrightarrow}}{\xrightarrow\sim}%
}

\newcommand{\xla}[2][]{%
	\ifbool{@display}%
	{\settowidth{\olen}{$\overset{#2}{\longleftarrow}$}%
		\settowidth{\ulen}{$\underset{#1}{\longleftarrow}$}%
		\settowidth{\xlen}{$\xleftarrow[#1]{#2}$}%
		\ifdimgreater{\olen}{\xlen}%
		{\underset{#1}{\overset{#2}{\longleftarrow}}}%
		{\ifdimgreater{\ulen}{\xlen}%
			{\underset{#1}{\overset{#2}{\longleftarrow}}}
			{\xleftarrow[#1]{#2}}}}%
	{\xleftarrow[#1]{#2}}
}

\DeclareMathOperator{\diag}{diag}

\DeclareMathOperator{\Fil}{Fil}

\DeclareMathOperator{\GL}{GL}

\DeclareMathOperator{\Hom}{Hom}

\DeclareMathOperator{\Ker}{ker}

\DeclareMathOperator{\Lie}{Lie}

\DeclareMathOperator{\Nm}{Nm}

\DeclareMathOperator{\SO}{SO}

\DeclareMathOperator{\Spec}{Spec}
\DeclareMathOperator{\Spf}{Spf}

\DeclareMathOperator{\tr}{tr}

\DeclareMathOperator{\val}{val}

\DeclareFontFamily{U}{matha}{\hyphenchar\font45}
\DeclareFontShape{U}{matha}{m}{n}{
	<5> <6> <7> <8> <9> <10> gen * matha
	<10.95> matha10 <12> <14.4> <17.28> <20.74> <24.88> matha12
}{}
\DeclareSymbolFont{matha}{U}{matha}{m}{n}
\DeclareFontFamily{U}{mathx}{\hyphenchar\font45}
\DeclareFontShape{U}{mathx}{m}{n}{
	<5> <6> <7> <8> <9> <10>
	<10.95> <12> <14.4> <17.28> <20.74> <24.88>
	mathx10
}{}
\DeclareSymbolFont{mathx}{U}{mathx}{m}{n}

\DeclareMathSymbol{\obot}         {2}{matha}{"6B}

\begin{document}

\title{A Kudla--Rapoport Formula for Exotic Smooth Models of Odd Dimension}

\author{Haodong Yao}
\address{Department of Mathematics, Columbia University, New York NY 10027, United States}
\email{haodong@math.columbia.edu}

\date{\today}

\begin{abstract}
  In this article, we prove a Kudla--Rapoport conjecture for $\cY$-cycles on exotic smooth unitary Rapoport--Zink spaces of odd arithmetic dimension, i.e.\ the arithmetic intersection numbers for $\cY$-cycles equals the derivatives of local representation density. We also compare $\cZ$-cycles and $\cY$-cycles on these RZ spaces. The method is to relate both geometric and analytic sides to the even dimensional case and reduce the conjecture to the results in \cite{LL22}.
\end{abstract}

\maketitle

\tableofcontents

\section{Introduction}
\label{ss:introduction}


\subsection{Background}

The classical \emph{Siegel--Weil formula} (\cites{Sie35,Siegel1951,Weil1965}) relates   special \emph{values} of certain Eisenstein series with theta functions, which are generating series of representation numbers of quadratic forms. Later on, Kudla (\cites{Kudla97, Kudla2004}) proposed an influential program and introduced analogues of theta series in arithmetic geometry. One of the goals of the program is to prove the so-called \emph{arithmetic Siegel--Weil formula} relating the \emph{central derivative} of certain Eisenstein series with a certain arithmetic analogue of theta functions, which are generating series of arithmetic intersection numbers of $n$ special divisors on Shimura varieties associated to $\SO(n-1,2)$ or $\mathrm{U}(n-1,1)$.

For  $\mathrm{U}(n-1,1)$-Shimura varieties,  Kudla and Rapoport (\cite{KR11}) formulated a conjectural  \emph{local arithmetic Siegel--Weil formula} at an \emph{unramified} place with hyperspecial level, now known as the  \emph{Kudla--Rapoport conjecture}. As a local analogue of the arithmetic Siegel--Weil formula, it relates the \emph{central derivative} of local densities of hermitian forms with the arithmetic intersection number of special cycles on unitary Rapoport--Zink spaces. This conjectural identity was recently proved by Li and Zhang in \cite{LZ}. We refer the readers to the introduction of \cite{LZ} for more backgrounds and related results. 

One of the distinguished features of the hyperspecial case \cite{KR11} is that the corresponding Rapoport--Zink space has good reduction. Accordingly, the analytic side has a clear formulation. A natural and important question is to formulate and prove analogues of the Kudla--Rapoport conjecture when  the level structure is nontrivial, where many unexpected new phenomena occur.

At a ramified place,
there are two well-studied unitary Rapoport--Zink spaces with different level structures. One of them is the \emph{exotic smooth model} which has good reduction, and the other one is the \emph{Kr\"amer model} which has bad (semistable) reduction. The analogue of Kudla--Rapoport conjecture for the even dimensional exotic smooth model was formulated and proved by Li and Liu in \cite{LL22} using a strategy similar to \cite{LZ}. For the Kr\"amer model, the analytic side is more involved. In fact, even the formulation of the conjecture is not clear and needs to be modified. This phenomenon in the presence of bad reduction was first discovered by Kudla and Rapoport in \cite{KRshimuracurve} via explicit computation in their study of the Drinfeld $p$-adic half plane.  The Kudla--Rapoport conjecture for Kr\"amer models, in general, was formulated in \cite{HSY3} with conceptual formulation for the modification and proved in \cite{HLSY}.

The present paper focuses on Kudla--Rapoport conjectures for the odd dimensional exotic smooth model. We propose and prove a Kudla--Rapoport conjecture for $\cY$-cycles. The main results we obtained can be used to relax the local assumption at ramified places in the arithmetic Siegel--Weil formula in \cite{LZ}. It should also be useful in extending the co-rank $1$ arithmetic Siegel--Weil formula established in \cite{RyanChen} to odd dimensional exotic smooth models. Lastly, it may be applied to the mixed arithmetic theta lifting and mixed arithmetic inner product formula proposed in \cite{LiuMixed}.

\subsection{Main results}

Let $p$ be an odd prime. Let $F_0$ be a finite extension of $\mathbb{Q}_p$ with residue field $k=\mathbb{F}_q$. Let $\bar{k}$ be a fixed algebraic closure of $k$. Let $F$ be a ramified quadratic extension of $F_0$. Denote by $a\mapsto \bar{a}$ the (nontrivial) Galois involution of $F/F_0$. Let $\pi$ be a uniformizer of $F$ such that $\bar{\pi}=-\pi$. Let $\pi_0=\pi^2$, a uniformizer of $F_0$. Let $\breve F$ be the completion of the maximal unramified extension of $F$. Let $ O_F, \OFb$ be the ring of integers of $F,\Fb$ respectively.

Let $n\geq 1$ be an integer. To define the exotic smooth unitary Rapoport--Zink space, we fix a hermitian formal $O_F$-module $\mathbb{X}=\dX_n$ of signature $(1,n-1)$ over $\bar{k}$. The Rapoport--Zink space $\cN=\mathcal{N}_n$ is the formal scheme over $\Spf O_{\breve F}$ parameterizing hermitian formal $O_F$-modules $X$ of signature $(1,n-1)$ (see Definition \ref{herm formal module}) within the quasi-isogeny class of $\mathbb{X}$. The space $\mathcal{N}$ is formally locally of finite type, formally smooth of relative dimension $n-1$ over $\Spf \OFb$.  

Let $\ubE$ be the framing hermitian formal $O_F$-module of signature $(0,1)$ over $\bar{k}$. We define the \emph{space of quasi-homomorphisms} to be $\dV=\mathbb{V}_n\coloneqq \Hom_{O_F}(\ubE, \mathbb{X}) \otimes_{O_F}F$. We can associate $\dV$ with  a natural $F/F_0$-hermitian form $h$ to make $(\mathbb{V},h)$ a nonsplit nondegenerate $F/F_0$-hermitian space of dimension $n$.  For any nonzero $x\in \mathbb{V}$, we define the \emph{special cycle} $\mathcal{Z}(x)=\cZ_n(x)$ (resp. $\mathcal{Y}(x)=\cY_n(x)$ ) (see Definition \ref{z and y cycles}) to be the deformation locus of $x$ (resp. $\lambda\circ x$) in $\cN$.

Given  an $O_F$-lattice $L\subset \mathbb{V}$ of full rank $n$, we can define integers: the \emph{arithmetic intersection number} $\Int_{n,\cZ}(L)$, $\Int_{n,\cY}(L)$ and the \emph{derived local density} $\partial\Den(L)$.

\begin{definition}
	Let $L\subset \mathbb{V}$ be an $O_F$-lattice and $x_1,\ldots, x_n$ be an $O_F$-basis of $L$. When $n$ is even, we define the \emph{arithmetic intersection number} for $\cZ$-cycles
	\begin{equation}
	\label{eq:IntL}
	\mathrm{Int}_{n,\cZ}(L)\coloneqq \chi(\mathcal{N},\mathcal{O}_{\mathcal{Z}(x_1)} \otimes^\mathbb{L}\cdots \otimes^\mathbb{L}\mathcal{O}_{\mathcal{Z}(x_n)} )\in \mathbb{Z},
	\end{equation}
	where $\mathcal{O}_{\mathcal{Z}(x_i)}$ denotes the structure sheaf of the special divisor $\mathcal{Z}(x_i)$, $\otimes^\mathbb{L}$ denotes the derived tensor product of coherent sheaves on $\mathcal{N}$, and $\chi$ denotes the Euler--Poincar\'e characteristic. By \cite{LL22}*{Corollary 2.35},   $\mathrm{Int}_{n,\cZ}(L)$ is independent of the choice of the basis $x_1,\ldots, x_n$ and hence is a well-defined invariant of $L$ itself.
	
	For general $n$ we define similarly the \emph{arithmetic intersection number} for $\cY$-cycles
	\begin{equation}
	\mathrm{Int}_{n,\cY}(L)\coloneqq \chi(\mathcal{N},\mathcal{O}_{\mathcal{Y}(x_1)} \otimes^\mathbb{L}\cdots \otimes^\mathbb{L}\mathcal{O}_{\mathcal{Y}(x_n)} )\in \mathbb{Z}.
	\end{equation}
	By Proposition \ref{y=pi z even} and Proposition \ref{geometric reduction}, $\mathrm{Int}_{n,\cY}(L)$ is independent of the basis $x_1,\ldots,x_n$ and hence is a well-defined invariant of $L$.
\end{definition}

To define the \emph{derived local density} $\partial\Den(L)$, we need to introduce local densities first. Let $L$ be a hermitian $O_F$-lattice of rank $n$. Let $M$ be another hermitian $O_F$-lattice  (of arbitrary rank) and $\Herm_{L,M}$  denote the $O_{F_0}$-scheme of hermitian $O_F$-module homomorphisms from $L$ to $M$.   Then we  define the corresponding \emph{local density}  to be 
$$\Den(M,L)\coloneqq \lim_{d\rightarrow +\infty}\frac{|\Herm_{L,M}(O_{F_0}/\pi_0^{d})|}{q^{d\cdot d_{L,M}}},$$ 
where $d_{L,M}$ is the dimension of $\Herm_{L,M} \otimes_{O_{F_0}}F_0$.  Let $H$ be the self-dual hermitian $O_F$-lattice of rank $2$ (see Definition \ref{defn for selfdual and unimodular}). It is well-known that  there exists a \emph{local density polynomial} $\Den(M,L,X)\in \mathbb{Q}[X]$  such that for any integer $k\geq 0$,
\begin{equation}\label{eq:Den poly introduction}
\Den(M, L, \q^{-2k})=\Den(H^k\obot M,L).
\end{equation}
Here $H^k$ denotes the orthogonal direct sum of $k$ copies of $H$ and $H^k\obot M$ denotes the orthogonal direct sum of $H^k$ and $M$.       

When $M$    also has rank $n$ and $M\otimes_{O_F}F$ is not isometric to $L\otimes_{O_F}F$, we have $\Den(M,L)=0$. In this case we write $$\Den'(M, L)\coloneqq -2\frac{\rd}{\rd X}\bigg|_{X=1} \Den(M,L, X),$$ and define the (normalized) \emph{derived local density}
\begin{equation}\label{eq: def of Den'}
\partial\mathrm{Den}(L)\coloneqq \frac{\Den'(M_n, L)}{\Den(M_n,M_n)}\in \mathbb{Q}.\end{equation}
Here $M_n=H^{n/2}$ if $n$ is even, and $M_n=H^{(n-1)/2}\obot I_1^1$ if $n$ is odd where $I^1_1$ denotes a rank $1$ hermitian lattice with moment matrix $1$.

Finally we propose and prove the following Kudla--Rapoport conjecture for $\cY$-cycles. For any $O_F$-lattice $L\subset\dV$, we view it as a hermitian $O_F$-lattice under the hermitian norm $h$ unless otherwise specified.

\begin{theorem}[Main Theorem, Theorem \ref{main theorem}]\label{conj: main}
	For any $O_F$-lattice $L\subset \mathbb{V}$, we have 
	$$\mathrm{Int}_{n,\cY}(L)=\partial\Den(L).$$
\end{theorem}

\subsection{Strategy of proof}

We first prove Theorem \ref{conj: main} for $n$ even by reformulating \cite{LL22}*{Theorem 2.11} in terms of $\cY$-cycles.

\begin{theorem}[\cite{LL22}*{Theorem 2.11}]
	Let $n\geq 2$ be even. Normalize the hermitian form on $\dV_n$ by $h'=\pi^{-2}h$. For any $O_F$-lattice $L\subset\dV_n$ viewed as a hermitian $O_F$-lattice under the hermitian form $h'$, we have
	$$\Int_{n,\cZ}(L)=\partial\Den(L).$$
\end{theorem}

Theorem \ref{conj: main} for $n$ even then follows from the fact that there is an isomorphism of hermitian $O_F$-lattices $(L,h)\simeq (\pi L,-h')$, and the following relation between $\cZ$-cycles and $\cY$-cycles for $n$ even.

\begin{proposition}[Proposition \ref{y=pi z even}]
	Let $n\geq 2$ be even. For any nonzero $x\in\dV_n$ we have natural identification $\cY(x)\simeq\cZ(\pi x)$. As a corollary we get $\mathrm{Int}_{n,\cY}(L)=\mathrm{Int}_{n,\cZ}(\pi L)$.
\end{proposition}

Theorem \ref{conj: main} for $n=1$ follows from direct computations as we shall explain in \S\ref{proof of main theorem}.

For the rest of the section we assume $n\geq 3$ is odd. We choose an $O_F$-isogeny $\phi_0$ of degree $q\colon\dX_{n+1}\to\dX_n\times\ubE$ as in \cite{RSZ18}*{\S9} (see \eqref{phi0}). This isogeny induces a natural orthogonal decomposition of hermitian spaces $\dV_{n+1}=\dV_n\obot \langle f\rangle_F$, where $f$ has hermitian norm $-1$. For any $O_F$-lattice $L\subset \dV_n$, view it as a hermitian lattice (of rank $n$) in $\dV_{n+1}$ via the orthogonal decomposition, and set $L^{\#}=L\obot \langle f\rangle_{O_F}$.  We show separately that 

\begin{proposition}[Analytic reduction, Proposition \ref{analytic reduction}]
	$$\partial\Den(L)=\frac{1}{2}\partial\Den(L^{\#}).$$
\end{proposition} 

\begin{proposition}[Geometric reduction, Proposition \ref{geometric reduction}]
	$$\mathrm{Int}_{n,\cY}(L)=\frac{1}{2}\mathrm{Int}_{n+1,\cY}(L^{\#}).$$
\end{proposition}

Theorem \ref{conj: main} for $L$ then follows from the above reduction Propositions and Theorem \ref{conj: main} for $L^{\#}$.

For the analytic reduction, this is done by looking at the possible image of $f$ in $H^s$ and describing the orthogonal complement of the image in details (Proposition \ref{ortho in H^s}).

For the geometric reduction, we first follow the computations in \cite{PR09}*{\S5c} for local models to get a detailed description of the tangent space of $\cN_{m}$ and $\cZ$-cycles on $\cN_m$ at geometric points for $m$ even (Proposition \ref{tangent space of special divisor}). We then make use of the closed embedding $\delta^{\pm}\colon\cN_n\to\cN_{n+1}$ constructed in \cite{RSZ18}*{Proposition 12.1}. Consider $u=\pi f\in\dV_{n+1}$, we prove that

\begin{theorem}[Theorem \ref{identify lower RZ space with special divisor}]
	The closed embedding $\delta^+\colon\cN_n\to\cN_{n+1}^+$ induces an isomorphism (still denoted by $\delta^+$)
	$$\delta^+\colon\cN_n\stackrel{\sim}{\lrarr}\cZ(u)^+\coloneqq\cZ(u)\cap\cN_{n+1}^+.$$
	The same is true for $\delta^-$.
\end{theorem}

Theorems of this kind identifying special $\cZ$-cycles with Rapoport--Zink spaces of lower dimension are proved in \cite{Ter13}*{Lemma 2} ($F/F_0$ unramified, hyperspecial level), \cite{Cho18}*{Proposition 5.10} ($F/F_0$ unramified, maximal parahoric level), \cite{HSY3}*{Proposition 2.6} ($F/F_0$ ramified, Kr\"amer model) and Proposition \ref{N n-2=Z n-1 x} ($F/F_0$ ramified, exotic smooth model, even to odd dimension). In these cases, the inverse morphism can be constructed directly, as the closed embedding of RZ spaces is straightforward. However in our setting, the definition of the morphism $\delta^+$ is much more involved (see \eqref{def delta}), and it is not clear what the inverse morphism looks like.

We adopt a different method and prove this theorem by three steps$\colon$
 \begin{itemize}
 	\item [(1)] Firstly we show that the closed embedding $\delta^+$ factors through $\cZ(u)^+$(Proposition \ref{factor through}).
 	\item [(2)] Secondly we prove that $\delta^+$ induces a bijection on geometric points between $\cN_n$ and $\cZ(u)^+$(Proposition \ref{surj of delta}). This part involves a detailed description of the Dieudonn\'e modules of the geometric points of the RZ spaces.
 	\item [(3)] As we know both $\cN_n$ and $\cZ(u)^+$ are formally smooth of same relative dimension (Corollary \ref{N_n(u) is smooth}) we deduce that we have isomorphism $\cN_n\stackrel{\sim}{\to}\cZ(u)^+$ by formal algebraic geometry argument.
 \end{itemize}

To relate arithmetic intersection numbers on different exotic smooth models, we also study the pullback of special cycles. For any nonzero $x\in\dV_n\subset\dV_{n+1}$, there are natural morphisms $\delta_x^+\colon \cZ_{n+1}(x)\cap\cN_n\to\cZ_n(x)$ and $\gamma_x^+\colon \cY_n(x)\to\cY_{n+1}(x)\cap\cN_n$. By showing the surjectivity on geometric points and formal smoothness of these morphisms, we prove that

\begin{proposition}[Proposition \ref{identify special divisors}]
	The morphism $\delta_x^+$ is an isomorphism if $h(x,x)\in\pi_0 O_{F_0}$.
\end{proposition}

\begin{proposition}[Proposition \ref{identification of Y-cycles}]
    The morphism $\gamma_x^+$ is an isomorphism for any nonzero $x\in\dV_n$.
\end{proposition}

In particular we get $\cY_n(x)\simeq\cY_{n+1}(x)\cap\cZ_{n+1}(u)^+\simeq\cY_{n+1}(x)\cap\cY_{n+1}(f)^+$. The same is true for $\delta^-$. Combine all the results we get $\mathrm{Int}_{n,\cY}(L)=\frac{1}{2}\mathrm{Int}_{n+1,\cY}(L^{\#})$ (Proposition \ref{geometric reduction}).

\begin{remark}
    By the identification Propositions we see that for $n\geq 3$ odd and nonzero $x\in\dV_n$, the cycle $\cY_n(x)$ is always a relative divisor on $\cN_n$, and $\cZ_n(x)$ is a relative divisor if $h(x,x)\in \pi_0 O_{F_0}$. As we shall see in Proposition \ref{N n-2=Z n-1 x}, when $h(x,x)=1$, $\cZ_n(x)$ is also a relative divisor. However when $h(x,x)\in O_{F_0}^{\times}\backslash \Nm O_F^{\times}$, $\cZ_n(x)$ is not well-behaved. This phenomenon causes problem if we want a Kudla--Rapoport formula for $\cZ$-cycles$\colon$we need nontrivial modification to the analytic side. This will be studied in a forthcoming paper.
\end{remark}

\subsection{Notation and terminology}
\begin{itemize}
    \item All rings are commutative and unital, and ring homomorphisms preserve units. For a ring $R$, we denote by $R^{\times}$ all its invertible elements.
    \item Let $R\to R'$ be a ring homomorphism and $M$ be an $R$-module. We put $M_{R'}\coloneqq M\otimes_R R'$.
	\item Let $p$ be an odd prime. Let $F_0$ be a finite extension of $\mathbb{Q}_p$ with residue field $k=\mathbb{F}_q$. Denote by $\bar{k}$ a fixed algebraic closure of $k$. Let $F$ be a ramified quadratic extension of $F_0$. Denote by $a\mapsto \bar{a}$ the (nontrivial) Galois involution of $F/F_0$. Let $\pi$ be a uniformizer of $F$ such that $\bar{\pi}=-\pi$. Let $\pi_0=\pi^2$, a uniformizer of $F_0$. Let $\breve F$ be the completion of the maximal unramified extension of $F$. Let $ O_F, \OFb$ be the ring of integers of $F,\Fb$ respectively.
	\item We say a scheme $S$ is a $\Spf O_{\breve{F}}$-scheme, if $S$ is a scheme over $\Spec O_{\breve{F}}$ and $\pi$ is locally nilpotent in $S$. For a $\Spf O_{\breve{F}}$-scheme $S$, we denote its special fibre by
	$$\ovS=S\times_{\Spec O_{\breve{F}}}\Spec\bar{k}.$$
	\item Throughout the paper, by an $O_F$-lattice $L$ inside an $F$-vector space $V$, without mentioning the rank of $L$ we always mean a lattice of full rank. A hermitian $O_F$-lattice $L$ is a pair $(L,h)$ where $L$ is an $O_F$-lattice equipped with an $O_{F_0}$-bilinear pairing $h(-,-)\colon L\times L\to F$ such that the induced $F$-valued pairing on $L\otimes_{ O_F}F$ is a nondegenerate hermitian pairing with respect to $F/F_0$. We shall simply write a hermitian $O_F$-lattice as $L$ without mentioning the pairing $h$ when it is clear from the context. We say a hermitian $O_F$-lattice $L$ is split (nonsplit) if the corresponding $F/F_0$-hermitian space $L\otimes_{O_F}F$ is a split (nonsplit) hermitian space.
    \item Let $R$ be a local ring. If $L$ and $M$ are two $R$-lattices of the same rank and $L\subset M$, we use the notation $L\subset^r M$ to indicate the co-length of $L$ in $M$ is $r$, and $L\subset^{\leq r}M$ to indicate that the co-length of $L$ in $M$ is smaller or equal to $r$.
    \item Let $\cO$ be a complete discrete valuation ring over $\dZ_p$. Let $S$ be a scheme over $\Spec \cO$. A pair $(X,\iota)$ consisting of a $p$-divisible group $X$ over $S$ and an action $\iota$ of $\cO$ on $X$ is called a strict $\cO$-module if the action of $\cO$ on $\Lie X$ is via the structure morphism $\cO\to\cO_S$. A strict $\cO$-module is called formal if the $p$-divisible group is formal. For simplicity we will omit 'strict' and just call them formal $\cO$-modules.
\end{itemize}

\subsection{Acknowledgement}
First, I would like to express my deep gratitude to Professor Chao Li, who suggested I consider the question, and had countless instructive conversations with me during this project, without which this article would never exist. Next, I would like to thank Qiao He and Baiqing Zhu for many helpful discussions. The author is supported by the Department of Mathematics at Columbia University in the city of New York.

\section{Rapoport--Zink spaces and special cycles}

\subsection{Rapoport--Zink space}
In this section, we review the definition and basic properties of Rapoport--Zink spaces and special cycles.

\begin{definition}\label{herm formal module}
	For any $\Spf O_{\breve{F}}$-scheme $S$, a hermitian formal $O_F$-module $(X,\iota,\lambda)$ of signature $(1,n-1)$ over $S$ is the following data$\colon$
	\begin{itemize}
		\item $X$ is a (strict) formal $O_{F_0}$-module over $S$ of dimension $n$ and relative height $2n$;
		\item $\iota$ is an action of $O_F$ on $X$ extending the $O_{F_0}$-action;
		\item $\lambda$ is an $\iota$-compatible polarization of $X$ such that $\operatorname{ker}\lambda\subset X[\iota(\pi)]$ is of rank $q^{2[\frac{n}{2}]}$. Here $\iota$-compatible means $\lambda\circ\iota(a)=\iota(\bar{a})^\vee\circ\lambda$ holds for every $a\in O_F$,
	\end{itemize}
	such that when $n$ is even, the triple $(X,\iota,\lambda)$ satisfies the following conditions$\colon$
	    \begin{itemize}
		\item [--](Kottwitz condition) 
		\begin{equation}\label{Kott cond}
		\begin{aligned}
		\text{the characteristic polynomial of }\iota(\pi)&\text{ on the }\cO_S\text{-module }\Lie(X)\text{ is }\\(T-\pi)(T+\pi)^{n-1}&\in\cO_S[T],
		\end{aligned}
		\end{equation}

		\item [--](Wedge condition)
		\begin{equation}\label{wedge cond}
		\bigwedge^2 \( \iota(\pi)+\pi\res\Lie(X) \)=0,
		\end{equation}
		
		\item [--](Spin condition) 
		\begin{equation}\label{spin cond}
		\text{for every geometric point }s\text{ of }S\text{, the action of }\iota(\pi)\text{ on }\Lie(X_s)\text{ is nonzero;}
		\end{equation}
	\end{itemize}
	and when $n$ is odd, we require the triple to satisfy condition \ref{odd spin cond} below$\colon$this condition is a little complicated to formulate and require some preparation; we postpone the formulation of the condition to $\S$\ref{condition for n odd}.
	
	An isomorphism of hermitian formal $O_F$-modules $(X,\iota,\lambda)\stackrel{\sim}{\lrarr}(X',\iota',\lambda')$ is an $O_F$-linear isomorphism $\varphi\colon X\stackrel{\sim}{\lrarr}X'$ such that $\varphi^*(\lambda')=\lambda$. To define the moduli problem, we fix an explicit choice of such triple $(\dX_n,\iota_{\mathbb{X}_n},\lambda_{\mathbb{X}_n})$ over $\bar{k}$ as the framing object in $\S$\ref{framing objects}.
\end{definition}

\begin{definition}
	Let $\operatorname{Nilp}_{O_{\breve{F}}}$ be the category of $O_{\breve{F}}$-schemes $S$ such that $\pi$ is locally nilpotent on $S$. Then the Rapoport--Zink space associated with $(\dX_n,\iota_{\mathbb{X}_n},\lambda_{\mathbb{X}_n})$ is the functor
	$$\cN_n\to\Spf O_{\breve{F}}$$
	sending $S\in\operatorname{Nilp}_{O_{\breve{F}}}$ to the set of isomorphism classes of tuples $(X,\iota,\lambda,\rho)$, where
	\begin{itemize}
		\item $(X,\iota,\lambda)$ is a hermitian formal $O_F$-module of signature $(1,n-1)$ over $S$;
		\item $\rho\colon X\times_S\ovS\to\dX_n\times_{\bar{k}}\ovS$ is an $O_F$-linear quasi-isogeny of height $0$ over $\ovS$ such that $\rho^*(\lambda_{\mathbb{X}_n,\ovS})=\lambda_{\ovS}$.
	\end{itemize}
\end{definition}

\begin{proposition}[\cite{RSZ18}*{Theorem 6.5, Theorem 7.3}]
	The functor $\mathcal{N}_n$ is (pro-)representable by a separated formal scheme, which is formally locally of finite type, essentially proper, and formally smooth of relative formal dimension $n-1$ over $\Spf O_{\breve{F}}$. In particular $\mathcal{N}_n$ is regular of formal dimension $n$.
\end{proposition}

When $n$ is even, there is a natural decomposition of $\cN_n$. We will need this in \S\ref{geom red section}. Let
\[
\dM=\dM_n
\quad\text{and}\quad
\dN=\dN_n \coloneqq \dM \otimes_{O_{\breve F_0}} \breve F_0
\]
denote the covariant relative Dieudonn\'e module and rational Dieudonn\'e module, respectively, of $\dX_n$.  The action $\iota_{\dX_n}$ makes $\dM$ into an $O_F \otimes_{O_{F_0}} O_{\breve F_0} = O_{\breve F}$-module. The polarization $\lambda_{\dX_n}$ induces a nondegenerate alternating $\breve F_0$-bilinear form \aform on $\dN$ satisfying
\[
\langle ax,y \rangle = \langle x, \bar{a} y \rangle
\quad\text{for all}\quad
x,y \in \dN,\ a \in \breve F.
\]
The form
\begin{equation*}
h(x,y) \coloneqq \langle \pi x, y \rangle + \pi\la x,y \rangle,
\quad
x,y \in \dN,
\end{equation*}
then makes $\dN$ into an $\breve F/\breve F_0$-hermitian space of dimension $n$.  By Dieudonn\'e theory, for a perfect field extension $K$ of $\bar{k}$, the set of $K$-points on $\cN_n$ identifies with a certain subset $\cS$ of $\pi$-modular (Definition \ref{defn for lattices}) $O_{\breve F} \otimes_{O_{\breve F_0}} W(K)$-lattices $M \subset \dN \otimes_{O_{\breve F_0}} W(K)$ (details in \S\ref{details of D modules}).
For a lattice $M \in \cS$, we say that the corresponding $K$-point on $\cN_n$ lies in $\cN_n^+$ or $\cN_n^-$ according as the $O_{\breve F} \otimes_{O_{\breve F_0}} W(K)$-length of the module
\begin{equation}\label{parity test lattice}
\bigl(M + \dM \otimes_{O_{\breve F_0}} W(K)\bigr) \big/ \dM \otimes_{O_{\breve F_0}} W(K)
\end{equation}
is even or odd.  The parity of this length may also be described as follows.  Since the $\pi$-modular lattices in a hermitian space are all conjugate under the unitary group, and since $\dM \otimes_{O_{\breve F_0}} W(K)$ is itself $\pi$-modular in $\dN \otimes_{O_{\breve F_0}} W(K)$, there exists some $g \in \operatorname{U}(\dN)(\breve F \otimes_{O_{\breve F_0}} W(K))$ such that $g\cdot (\dM \otimes_{O_{\breve F_0}} W(K)) = M$.  The determinant $\det g$ is a norm one element in $\breve F \otimes_{O_{\breve F_0}} W(K)$, and hence lies in $O_{\breve F} \otimes_{O_{\breve F_0}} W(K)$ with reduction mod $\pi$ equal to $\pm 1 \in K$.  Then the length of \eqref{parity test lattice} is even or odd according as this reduction is $1$ or $-1$ (independent of the choice of $g$) by \cite{RSZ17}*{Lemma 3.2}.

\begin{proposition}[\cite{RSZ18}*{Proposition 6.4}]\label{even n decomp lem}
	$\cN_n^+$ and $\cN_n^-$ define a decomposition
	\[
	\cN_n = \cN_n^+ \amalg \cN_n^-
	\]
	into open and closed formal subschemes.
\end{proposition}

\subsection{Condition for $n$ odd}\label{condition for n odd}

In this section we specify the condition for the triple $(X,\iota,\lambda)$ when $n$ is odd, and describe its relation with the Kottwitz condition, Wedge condition and Spin condition. We follow everything from \cite{RSZ18}*{\S7} and nothing new is proved here.

Since we will also need an analog of this condition in $\S$\ref{other condition for even} for even $n$, for the moment let $n$ be any positive integer. Let
\[
m \coloneqq  \lfloor n/2 \rfloor.
\]
Let $e_1,\dotsc,e_n$ denote the standard basis in $F^n$, and let $h$ be the standard split $F/F_0$-hermitian form on $F^n$ with respect to this basis,
\begin{equation}\label{herm form}
h(ae_i,be_j) \coloneqq  a \bar{b} \delta_{i,n+1-j} \quad\text{(Kronecker delta)}.
\end{equation}
Let \aform and \sform be the respective alternating and symmetric $O_{F_0}$-bilinear forms $F^n \times F^n \to F_0$ defined by
\begin{equation}\label{aform and sform}
\langle x,y \rangle \coloneqq  \frac 1 2 \tr_{F/F_0} \bigl(\pi^{-1}h(x,y)\bigr)
\quad\text{and}\quad
(x,y) \coloneqq  \frac 1 2 \tr_{F/F_0}h(x,y).
\end{equation}
For $i = bn + c$ with $0 \leq c < n$, define the $O_F$-lattice
\[
\Lambda_i \coloneqq  \sum_{j = 1}^c \pi^{-b-1}O_F e_j + \sum_{j = c+1}^n \pi^{-b} O_F e_j \subset F^n.
\]
For each $i$, the form \aform induces a perfect pairing
\[
\Lambda_i \times \Lambda_{-i} \to O_{F_0}.
\]
In this way, for fixed nonempty $I \subset \{0,\dotsc,m\}$, the set
\[
\Lambda_I \coloneqq  \{\, \Lambda_i \mid i \in \pm I + n\dZ \, \}
\]
forms a polarized chain of $O_F$-lattices over $O_{F_0}$ in the sense of \cite{RZ96}*{Definition 3.14}.

Now define the $2n$-dimensional $F$-vector space
\begin{equation*}
V \coloneqq  F^n \otimes_{F_0} F,
\end{equation*}
where $F$ acts on the right tensor factor.
The $n$-th wedge power $\tensor*[^n] V {} \coloneqq  \bigwedge_F^n V$ admits a canonical decomposition
\begin{equation}\label{nV decomp}
\tensor*[^n] V {} = 
\bigoplus_{\substack{r+s = n\\\epsilon \in \{\pm1\}}} \tensor*[^n]{V}{_\epsilon^{r,s}}
\end{equation}
which is described in \cite{Smi15}*{\S2.3, 2.5}.%
\footnote{Here and below we replace the symbol $W$ used in loc.\ cit.\ with $\tensor*[^n] V{}$.}
Let us briefly review it.  The operator $\pi \otimes 1$ acts $F$-linearly on $V$ with eigenvalues $\pm \pi$; let
\[
V = V_\pi \oplus V_{-\pi}
\]
denote the corresponding eigenspace decomposition. For a partition $r + s = n$, define%
\footnote{Here and below we interchange $r$ and $s$ in the notation relative to \cite{Smi15}.}
\[
\tensor*[^n]V{^{r,s}} \coloneqq  \sideset{}{_F^r}\bigwedge V_\pi \otimes_F \sideset{}{_F^s}\bigwedge V_{-\pi},
\]
which is naturally a subspace of $\tensor*[^n] V{}$.  Furthermore, the symmetric form \sform splits after base change to $V$, and therefore there is a decomposition
\begin{equation}\label{^nV decomp}
\tensor*[^n] V{} = \tensor*[^n] V{_1} \oplus \tensor*[^n] V{_{-1}}
\end{equation}
as an $\SO(\sform)(F)$-representation.  The subspaces $\tensor*[^n]V{_{\pm 1}}$ have the property that for any Lagrangian (i.e.\ totally isotropic $n$-dimensional) subspace $\cF \subset V$, the line $\bigwedge_F^n \cF \subset \tensor*[^n] V{}$ is contained in one of them, and in this way they distinguish the two connected components of the orthogonal Grassmannian $\OGr(n,V)$ over $\Spec F$.  The subspaces $\tensor*[^n]V{_{\pm 1}}$ are canonical up to labeling, and we will follow the labeling conventions in loc.\ cit., to which we refer the reader for details.  The summands in the decomposition \eqref{nV decomp} are then given by
\[
\tensor*[^n] V{_\epsilon^{r,s}} \coloneqq  \tensor*[^n] V{^{r,s}} \cap \tensor*[^n] V {_\epsilon}
\]
(intersection in $\tensor*[^n] V{}$) for $\epsilon \in \{\pm 1\}$.

Given an $O_F$-lattice $\Lambda \subset F^n$, now define
\[
\tensor*[^n]\Lambda{} \coloneqq  \sideset{}{_{O_F}^n} \bigwedge (\Lambda \otimes_{O_{F_0}} O_F),
\]
which is naturally a lattice in $\tensor*[^n]V{}$.  For fixed $r$, $s$, and $\epsilon$, define
\begin{equation}\label{^nLambda_epsilon^r,s}
\tensor*[^n]\Lambda{_{\epsilon}^{r,s}} \coloneqq  \tensor*[^n]\Lambda{}\cap \tensor*[^n] V{_\epsilon^{r,s}}
\end{equation}
(intersection in $\tensor*[^n] V{}$).  Then $\tensor*[^n]\Lambda{_{\epsilon}^{r,s}}$ is a direct summand of $\tensor*[^n]\Lambda{}$, since the quotient $\tensor*[^n]\Lambda{} / \tensor*[^n]\Lambda{_{\epsilon}^{r,s}}$ is torsion-free.  For an $O_F$-scheme $S$, define
\begin{equation}\label{L def}
L_{i,\epsilon}^{r,s}(S) \coloneqq  
\operatorname{im} \bigl[ \tensor*[^n]{(\Lambda_i)}{_{\epsilon}^{r,s}} \otimes_{O_F} \cO_S \rightarrow \tensor*[^n]\Lambda{_i} \otimes_{O_F} \cO_S \bigr].
\end{equation}
For nonempty $I \subset \{0,\dotsc,m\}$, let $\underline{\operatorname{Aut}}(\Lambda_I)$ denote the scheme of automorphisms of the polarized $O_F$-lattice chain $\Lambda_I$ over $\Spec O_{F_0}$, in the sense of \cite{RZ96}*{Theorem 3.16} or \cite{Pap00}*{Page 581} (this is denoted by $\cP$ in \cite{Pap00}).

\begin{lemma}[\cite{RSZ18}*{Lemma 7.1}]\label{L stability}
	For any $O_F$-scheme $S$ and $\Lambda_i \in \Lambda_I$, the submodule $L_{i,\epsilon}^{r,s}(S) \subset \tensor*[^n]\Lambda{_i} \otimes_{O_F} \cO_S$ is stable under the natural action of $\underline{\operatorname{Aut}}(\Lambda_I)(S)$ on $\tensor*[^n]\Lambda{_i} \otimes_{O_F} \cO_S$.
\end{lemma}

This concludes our discussion for general $n$.  We now formulate our condition on the triple $(X,\iota,\lambda)$ over a $\Spf O_{\breve F}$-scheme $S$ in the case of odd $n$, which will make use of the above discussion in the case $I = \{m\}$. Let $M(X)$ and $M(X^\vee)$ denote the respective Lie algebras of the universal vector extensions of $X$ and $X^\vee$.  Since $\ker\lambda \subset X[\iota(\pi)]$, there is a unique (necessarily $O_F$-linear) isogeny $\lambda'$ such that the composite
\[
X \xra{\lambda} X^\vee \xra{\lambda'} X
\]
is $\iota(\pi)$.  Since $\ker\lambda$ furthermore has rank $q^{n-1}$, the induced diagram
\[
M(X) \xra{\lambda_*} M(X^\vee) \xra{\lambda'_*} M(X)
\]
extends periodically to a polarized chain of $O_F \otimes_{O_{F_0}} \cO_S$-modules of type $\Lambda_{\{m\}}$, in the terminology of \cite{RZ96}.  By Theorem 3.16 in loc.\ cit., \'etale-locally on $S$ there exists an isomorphism of polarized chains
\begin{equation}\label{chain triv}
[{}\dotsb \xra{\lambda'_*} M(X) \xra{\lambda_*} M(X^\vee) \xra{\lambda'_*} \dotsb{}] \stackrel{\sim}{\lrarr} \Lambda_{\{m\}} \otimes_{O_{F_0}} \cO_S,
\end{equation}
which in particular gives an isomorphism of $O_F \otimes_{O_{F_0}} \cO_S$-modules
\begin{equation}\label{M(X) triv}
M(X) \stackrel{\sim}{\lrarr} \Lambda_{-m} \otimes_{O_{F_0}} \cO_S.
\end{equation}
The module $M(X)$ fits into the covariant Hodge filtration
\[
0 \to \Fil^1 \to M(X) \to \Lie X \to 0
\]
for $X$, and the condition we finally impose is that
\begin{align*}
&\text{\em upon identifying $\Fil^1$ with a submodule of $\Lambda_{-m} \otimes_{O_{F_0}} \cO_S$ via \eqref{M(X) triv}, the line bundle}\\
\shortintertext{
	\begin{equation}\label{odd spin cond}
	\sideset{}{_{\cO_S}^n}\bigwedge \Fil^1 \subset  \tensor[^n]\Lambda{_{-m}} \otimes_{O_F} \cO_S
	\end{equation}
}
&\text{\em is contained in $L_{-m,-1}^{n-1,1}(S)$.}
\end{align*}
Note that Lemma \ref{L stability} gives exactly what is needed to conclude that condition \eqref{odd spin cond} is independent of the choice of chain isomorphism in \eqref{chain triv}.

\begin{remark}\label{odd N_n crit}
	The Kottwitz condition \eqref{Kott cond}, Wedge condition \eqref{wedge cond}, and Spin condition \eqref{spin cond} all continue to make sense as written in the odd ramified setting.  The first two of these conditions are implied by condition \eqref{odd spin cond}, cf.~\cite{Smi15}*{Lemma~5.1.2, Remark~5.2.2} (which shows that these implications hold on the local model).  On the other hand, let $\cN_n^\circ$ be the moduli space of quadruples $(X,\iota,\lambda,\rho)$ as in the definition of $\cN_n$, except that instead of imposing condition \eqref{odd spin cond}, we impose conditions \eqref{wedge cond} and \eqref{spin cond}.  Then \eqref{wedge cond} and \eqref{spin cond} imply condition \eqref{odd spin cond}, and in this way $\cN_n^\circ$ is an open formal subscheme of $\cN_n$.  Indeed, this statement follows from the analogous statement for the corresponding local models, which is explained in \cite{Smi15}*{\S3.3}.   (More precisely, loc.~cit.~shows that the local model for $\cN_n^\circ$ is the complement of the ``worst point'' in the local model for $\cN_n$.)
\end{remark}

\begin{remark}
	There is a natural analog of condition \eqref{odd spin cond} on $\cN_n$ when $n$ is even (still with $F/F_0$ ramified).  However this analog is automatically satisfied on the whole space, which follows from the fact that this condition is automatically satisfied in the generic fiber of the local model for $\cN_n$, and this local model is already flat (in fact smooth).
\end{remark}

\subsection{Framing objects}\label{framing objects}
To complete the definition of $\cN_n$, it remains to specify a framing object $(\dX_n,\iota_{\mathbb{X}_n},\lambda_{\dX_n})$ for this moduli problem over $\bar{k}$.

We denote by $\mathbb{E}$ the unique formal $O_{F_0}$-module of dimension 1 and relative height 2 over $\bar{k}$. The Dieudonn\'e module $\mathbb{M}_0$ of $\mathbb{E}$ can be identified with $W_{O_{F_0}}(\bar{k})^2=O_{\breve{F}_0}^2$ endowed with the Frobenius and Verschiebung operator given in matrix form by
\begin{equation}\label{Frob on E}
F=\begin{pmatrix}
0&\pi_0\\1&0
\end{pmatrix}\sigma\text{ , }
V=\begin{pmatrix}
    0&\pi_0\\1&0
\end{pmatrix}\sigma^{-1}
\end{equation}
where $\sigma$ denotes the usual Frobenius homomorphism on the Witt vectors. To give a polarization on $\mathbb{E}$ is to give an alternating bilinear pairing on $\mathbb{M}_0$ with associated matrix of the form
$$\begin{pmatrix}
0&\delta\\-\delta&0
\end{pmatrix}$$
for $\delta\in O_{\breve{F}_0}$ satisfying $\sigma(\delta)=-\delta$. We define the principal polarization $\lambda_{\mathbb{E}}$ by fixing any such $\delta\in O_{\breve{F}_0}^\times$ once and for all. Note that any other principal polarization of $\mathbb{E}$ differs from $\lambda_{\mathbb{E}}$ by an $O_{{F}_0}^\times$-multiple.

Let $\iota_{\mathbb{E}}$ be an embedding
$$\iota_{\mathbb{E}}\colon O_F\hookrightarrow \operatorname{End}_{O_{F_0}}(\mathbb{E})$$
which makes $(\mathbb{E},\iota_{\dE},\lambda_{\dE})$ into a hermitian formal $O_F$-module of signature $(1,0)$ over $\bar{k}$. Explicitly we take
$$\iota_{\mathbb{E}}(a+b\pi)\coloneqq \begin{pmatrix}
a & b\pi_0 \\
b & a
\end{pmatrix},\quad a,b\in O_{F_0}$$
acting on the Dieudonn\'e module.

We denote by $(\ubE,\iota_{\ubE},\lambda_{\ubE})$ the same object as $(\mathbb{E},\iota_{\dE},\lambda_{\dE})$, except where the $O_F$-action $\iota_{\ubE}$ is given by the composition of $\iota_{\mathbb{E}}$ and the Galois involution on $F$.

\begin{itemize}

\item [$\bullet$] When $n$ is even, up to $O_F$-linear quasi-isogeny compatible with polarizations, there is a unique hermitian formal $O_F$-module $(\dX_n,\iota_{\mathbb{X}_n},\lambda_{\dX_n})$ over $\bar{k}$, which follows from Proposition 3.1 and its proof in \cite{RSZ17}.

When $n = 2$, we take
\[
\dX_2 \coloneqq\dE \times \dE
\]
as a formal $O_{F_0}$-module, and we define $\iota_{\dX_2}$ by
\[
\iota_{\dX_2}(a + b\pi)\coloneqq
\begin{pmatrix}
a  &  b\pi_0\\
b  &  a
\end{pmatrix},
\quad
a,b \in O_{F_0}.
\]
(This identifies $\dX_2$ with the Serre tensor construction $O_F \otimes_{O_{F_0}} \dE$.)  For the polarization, we take
\begin{equation}\label{lambda_BX_2}
\lambda_{\dX_2} \coloneqq 
\begin{pmatrix}
-2\lambda_\dE\\
&  2\pi_0 \lambda_\dE
\end{pmatrix}.
\end{equation}
We then take
\begin{equation}\label{frame}
\begin{split}
\dX_n &\coloneqq  \dX_2 \times \ubE^{n-2},\\
\iota_{\dX_n} &\coloneqq  \iota_{\dX_2} \times \iota_{\ubE}^{n-2},\\
\lambda_{\dX_n} &\coloneqq  \lambda_{\dX_2} \times 
\diag \biggl(
\underbrace{\begin{pmatrix} 
	0  &  \lambda_{\ubE}\, \iota_{\ubE}(\pi)\\
	-\lambda_{\ubE}\, \iota_{\ubE}(\pi)  &  0
	\end{pmatrix}
	,\dotsc,
	\begin{pmatrix} 
	0  &  \lambda_{\ubE}\, \iota_{\ubE}(\pi)\\
	-\lambda_{\ubE}\, \iota_{\ubE}(\pi)  &  0
	\end{pmatrix}}_{(n-2)/2 \text{ times}}\biggr).
\end{split}
\end{equation}

It is straightforward to see that the framing object satisfies the Kottwitz condition, Wedge condition and Spin condition. 

\item [$\bullet$] When $n$ is odd, in contrast to the previous case, a triple $(\dX_n,\iota_{\dX_n},\lambda_{\dX_n})$ of hermitian formal $O_F$-module over $\bar{k}$ is \textit{not unique up to quasi-isogeny}. In fact, there are two isogeny classes (as always, up to $O_F$-linear quasi-isogeny compatible with the polarizations), corresponding to the two possible isometry classes of the hermitian space $C$ in the proof of \cite{RSZ17}*{Proposition~3.1}.\footnote{The hermitian space $C$ is isomorphic to the hermitian space of quasi-homomorphisms $\dV_n$ defined in \eqref{space of special homs}.}

As an explicit representative for which $C$ is nonsplit, we take the same framing object as in \cite{RSZ17}.  Thus when $n = 1$ we define
\begin{equation}\label{BX_1}
\dX_1^{(1)} \coloneqq  \dE,
\quad
\iota_{\dX_1^{(1)}} \coloneqq  \iota_\dE,
\quad\text{and}\quad
\lambda_{\dX_1^{(1)}} \coloneqq  -\lambda_\dE.
\end{equation}

When $n \geq 3$, we take 
\begin{equation}\label{BX_n odd}
\dX_n^{(1)} \coloneqq  \dX_{n-1} \times \ubE,
\quad
\iota_{\dX_n^{(1)}} \coloneqq  \iota_{\dX_{n-1}} \times \iota_{\ubE},
\quad\text{and}\quad
\lambda_{\dX_n^{(1)}} \coloneqq  \lambda_{\dX_{n-1}} \times \lambda_{\ubE};
\end{equation}
here $n-1$ is even and $(\dX_{n-1},\iota_{\dX_{n-1}},\lambda_{\dX_{n-1}})$ is as defined in \eqref{frame}. 

To fix a framing object in the other isogeny class, we fix $\epsilon\in O_{F_0}^\times\smallsetminus \operatorname{Nm} F^\times$ and define
\begin{equation}\label{BX_n^{(0)}}
\dX_n^{(0)} \coloneqq  \dX_n^{(1)},
\quad
\iota_{\dX_n^{(0)}} \coloneqq  \iota_{\dX_n^{(1)}},
\quad\text{and}\quad
\lambda_{\dX_n^{(0)}} \coloneqq  \epsilon\lambda_{\dX_n^{(1)}}.
\end{equation}
Note that such an $\epsilon$ exists since $F/F_0$ is ramified.

Taking $\dX_n^{(0)}$ and $\dX_n^{(1)}$ as the framing objects, we obtain respective moduli spaces  $\cN_n^{(0)}$ and $\cN_n^{(1)}$. However, these spaces are isomorphic via the map


\[
\xymatrix@R=0ex{
	\cN_n^{(1)}\quad \ar[r]^-\sim  & \quad \cN_n^{(0)}\\
	(X, \iota, \lambda, \rho) \ar@{|->}[r]  &  \bigl(X, \iota, \lambda\circ\iota(\epsilon), \rho\bigr).
}
\]
To simplify notation in the rest of the paper, from now on we set
\[
\cN_n \coloneqq  \cN_n^{(1)}
\quad\text{and}\quad
\dX_n \coloneqq  \dX_n^{(1)}.
\]

We note that when $n \geq 3$, the framing objects $\dX_n$ obviously satisfies \eqref{wedge cond} and \eqref{spin cond} (because $\dX_{n-1}$ does), and therefore they indeed satisfy \eqref{odd spin cond}; and it is trivial to check directly that $\dX_1$ satisfies \eqref{odd spin cond} when $n=1$.

\end{itemize}

\begin{example}[$n=1$]\label{CN_1}
	When $n = 1$, the moduli problem for $\cN_1$ is just the moduli problem of lifting $\dE$ as a formal $O_F$-module.  By the theory of canonical lifting (\cite{Gr86}), we have $\cN_1 = \Spf O_{\breve F}$, with universal object the canonical lift $(\cE, \iota_\cE, -\lambda_\cE, \rho_\cE)$.  (In this case condition \eqref{odd spin cond} is redundant in the moduli problem.)
\end{example}

\subsection{Special cycles}

We define the space of \textit{special quasi-homomorphisms}
\begin{equation}\label{space of special homs}
\mathbb{V}_n=\dV(\dX_n)\coloneqq \operatorname{Hom}_{O_F}(\ubE,\mathbb{X}_n)\otimes_{O_{F}}F
\end{equation}
and for $x,y\in\mathbb{V}_n$, define $h(x,y)$ to be the composition
$$\ubE\stackrel{x}{\lrarr}\mathbb{X}_n\stackrel{\lambda_{\mathbb{X}_n}}{\lrarr}\mathbb{X}_n^\vee\stackrel{y^\vee}{\lrarr}\ubE^\vee\stackrel{\lambda_{\ubE}^{-1}}{\lrarr}\ubE$$
which lies in $$\operatorname{End}_{O_F}(\ubE)\otimes_{O_F}F\stackrel{\iota_{\ubE}^{-1}}{\simeq}F.$$

It turns out that $h$ defines a hermitian form on $\mathbb{V}_n$ which makes $(\mathbb{V}_n,h)$ a nondegenerate and nonsplit hermitian space. (\cite{RSZ17}*{Lemma 3.5}, note that for $n$ odd we take $\dX_n=\dX_n^{(1)}$.)

We denote by $\mathcal{E}$ the canonical lift of $\mathbb{E}$ over $\operatorname{Spf}O_{\breve{F}}$ with respect to $\iota_{\mathbb{E}}$, equipped with its $O_F$-action $\iota_{\mathcal{E}}$, $O_F$-linear framing isomorphism $\rho_{\mathcal{E}}\colon \mathcal{E}_{\bar{k}}\stackrel{\sim}{\to}\mathbb{E}$ and principal polarization $\lambda_{\mathcal{E}}$ lifting $\rho_{\mathcal{E}}^*(\lambda_{\mathbb{E}})$. We denote by $(\ucE,\iota_{\ucE},\lambda_{\ucE},\rho_{\ucE})$ the same object but with $O_F$-action twisted by the Galois involution on $F$.

\begin{definition}\label{z and y cycles}
	For every nonzero $x\in\mathbb{V}_n$, we define the special $\cZ$-cycle $\cZ(x)=\cZ_n(x)$ of $\mathcal{N}_n$ to be the maximal closed formal subscheme over which the quasi-homomorphism
	$$\rho^{-1}\circ x\colon \ubE\times_{\bar{k}} \ovS\to X\times_S \ovS$$
	extends to a homomorphism $\ucE_S\to
	X$.
	
	We also define the special $\cY$-cycle $\cY(x)=\cY_n(x)$ of $\cN_n$ to be the maximal closed formal subscheme over which the quasi-homomorphism
	$$\lambda\circ\rho^{-1}\circ x\colon \ubE\times_{\bar{k}} \ovS\to X\times_S \ovS\to X^{\vee}\times_S\ovS$$
	extends to a homomorphism $\ucE_S\to
	X^{\vee}$.
\end{definition}

By Grothendieck--Messing theory, $\cZ(x)$ and $\cY(x)$ are closed formal subschemes. It is proved that when $n$ is even, $\cZ(x)$ is a relative divisor (\cite{LL22}*{Lemma 2.40}).

\begin{definition}
	For any $O_F$-lattice $L\subset\mathbb{V}_n$, we define the arithmetic intersection numbers for $\cZ$-cycles and $\cY$-cycles. 
	
	For $n$ even, the Serre intersection multiplicity
	$$\chi\left(\cO_{\cZ_n(x_1)}\otimes^{\mathbb{L}}_{\cO_{\mathcal{N}_n}}\cdots\otimes^{\mathbb{L}}_{\cO_{\mathcal{N}_n}}\cO_{\cZ_n(x_n)}\right)$$
	does not depend on the choice of a basis $\{x_1,\dots,x_n\}$ of $L$ by \cite{LL22}*{Corollary 2.35}, which we define to be $\operatorname{Int}_{n,\cZ}(L)$.
	
	For general $n$, the Serre intersection multiplicity
	$$\chi\left(\cO_{\cY_n(x_1)}\otimes^{\mathbb{L}}_{\cO_{\mathcal{N}_n}}\cdots\otimes^{\mathbb{L}}_{\cO_{\mathcal{N}_n}}\cO_{\cY_n(x_n)}\right)$$
	does not depend on the choice of a basis $\{x_1,\dots,x_n\}$ of $L$ by Proposition \ref{y=pi z even} and Proposition \ref{geometric reduction}, which we define to be $\operatorname{Int}_{n,\cY}(L)$.
\end{definition}

For $n\geq 1$, let $g\mapsto g^{\dagger}$ denote the Rosati involution on $\operatorname{End}_{O_F}^{\circ}(\mathbb{X}_n)$ induced by $\lambda_{\mathbb{X}_n}$. Define
$$U(\mathbb{X}_n)\coloneqq \left\{ g\in\operatorname{End}_{O_F}^{\circ}(\mathbb{X}_n) \res gg^{\dagger}=\operatorname{id}_{\mathbb{X}_n}\right\}.$$
Thus $U(\mathbb{X}_n)$ is the group of $O_F$-linear self-quasi-isogenies of $\mathbb{X}_n$ which preserve $\lambda_{\mathbb{X}_n}$ on the nose.

The group $U(\mathbb{X}_n)$ acts naturally from the left on $\mathbb{V}_n$, and in this way identifies with the unitary group $U(\mathbb{V}_n,h)$. For any $\beta\in O_{F_0}^{\times}$, the group $U(\mathbb{X}_n)\simeq U(\mathbb{V}_n)$ acts transitively on the subset of elements in $\mathbb{V}_n$ whose hermitian norm is $\beta$.

On the other hand, each $g$ in $U(\mathbb{X}_n)$ is a quasi-isogeny of height $0$, and therefore $U(\mathbb{X}_n)$ acts naturally on $\mathcal{N}_n$ on the left via the rule $g\cdot (X,\iota,\lambda,\rho)=(X,\iota,\lambda,g\rho)$. It is easy to see that under this action, for any $x\in\mathbb{V}_n$ and $g\in U(\mathbb{X}_n)\simeq U(\mathbb{V}_n)$ we have \begin{equation}\label{action of unitary group}
    g\cdot\cZ_n(x)=\cZ_n(gx)\quad\text{ and }\quad g\cdot\cY_n(x)=\cY_n(gx).
\end{equation}

\section{Local density and Kudla--Rapoport conjecture}

\subsection{Local density}
	In this section we study local representation densities of hermitian lattices. We state all results for $F/F_0$-hermitian forms, although for $\breve{F}/\Breve{F}_0$-hermitian forms everything works as well. We first introduce some notions about hermitian $O_F$-lattices.

\begin{definition}\label{defn for lattices}
	Let $(V,h)$ be a hermitian space over $F$ of dimension $m$.
	\begin{itemize}
		\item [(1)] For an $ O_F$-lattice $L$ of $V$, define
		$$\begin{aligned}
		L^\vee&\coloneqq \{x\in V\res h(x,y)\in \pi^{-1} O_F\text{ for every }y\in L\},\\
		L^*&\coloneqq \{x\in V \res h(x,y)\in O_F\text{ for every }y\in L\}.
		\end{aligned}$$
		We have $L^*=\pi L^{\vee}$. We say $L$ is integral if $L\subset L^\vee$ and self-dual if $L=L^\vee$. When $m$ is even, we say $L$ is $\pi$-modular if $L^{*}=\pi^{-1}L$, and when $m$ is odd, we say $L$ is almost $\pi$-modular if $L\subset L^{*}\subset^1 \pi^{-1}L$.
		\item [(2)] For an integral $ O_F$-lattice $L$ of $V$, we define
		\begin{itemize}
			\item the fundamental invariants of $L$ to be the unique integers $0\leq a_1\leq\cdots\leq a_m$ such that $L^\vee/L\simeq  O_F/(\pi^{a_1})\oplus\cdots\oplus  O_F/(\pi^{a_m})$ as $ O_F$-modules;
			\item the valuation of $L$ to be $\operatorname{val}(L)=\sum_{i=1}^m a_i$.
		\end{itemize}
	\end{itemize}
\end{definition}

\begin{remark}\label{normal basis}
	For a hermitian $ O_F$-lattice $L$, we say a basis $e_1,\dots,e_m$ of $L$ is a \textit{normal basis} if its moment matrix $T=(h(e_i,e_j))_{i,j=1}^m$ is conjugate to
	$$(\beta_1\pi^{2b_1})\oplus\cdots\oplus(\beta_s\pi^{2b_s})\oplus\begin{pmatrix}
	0&\pi^{2c_1-1}\\ -\pi^{2c_1-1}&0 
	\end{pmatrix}\oplus\cdots\oplus\begin{pmatrix}
	0&\pi^{2c_t-1}\\ -\pi^{2c_t-1}&0 
	\end{pmatrix}$$
	by a permutation matrix, for some $\beta_1,\dots,\beta_s\in O_{F_0}^{\times}$ and $b_1,\dots,b_s,c_1,\dots,c_t\in\dZ$.  We have (\cite{LL22}*{Lemma 2.12})
	\begin{itemize}
		\item [(1)] normal basis exists;
		\item [(2)] the invariants $s, t$ and $b_1,\dots,b_s,c_1,\dots,c_t$ depend only on $L$;
		\item [(3)] when $L$ is integral, the fundamental invariants of $L$ are the unique nondecreasing rearrangement of $(2b_1+1,\dots,2b_s+1,2c_1,\dots,2c_t)$.
	\end{itemize}
\end{remark}

\begin{definition}\label{defn for selfdual and unimodular}
	Denote by $H$ the standard hyperbolic hermitian $ O_F$-lattice of rank $2$ given by the matrix $\begin{psmallmatrix}
	0& \pi^{-1}\\ -\pi^{-1}&0
	\end{psmallmatrix}$. For an integer $s\geq 0$, let $H^s\coloneqq H^{\obot s}$ be the orthogonal direct sum of $s$ copies of $H$. Then $H^s$ is a self-dual hermitian $ O_F$-lattice of rank $2s$. 
	
	Let $\epsilon\in O_{F_0}^{\times}$. Denote by $I_1^{\epsilon}$ the rank $1$ hermitian $ O_F$-lattice $\langle x\rangle$ with $h(x,x)=\epsilon$.

    We also use the same notations for $O_{\breve{F}}$-lattices.
\end{definition}

Let $L$ be an integral hermitian $O_F$-lattice. For any $d\geq 1$, the hermitian form $h$ on $L$ induces a pairing $\bar{h}$ valued in $O_F/\pi^{2d-1}O_F$ on $L/\pi^{2d}L$ by
    $$\bar{h}(x,y)\coloneqq h(\tilde{x},\tilde{y})\mod \pi^{2d-1}O_F,\quad\text{ for any }x,y\in L/\pi^{2d}L$$
    where $\tilde{x},\tilde{y}\in L$ is any lift of $x,y\in L/\pi^{2d}L$. The definition is independent of choices of lifts as $L$ is integral.

\begin{definition}\label{Den def 1}
	Let $(M,h_M)$ and $(L,h_L)$ be two integral hermitian $ O_F$-lattices. We denote by $\operatorname{Herm}_{L,M}$ the scheme of hermitian $ O_F$-module homomorphisms from $L$ to $M$, which is a scheme of finite type over $\Spec O_{F_0}$. We define the local density to be
	$$\operatorname{Den}(M,L)\coloneqq \operatorname{lim}_{d\to +\infty}\frac{|\operatorname{Herm}_{L,M}( O_{F_0}/(\pi_0^{d}))|}{q^{d\cdot d_{L,M}}}$$
	where $d_{L,M}$ is the dimension of $\operatorname{Herm}_{L,M}\otimes_{ O_{F_0}}F_0$. Here $\operatorname{Herm}_{L,M}(O_{F_0}/\pi_0^d)$ is given by the set
	$$\left\{\phi\in\operatorname{Hom}_{O_F}(L/\pi^{2d}L,M/\pi^{2d}M) \;\middle|\;
    \begin{array}{c}
    \bar{h}_M(\phi(x),\phi(y))= \bar{h}_L(x,y)
    \text{ for any }x,y\in L/\pi^{2d}L
    \end{array} \right\}.$$
\end{definition}

\begin{remark}\label{dim}
Let $M$ and $L$ be two hermitian $ O_F$-lattices of rank $m$ and $n$ and assume $\operatorname{Herm}_{L,M}\otimes_{ O_{F_0}}{F_0}\neq\emptyset$. Then
	$$d_{L,M}=\operatorname{dim}U_m-\operatorname{dim}U_{m-n}=n(2m-n).$$
\end{remark}

It is well-known that there is a local density polynomial $\operatorname{Den}(M,L,X)\in\dQ[X]$ such that
$$\operatorname{Den}(M,L,q^{-2k})=\operatorname{Den}(M\obot H^k,L).$$

\begin{definition}\label{local density}
	Let $L$ be a nonsplit hermitian $O_F$-lattice of rank $n$. 	Define
    \begin{equation*}
        M_n\coloneqq \begin{cases}
            H^r,&\text{ if }n=2r\text{ is even},\\
            H^r\obot I_1^1,&\text{ if }n=2r+1\text{ is odd}.
        \end{cases}
    \end{equation*}
    Then $M_n$ is always a split hermitian $O_F$-lattice and $\Den(M_n,L)=0$. We define
	$$\operatorname{Den}'(L)\coloneqq -2\left.\frac{\rd}{\rd X}\right|_{X=1}\operatorname{Den}(M_n,L,X)\text{ and }\partial\operatorname{Den}(L)\coloneqq \frac{\operatorname{Den}'(L)}{\operatorname{Den}(M_n,M_n)}.$$
\end{definition}

\begin{remark}\label{re:whittaker}
	Let $L\subset\dV_n$ be an $O_F$-lattice. Let $T\in\GL_n(F)$ be a representing moment matrix of $L$, and consider the $T$-th Whittaker function $W_T(s,1_{2n},\CF_{M_n^n})$ of the Schwartz function $\CF_{M_n^n}$ at the identity element $1_{2n}$. By \cite{KR14}*{Proposition~10.1},\footnote{In \cite{KR14}*{Proposition~10.1} and its proof, the lattice $L_{r,r}$ should be replaced by $M_n$.} we have
	\[
	W_T(s,1_{2n},\CF_{M_n^n})=\Den(M_n\obot H^s,L)
	\]
	for every integer $s\geq 0$. Thus, we obtain
	\[
	\log q\cdot\partial\Den(L)=\frac{W'_T(0,1_{2n},\CF_{M_n^n})}{\Den(M_n,M_n)}
	\]
	by Definition \ref{local density}.
\end{remark}

\subsection{Kudla--Rapoport conjecture}

The main result of this paper is the following theorem.

\begin{theorem}\label{main theorem}
	For any hermitian $O_F$-lattice $L\subset \dV_n$ we have
	$$\Int_{n,\cY}(L)=\partial\Den(L).$$
\end{theorem}

The rest of the section will prove this for $n$ even by reformulating the result in \cite{LL22}.

In \cite{LL22}, Li and Liu defined a normalized hermitian form $h'$ on $\dV_n\colon$
$$\text{ for any }x,y\in\dV_n, h'(x,y)\coloneqq \pi^{-2}h(x,y).$$
For any $O_F$-lattice $L\subset\dV_n$, we shall use the notation $\partial\Den_{h'}(L)$ to emphasis that $L$ is viewed as a hermitian $O_F$-lattice via $h'$ (instead of $h$). When $n$ is even, normalizing the hermitian form on $L$ by $-1$ does not change $\partial\Den(L)$ by the formula given in \cite{LL22}*{Lemma 2.15}. We then have relation $\partial\Den_{h'}(L)=\partial\Den_{-h'}(L)=\partial\Den(\pi^{-1} L)$.

The following theorem is one of the main results in \cite{LL22}.

\begin{theorem}[\cite{LL22}*{Theorem 2.11}]\label{even z KR}
	Let $n\geq 2$ be even. For any $O_F$-lattice $L\subset\dV_n$ we have
	$$\Int_{n,\cZ}(L)=\partial\Den_{h'}(L).$$
\end{theorem}

\begin{remark}\label{re:polarization}
	Suppose $n\geq 2$ is even. For any hermitian formal $O_F$-module $(X,\iota_X,\lambda_X)$ of signature $(1,n-1)$, the polarization $\lambda_X$ of $X$ satisfies $\Ker(\lambda_X)=X[\iota_X(\pi)]$. Then there is a unique $\iota_X$-compatible morphism $\sigma_X\colon X\to X^\vee$ satisfying $\lambda_X=\sigma_X\circ\iota_X(\pi)$, which is in fact a symmetrization, i.e.\ an isomorphism with $\sigma_X^{\vee}=\sigma_X$, . Conversely, given a $\iota_X$-compatible symmetrization $\sigma_X$ of $X$, we may recover $\lambda_X$ as $\sigma_X\circ\iota_X(\pi)$. In what follows, we call $\sigma_X$ the symmetrization of $\lambda_X$.
\end{remark}

\begin{proposition}\label{y=pi z even}
	Let $n\geq 2 $ be even. For any nonzero $x\in\dV_n$, we have a natural morphism
    $$\xymatrix@R=0ex{\cY(x)\quad\ar[r]&\quad\cZ(\pi x)\\
    (X,\ucE\to X^{\vee})\ar@{|->}[r]&(X,\ucE\to X^{\vee}\stackrel{\sigma_X^{-1}}{\lrarr}X)   
    }$$
	which turns out to be an isomorphism. Consequently, for any $O_F$-lattice $L\subset \dV_n$ we have
	\begin{equation*}
	\Int_{n,\cY}(L)=\Int_{n,\cZ}(\pi L).
	\end{equation*}
    In particular $\Int_{n,\cY}(L)$ is well-defined, independent of a choice of basis of $L$.
\end{proposition}
\begin{proof}
	The statements are all clear and easy to verify$\colon$the inverse morphism is just given by
	$$\xymatrix@R=0ex{\cZ(\pi x)\quad\ar[r]&\quad\cY(x)\\
    (X,\ucE\to X^)\ar@{|->}[r]&(X,\ucE\to X\stackrel{\sigma_X}{\lrarr}X^{\vee}).   
    }$$
\end{proof}

\begin{corollary}\label{main theorem for even}
	Theorem \ref{main theorem} holds for any $n$ even.
\end{corollary}
\begin{proof}
	When $n$ is even, for any $O_F$-lattice $L\subset\dV_n$, the $h$-hermitian local density $\partial\Den(L)$ is the same as the $h'$-hermitian local density $\partial\Den_{h'}(\pi L)$. Combine Theorem \ref{even z KR} and Proposition \ref{y=pi z even} we get
	$$\Int_{n,\cY}(L)=\Int_{n,\cZ}(\pi L)=\partial\Den_{h'}(\pi L)=\partial\Den(L).$$
\end{proof}

\section{Reduction at analytic side}

To simplify the notations, we first define \textit{standard normal base} for $H^s$ and $H^s\obot I_1^\epsilon$.

\begin{definition}\label{standard normal basis}
	We say an $O_F$-basis $e_1,f_1,\dots,e_s,f_s$ (or an $O_{\breve{F}}$-basis if we view $H^s$ as an $O_{\breve{F}}$-lattice) of $H^s$ is a \textit{standard normal basis} if the only nonzero hermitian pairing between them is
	$$h(e_i,f_i)=-h(f_i,e_i)=\pi^{-1}\text{ for each }i.$$
	In other words, for each $i$, $e_i,f_i$ is a normal basis for $H$, and for $i\neq j$, these $H$ are orthogonal to each other.
	
	We say a basis $e_1,f_1,\dots,e_s,f_s,\varphi$ of $H^s\obot I_1^\epsilon$ is a \textit{standard normal basis} if $e_1,f_1,\dots,e_s,f_s$ is a standard normal basis for $H^s$, $\varphi$ is a basis for $I_1^\epsilon$ orthogonal to all $e_i,f_j$ with $h(\varphi,\varphi)=\epsilon$.
\end{definition}

Our strategy to deal with the analytic side is to relate it to the even dimensional case which is one dimensional higher. We first prove a key lemma which will be used repeatedly later.

\begin{lemma}\label{ortho in H^s}
	Let $M=H^s$ and $\phi\in M$ be any element such that $\phi\notin\pi M$. Then the submodule
	$$M(\phi)\coloneqq \left\{x\in M \res h(x,\phi)=0\right\}$$
	can be described as
	$$M(\phi)=H^{s-1}\obot I$$
	where $I=\langle\phi'\rangle_{O_F}$ is a rank one $O_F$-lattice such that $h(\phi',\phi')=-h(\phi,\phi)$ and $\phi'-\phi\in\pi M$.
\end{lemma}
\begin{proof}
 Choose a standard normal basis $e_1,f_1,e_2,f_2,\dots,e_s,f_s$ of $M$. Write
	$$\phi=a_1e_1+b_1f_1+a_2e_2+b_2f_2+\dots+a_se_s+b_sf_s.$$
 As $\phi\notin \pi M$, without loss of generality we may assume $a_1\in O_F^{\times}$. Then divide $\phi$ by $a_1$ does not change the submodule $M(\phi)$ so we may assume $a_1=1$.
	
	For $i\geq 2$ we define
	\begin{equation}\label{sigma and tau}
	    \sigma_i\coloneqq e_i+\bar{b}_if_1,\quad \tau_i\coloneqq f_i-\bar{a}_if_1.
	\end{equation}
	
	It is a direct computation to verify that
	$$\begin{aligned}
	h(\phi,\sigma_i)&=h(\phi,\tau_i)=0,\\
	h(\sigma_i,\sigma_j)&=h(\tau_i,\tau_j)=0,\\
	h(\sigma_i,\tau_j)&=-h(\tau_i,\sigma_j)=\pi^{-1}\delta_{i,j}.
	\end{aligned}$$
	In particular the $O_F$-span of $\{\sigma_i,\tau_i\}$ for $i\geq 2$ is isomorphic to $H^{s-1}\subset M(\phi)$.
	
	Now we let 
    \begin{equation}\label{phi'}
        \phi'\coloneqq \phi+h(\phi,\phi)\pi f_1=e_1+(b_1+h(\phi,\phi)\pi)f_1+\sum_{i=2}^s (a_i e_i+b_i f_i).
    \end{equation}
    Then it is easy to compute
	$$\begin{aligned}
	h(\phi',\phi)&=h(\phi',\sigma_i)=h(\phi',\tau_i)=0,\\
	h(\phi',\phi')&=-h(\phi,\phi).
	\end{aligned}$$
	Let $I\coloneqq \langle\phi'\rangle_{O_F}$ and we have
	$$\langle\phi',\sigma_2,\tau_2,\dots,\sigma_s,\tau_s\rangle_{O_F}\simeq I\obot H^{s-1}\subset M(\phi).$$
	
	For any $x\in M(\phi)$, we may use $\phi'$ to get rid of $e_1$ in the expression of $x$, $\sigma_i$ to remove $e_i$, and $\tau_i$ to remove $f_i$ in $x$ for all $i\geq 2$. Then $x\in M(\phi)$ is just a multiple of $f_1$, which must be $0$. We have shown that indeed $M(\phi)\simeq I\obot H^{s-1}$ as described.
\end{proof}

\begin{corollary}\label{normal basis and ortho}
	If $\phi\in M=H^s$ and $\phi\notin\pi M$ with $h(\phi,\phi)=\beta$, then there is a standard normal basis $\sigma_i,\tau_i$ of $M$ such that $\sigma_1+\frac{\pi\beta}{2}\tau_1,\sigma_2,\tau_2,\dots,\sigma_s,\tau_s$ is a normal basis of $M(\phi)$ while $\phi=\sigma_1-\frac{\pi\beta}{2}\tau_1$. Consequently if $\beta\neq 0$, i.e.~ $\phi$ is not isotropic, then
    \begin{equation*}
        M(\phi)\obot\langle \phi\rangle_{O_F}\subset^b M\quad\text{ and }\quad M(\phi)_F\obot\langle \phi\rangle_F\simeq M\otimes_{O_F}F
    \end{equation*}
    where $b=1+\val_\pi(\beta)$.
\end{corollary}
\begin{proof}
	As in the previous lemma, choose a standard normal basis $e_i,f_i$ of $M$. Write $\phi=a_1 e_1+b_1 f_1+\dots+a_s e_s+b_s f_s$ and we may assume $a_1\in O_F^{\times}$ and for simplicity we shall just assume it is already 1. Then by the proof of the previous lemma, we can construct $\sigma_2,\tau_2,\dots,\sigma_s,\tau_s$. All we need to do is to find $\sigma_1$ and $\tau_1$.
	
	Use the same notation, $\phi'=\phi+\beta\pi f_1$. We just set
	\begin{equation}\label{sigma 1}
    \sigma_1\coloneqq \frac{\phi+\phi'}{2}=\phi+\frac{\pi\beta}{2}f_1\text{ and }\tau_1\coloneqq \frac{\phi'-\phi}{\pi\beta}=f_1.
    \end{equation}
	It is easy to verify $\sigma_1,\tau_1$ is a normal basis of $H$.

    Then $M(\phi)\obot\langle\phi\rangle_{O_F}=\langle \sigma_1,\beta\pi \tau_1,\sigma_2,\tau_2,\dots,\sigma_s,\tau_s\rangle_{O_F}$ has co-length $1+\val_\pi(\beta)$ in $M$ if $\beta\neq 0$.
\end{proof}

\begin{lemma}\label{density normalization}
    Let $M_n=H^{n/2}$ if $n$ is even, and $M_n=H^{(n-1)/2}\obot I_1^1$ if $n$ is odd. Then
    \begin{equation*}
        \Den(M_n,M_n)=\begin{cases}
            {\displaystyle\prod_{i=1}^{r}(1-q^{-2i})},&\quad\text{if }n=2r\text{ even},\\[2ex]
            2{\displaystyle\prod_{i=1}^{r}(1-q^{-2i})},&\quad\text{if }n=2r+1\text{ odd}.
        \end{cases}
    \end{equation*}
\end{lemma}
\begin{proof}
    By \cite{LL22}*{Lemma 2.15}, if $n=2r$ is even, we get $\Den(H^r,H^r)=\prod\limits_{i=1}^r (1-q^{-2i})$.

    Now assume $n=2r+1$ is odd. By \cite{HLSY}*{Lemma 5.13} we have
    \begin{equation*}
    \begin{split}
    \Den(M_n,M_n)&=\Den(I_1^1,M_n,q^{-2r})\\
    &=\( \prod_{\ell=0}^{r-1}(1-q^{2\ell-2r})\)\Den(I_1^1,I_1^1,1)=\( \prod_{i=1}^r(1-q^{-2i})\)\Den(I_1^1,I_1^1)
    \end{split}
    \end{equation*}
    and by \cite{HLSY}*{Corollary 5.17} we have
    \begin{equation*}
        \Den(I_1^1,I_1^1)=2.
    \end{equation*}
\end{proof}

\begin{proposition}\label{analytic reduction}
	Let $n+1=2r+2\geq 2$ be even. Let $L^{\#}\subset\dV_{n+1}$ be a hermitian $O_F$-lattice and assume $L^{\#}$ has a hermitian lattice decomposition $L^{\#}=L\obot I_1^{-1}$. Then $$\partial \operatorname{Den}(L)=\frac{1}{2}\partial\Den(L^{\#}).$$
\end{proposition}

\begin{proof}
	Let $M=H^s$ and choose a standard normal basis $e_1,f_1,e_2,f_2,\dots,e_s,f_s$ of $M$.
	
	
	Let $d\geq 1$. Consider the restriction map
	$$R\colon\operatorname{Herm}_{L^{\#},M}( O_{F_0}/\pi_0^{d})\to\operatorname{Herm}_{I_1^{-1},M}( O_{F_0}/\pi_0^{d}).$$
	We claim that the fibre of this map can be identified with $\operatorname{Herm}_{L,M'}( O_{F_0}/\pi_0^{d})$ where $M'= M_{2s-1}=I_1^1\obot H^{s-1}$.
	
	Let $\phi\in\operatorname{Herm}_{I_1^{-1},M}( O_{F_0}/\pi_0^{d})$. Then $\phi$ is determined by
	$\bar{x}\in M/\pi^{2d}M$ with $\bar{h}(\bar{x},\bar{x})=-1$. For any lift $x\in M$ we have $h(x,x)\equiv -1 \mod\pi^{2d}O_F\in O_{F_0}$. We first modify the lift such that $M(x)\simeq M'$.
 
    Write $h(x,x)=-1+\pi^{2d}\alpha$ for some $\alpha\in O_{F_0}$. From Lemma \ref{ortho in H^s} we can find $y\in M(x)$ with $h(y,y)=-h(x,x)$. Then $x'=x+\pi^{2d}(ax+ay)$ for $a=-\frac{\alpha}{2}$ is also a lift of $\bar{x}$. Now we compute
	$$\begin{aligned}
	h(x',x')&=h((1+\pi^{2d}a)x+\pi^{2d}ay,(1+\pi^{2d}a)x+\pi^{2d}ay)\\
	&=(1+\pi^{2d}a)^2h(x,x)+(\pi^{2d}a)^2h(y,y)\\
	&=(1+2\pi^{2d}a)h(x,x)=-h(x,x)^2.
	\end{aligned}$$
	As $h(x,x)$ is in $O_{F_0}^{\times}$, the submodule $M(x')$ is the same as $M(x'/h(x,x))$, and the latter is just $M'$. We shall fix such a lift and still denote it by $x$ from now on.
	
	Choose an $ O_F$-basis $\ell_1,\dots,\ell_{n}$ for $L$. The fibre over $\phi$ is given by
	$$R^{-1}(\phi)=\left\lbrace \bar{x}_1,\dots,\bar{x}_{n}\in M/\pi^{2d}M \;\middle|\;
    \begin{array}{c}
	    \bar{h}(\bar{x}_i,\bar{x})= 0 \text{ for any }i,\\ \bar{h}(\bar{x}_i,\bar{x}_j)= h(\ell_i,\ell_j) \mod\pi^{2d-1}O_F\text{ for any }i,j
	\end{array}\right\rbrace.$$
	Let $K=\{\bar{y}\in M/\pi^{2d}M \res \bar{h}(\bar{y},\bar{x})= 0 \}$. Clearly there is an injective map $$M(x)/\pi^{2d}M(x)\to M/\pi^{2d}M$$ and the image lies in $K$. We next show that the image is exactly $K$.
	
	For any $\overline{y}\in K$ we need to find a lift of $\bar{y}$ in $M(x)$. Choose any lift $y\in M$, then $h(y,x)\in\pi^{2d-1} O_F$. As $h(x,x)\neq 0$, by Corollary \ref{normal basis and ortho} we have $M\otimes_{ O_F}F\simeq \langle x\rangle_{F}\obot M(x)_{F}$. So we may write $y=y_x+y_r$ according to the decomposition. Then $h(y,x)=h(y_x,x)\in\pi^{2d-1} O_F$. Write $y_x=a'\cdot x$ for some $a'\in F$, we must have $a'\cdot h(x,x)\in\pi^{2d-1} O_F$. As $h(x,x)$ is a unit, $a'\in \pi^{2d-1} O_F$ and $y_x\in \pi^{2d-1}M$. Thus $y_r=y-y_x\in M(x)$. If $a'\in \pi^{2d}O_F$ then $y_r$ is a lift of $\overline{y}$ in $M(x)$. Otherwise suppose $a'=\pi^{2d-1}\lambda$ for some $\lambda\in O_F^{\times}$ and write $h(x,x)=\beta\in O_{F_0}^{\times}$. By Corollary \ref{normal basis and ortho}, we can find a standard normal basis $\sigma_i,\tau_i$ of $M$ such that $x=\sigma_1-\frac{\pi\beta}{2}\tau_1$. Consider $y'=y+\pi^{2d}z$ where $z=\lambda\beta\tau_1$. Then $y'$ is a lift of $\bar{y}$ and
	$$h(y',x)=h(y_x+\pi^{2d}z,x)=\pi^{2d-1}\lambda\beta+\pi^{2d}h(z,x)=\pi^{2d-1}(\lambda\beta+\pi\cdot h(\lambda\beta\tau_1,\sigma_1-\frac{\pi\beta}{2}\tau_1))=0.$$
	Thus $y'$ is a lift of $\bar{y}$ in $M(x)$. We have shown that $$K=M(x)/\pi^{2d}M(x)\simeq M'/\pi^{2d}M'$$
	and consequently $$R^{-1}(\phi)\simeq \operatorname{Herm}_{L,M'}( O_{F_0}/\pi_0^d).$$
	
	As a corollary $$|\operatorname{Herm}_{L^{\#},M_{2s}}( O_{F_0}/\pi_0^d)|=|\operatorname{Herm}_{L,M_{2s-1}}( O_{F_0}/\pi_0^d)|\cdot|\operatorname{Herm}_{I_1^{-1},M_{2s}}( O_{F_0}/\pi_0^d)|.$$
	
	Then by Definition \ref{Den def 1} and Remark \ref{dim},
	$$\begin{aligned}
	\operatorname{Den}(M_{n+1},L^{\#},q^{-2k})&=\operatorname{Den}(M_{n+1}\obot H^k,L^{\#})\\
	&=\operatorname{lim}_{d\to+\infty}\frac{|\operatorname{Herm}_{L^{\#},M_{n+1+2k}}( O_{F_0}/\pi_0^d)|}{q^{d\cdot (n+1)(2(n+1+2k)-n-1)}}\\
	&=\operatorname{lim}_{d\to+\infty}\frac{| \operatorname{Herm}_{L,M_{n+2k}}( O_{F_0}/\pi_0^d) |\cdot| \operatorname{Herm}_{I_1^{-1},M_{n+1+2k}}( O_{F_0}/\pi_0^d) |}{q^{d\cdot (n+1)(n+1+4k)}}\\
	&=\operatorname{lim}_{d\to+\infty}\frac{| \operatorname{Herm}_{L,M_{n+2k}}( O_{F_0}/\pi_0^d) |}{q^{d\cdot n(2(n+2k)-n)}} \cdot
	\frac{| \operatorname{Herm}_{I_1^{-1},M_{n+1+2k}}( O_{F_0}/\pi_0^d) |}{q^{d\cdot 1 \cdot (2(n+1+2k)-1)}}\\
	&=\operatorname{Den}(M_{n}\obot H^k,L)\cdot\operatorname{Den}(M_{n+1}\obot H^k,I_1^{-1})\\
	&=\operatorname{Den}(M_{n},L,q^{-2k})\cdot (1-q^{-1-n-2k})
	\end{aligned}$$
 where the last equality $\operatorname{Den}(M_{n+1}\obot H^k,I_1^{-1})=1-q^{-1-n-2k}$ comes from \cite{LL22}*{Lemma 2.15}.
	
	Thus
	$$\begin{aligned}
	\operatorname{Den}(M_{n+1},L^{\#},X)&=\operatorname{Den}(M_{n},L,X)\cdot(1-q^{-1-n}X)\\
	\operatorname{Den}'(L^{\#})&=\operatorname{Den}'(L)\cdot(1-q^{-1-n})+2\operatorname{Den}(M_{n},L,1)\cdot q^{-1-n}\\&=(1-q^{-1-n})\operatorname{Den}'(L).\quad\quad\text{ the second term vanishes as }L\text{ is nonsplit.}
	\end{aligned}$$
	
	Divide both sides by $\prod\limits_{i=1}^{r+1}(1-q^{-2i})$. By the computation result from Lemma \ref{density normalization}, we get 
    $$\partial\operatorname{Den}(L^{\#})=\frac{\Den'(L^{\#})}{\Den(M_{n+1},M_{n+1})}=\frac{(1-q^{-1-n})\Den'(L)}{\prod\limits_{i=1}^{r+1}(1-q^{-2i})}\stackrel{n+1=2r+2}{=\joinrel=\joinrel=\joinrel=\joinrel=}\frac{\Den'(L)}{\prod\limits_{i=1}^r(1-q^{-2i})}=2\partial\Den(L).$$
\end{proof}

\section{Reduction at geometric side}\label{geom red section}
Throughout the section we let $n=2m\geq 4$ be even.

\subsection{Auxiliary Rapoport--Zink spaces}\label{other condition for even}

	 Following \cite{RSZ18}*{\S9}, to describe the relation between $\mathcal{N}_{n-1}$ and $\mathcal{N}_n$ we first introduce some auxiliary spaces.

\begin{definition}
	We consider the moduli space $\mathcal{P}_n$ over $\operatorname{Spf}O_{\breve{F}}$, sending $S\in\operatorname{Nilp}_{O_{\breve{F}}}$ to the set of isomorphism classes of quadruples $(X,\iota,\lambda,\rho)$ where
    \begin{itemize}
        \item $X$ is a (strict) formal $O_{F_0}$-module over $S$ of dimension $n$ and relative height $2n$;
        \item $\iota$ is an action of $O_F$ on $X$ extending the $O_{F_0}$-action;
        \item $\lambda$ is an $\iota$-compatible polarization of $X$;
        \item $\rho\colon X\times_S\ovS\to\dX'_n\times_{\bar{k}}\ovS$ is an $O_F$-linear quasi-isogeny of height $0$ over $\ovS$ such that $\rho^*(\lambda_{\dX'_n,\ovS})=\lambda_{\ovS}$,
    \end{itemize}
    with the framing object $(\mathbb{X}_n',\iota_{\mathbb{X}_n'},\lambda_{\mathbb{X}_n'})$ to be specified below. Here we impose on the polarization $\lambda$ that
	$$\operatorname{ker}\lambda\subset X[\iota(\pi)]\text{ is of rank }q^{n-2},$$
	and we require the triple $(X,\iota,\lambda)$ to satisfy condition \eqref{even spin cond} below, which is the natural analog of condition \eqref{odd spin cond}.\footnote{In \cite{RSZ18}*{\S9}, the triple is also required to satisfy the Wedge condition \eqref{wedge cond}. It is shown in \cite{Yu19}*{Remark 1} that condition \eqref{even spin cond} implies Wedge condition.}
\end{definition}

  As in \S\ref{condition for n odd}, let $M(X)$ and $M(X^\vee)$ denote the respective Lie algebras of the universal vector extensions of $X$ and $X^\vee$.  Since $\ker \lambda$ is contained in $X[\iota(\pi)]$ and of rank $q^{n-2}$, there is a unique (necessarily $O_F$-linear) isogeny $\lambda'$ such that the composite
\[
X \xra{\lambda} X^\vee \xra{\lambda'} X
\]
is $\iota(\pi)$, and the induced diagram
\[
M(X) \xra{\lambda_*} M(X^\vee) \xra{\lambda'_*} M(X)
\]
then extends periodically to a polarized chain of $O_F \otimes_{O_{F_0}} \cO_S$-modules of type $\Lambda_{\{m-1\}}$.  By \cite{RZ96}*{Theorem~3.16}, \'etale-locally on $S$ there exists an isomorphism of polarized chains
\begin{equation*}
[{}\dotsb \xra{\lambda'_*} M(X) \xra{\lambda_*} M(X^\vee) \xra{\lambda'_*} \dotsb{}] \stackrel{\sim}{\lrarr} \Lambda_{\{m-1\}} \otimes_{O_{F_0}} \cO_S,
\end{equation*}
which in particular gives an isomorphism of $O_F \otimes_{O_{F_0}} \cO_S$-modules
\begin{equation}\label{M(X) triv 2}
M(X) \stackrel{\sim}{\lrarr} \Lambda_{-(m-1)} \otimes_{O_{F_0}} \cO_S.
\end{equation}
Denoting by $\Fil^1 \subset M(X)$ the covariant Hodge filtration for $X$, the analog of \eqref{odd spin cond} we impose is that
\begin{align*}
&\text{\em upon identifying $\Fil^1$ with a submodule of $\Lambda_{-(m-1)} \otimes_{O_{F_0}} \cO_S$ via \eqref{M(X) triv 2}, the line bundle}\\
\intertext{
	\begin{equation}\label{even spin cond}
	\sideset{}{_{\cO_S}^n}\bigwedge \Fil^1 \subset  \tensor[^n]\Lambda{_{-(m-1)}} \otimes_{O_F} \cO_S
	\end{equation}
}
&\text{\em is contained in $L_{-(m-1),-1}^{n-1,1}(S)$, cf.~(\ref{L def}).}
\end{align*}
Just as before, condition \eqref{even spin cond} is independent of the above choice of chain isomorphism by Lemma \ref{L stability}.

As in $\S$\ref{framing objects}, over $\bar{k}$ there are \emph{two} supersingular isogeny classes of framing objects for this moduli problem, distinguished by the splitness of the hermitian space $\dV$ in \eqref{space of special homs}.  To complete the definition of $\cP_n$, we will choose a framing object $\dX'_n$ for which $\dV(\dX'_n)$ is nonsplit.  We take $\dX'_n$ to be the product of the framing object $\dX_{n-1} = \dX_{n-1}^{(1)}$ (cf.\ \eqref{BX_n odd}) and $\ubE$,
\begin{equation}\label{wtBX_n^(1) def}
\bigl(\dX'_n, \iota_{\dX'_n}, \lambda_{\dX'_n} \bigr) := 
\bigl(\dX_{n-1} \times \ubE, \iota_{\dX_{n-1}} \times \iota_{\ubE}, \lambda_{\dX_{n-1}} \times (-\lambda_{\ubE}) \bigr).
\end{equation}
Since $\dV(\dX_{n-1})$ is nonsplit and $n$ is even, it follows from definition that $\dV(\dX'_n)$ is indeed the nonsplit space of dimension $n$.  Furthermore, it is easy to see that $\dX'_n$ satisfies \eqref{even spin cond} because $\dX_{n-1}$ satisfies \eqref{odd spin cond}.

It follows from the general theory of Rapoport--Zink spaces \cite{RZ96} that $\mathcal{P}_n$ is represented by a formal scheme which is formally locally of finite type and essentially proper over $\operatorname{Spf}O_{\breve{F}}$. However it is not regular.

To define the second moduli space we first fix an $O_F$-linear isogeny of degree $q$,
$$\phi_0\colon\mathbb{X}_n\lrarr\mathbb{X}_n'$$
such that $\phi_0^*(\lambda_{\mathbb{X}_n'})=\lambda_{\mathbb{X}_n}$. Since $n\geq 4$, we have
$$\mathbb{X}_n=\mathbb{X}_n'=\mathbb{X}_{n-2}\times\ubE\times\ubE$$
as $O_F$-modules, and we define
\begin{equation}\label{phi0}
\phi_0\coloneqq\operatorname{id}_{\mathbb{X}_{n-2}}\times\begin{pmatrix}
1&\frac{\iota_{\ubE}(\pi)}{2}\\1&-\frac{\iota_{\ubE}(\pi)}{2}
\end{pmatrix}.
\end{equation}
It is easy to check that $\phi_0^*(\lambda_{\mathbb{X}_n'})=\lambda_{\mathbb{X}_n}$.

\begin{definition}
	We consider the moduli space $\mathcal{P}_n'$ for tuples
	$$(X,\iota,\lambda,\rho,X',\iota',\lambda',\rho',\phi\colon X\to X')$$
	where $(X,\iota,\lambda,\rho)$ is a point on $\mathcal{N}_n$, $(X',\iota',\lambda',\rho')$ is a point on $\mathcal{P}_n$, and $\phi$ is an $O_F$-linear isogeny of degree $q$ lifting $\phi_0$ in the sense that the following diagram commutes
	$$\xymatrix{
		X_{\ovS}\ar[rr]^{\bar{\phi}}\ar[d]_{\rho}&&X'_{\ovS}\ar[d]^{\rho'}\\\mathbb{X}_{n,\ovS}\ar[rr]^{\phi_0}&&\mathbb{X}'_{n,\ovS}.}$$
	The notion of isomorphism between tuples as above is the obvious one.
\end{definition}

By definition, there are tautological projection maps
$$\xymatrix{&\mathcal{P}_n'\ar[ld]\ar[rd]^-{\varphi}&\\
	\mathcal{P}_n& &\mathcal{N}_n.}$$

By Proposition \ref{even n decomp lem}, $\cN_n$ decomposes into a disjoint union $\cN_n = \cN_n^+ \amalg \cN_n^-$.  Pulling back along $\varphi$, we obtain a decomposition
\begin{equation}\label{wtCN'pm}
\cP_n' = (\cP_n')^+ \amalg (\cP_n')^-.
\end{equation}

\begin{theorem}[\cite{RSZ18}*{Theorem 9.3}]\label{wtCN'pm isom}
	Writing $(\cP_n')^\pm$ for either of the summands in \eqref{wtCN'pm}, the projection $\cP_n' \to \cP_n$ induces an isomorphism
	\[
	(\cP_n')^\pm \stackrel{\sim}{\lrarr} \cP_n.
	\]
\end{theorem}

We denote by $\psi^{\pm}\colon\cP_n\stackrel{\sim}{\lrarr}(\cP'_n)^{\pm}$ the inverse isomorphism.

There is a natural closed embedding
$$\widetilde{\delta}\colon \mathcal{N}_{n-1}\lrarr\mathcal{P}_n$$
sending a quadruple $(X',\iota',\lambda',\rho')\in \mathcal{N}_{n-1}(S)$ for $S\in\operatorname{Nilp}_{O_{\breve{F}}}$ to $$(X'\times\ucE_S,\iota'\times\iota_{\ucE},\lambda'\times(-\lambda_{\ucE}),\rho'\times\rho_{\ucE})\in\mathcal{P}_n(S).$$
It is straightforward to verify that $\widetilde{\delta}$ is well-defined.

Consider the composite morphism
\begin{equation}\label{def delta}
    \delta^{\pm}\colon\mathcal{N}_{n-1}\stackrel{\widetilde{\delta}}{\lrarr}\mathcal{P}_n\stackrel{\psi^{\pm}}{\lrarr}(\mathcal{P}_n')^{\pm}\stackrel{\varphi}{\lrarr}\mathcal{N}_n.
\end{equation}
Explicitly, the morphism $\delta^{\pm}$ sends $(X',\iota',\lambda',\rho')\in\cN_{n-1}(S)$ for $S\in\operatorname{Nilp}_{O_{\breve{F}}}$ to $(X,\iota,\lambda,\rho)\in\cN_n^{\pm}(S)$ such that there exists an $O_F$-isogeny $\phi\colon X\to X'\times\ucE_S$ of degree $q$ lifting $\phi_0$.

\begin{proposition}[\cite{RSZ18}*{Proposition 12.1}]
	The morphism $\delta^{\pm}$ is a closed embedding.
\end{proposition}

The isogeny $\phi_0\colon\dX_n\to\dX'_n=\dX_{n-1}\times\ubE$ gives an isomorphism of hermitian spaces
\begin{equation}\label{embed of dV}
    \phi_0\colon\mathbb{V}_n\stackrel{\sim}{\lrarr}\mathbb{V}_{n-1}\obot\langle f\rangle_F
\end{equation}
where $f=\operatorname{id}_{\ubE}$ has hermitian norm $-1$. This allows us to have an embedding of hermitian spaces $\dV_{n-1}\lrarr\dV_{n}$. Consider $u=(0_{n-2},\frac{\iota_{\ubE}(\pi)}{2},-1)\in\mathbb{V}_n$ which is of hermitian norm $\pi^2$. Then under $\phi_0$, $u$ is mapped to $\pi f$.

The first main result of the section is the following theorem.

\begin{theorem}\label{identify lower RZ space with special divisor}
	The closed embedding $\delta^{\pm}$ identifies $\mathcal{N}_{n-1}$ with $\cZ_n(u)\cap\mathcal{N}_n^{\pm}$.
\end{theorem}

For the rest of the section we will state everything for the morphism $\delta^+$, although the same argument applies to $\delta^-$ as well. For simplicity we set $\cZ_n(u)^+\coloneqq\cZ_n(u)\cap\cN_n^+$.

\subsection{Description of Dieudonn\'e modules}\label{details of D modules}

Before going into the proof of Theorem \ref{identify lower RZ space with special divisor}, we need to do some preparations. Let $r\geq 1$ be any integer. Firstly we give a description of $\bar{k}$-points on $\mathcal{N}_r$ using Dieudonn\'e theory.

Let $\mathbb{N}_r, \overline{\mathbb{N}}_1$ denote the covariant rational Dieudonn\'e modules of $\mathbb{X}_r$ and $\ubE$ respectively. Let $F$ and $V$ be the Frobenius operator and Verschiebung on $\mathbb{N}_r$. The polarization $\lambda_{\mathbb{X}_r}$ induces a nondegenerate alternating $\breve{F_0}$-bilinear form $\langle-,- \rangle$ on $\mathbb{N}_r$ satisfying
$$\langle Fx,y\rangle=\langle x,Vy\rangle^{\sigma}\text{ for all }x,y\in\mathbb{N}_r$$
where $\sigma$ denotes the extension of Frobenius operator on $W_{O_{F_0}}(\bar{k})=O_{\breve{F}_0}$ to $\breve{F_0}$. The action $\iota_{\mathbb{X}_r}$ makes $\mathbb{N}_r$ into an $\breve{F}$-vector space such that the $\breve{F}$-action commutes with $F$ and $V$ and
$$\langle ax,y\rangle=\langle x,\bar{a}y\rangle\text{ for all }x,y\in\mathbb{N}_r, a\in\breve{F}.$$
The form
$$h(x,y)\coloneqq \langle \pi x,y\rangle+\pi\langle x,y\rangle, x,y\in\mathbb{N}_r$$
then makes $\mathbb{N}_r$ into an $\breve{F}/\breve{F_0}$-hermitian space of dimension $r$. Similarly $\overline{\mathbb{N}}_1$ is made into an $\breve{F}/\breve{F_0}$-hermitian space. 

By Dieudonn\'e theory, for a perfect extension $K$ of $\bar{k}$, the set of $K$-points on $\mathcal{N}_r$ identifies with the set of $O_{\breve{F}}\otimes_{O_{\breve{F}_0}}W_{O_{F_0}}(K)$-lattices $M\subset \mathbb{N}_r\otimes_{O_{\breve{F}_0}}W_{O_{F_0}}(K)$ such that$\colon$
\begin{itemize}
	\item [$\bullet$] if $r$ is even, then $M$ is $\pi$-modular, and we have
	\begin{equation}\label{VM}
	\pi_0 M\subset VM\subset M, VM\subset^1VM+\pi M;
	\end{equation}
	\item [$\bullet$] if $r$ is odd, then $M$ is almost $\pi$-modular, and we have
	\begin{equation}\label{VModd}
	\pi_0 M\subset VM\subset M, VM\subset^{\leq 1}VM+\pi M.
	\end{equation}
\end{itemize}
In what follows, for simplicity, we only state our results for $K=\bar{k}$, although all results hold for general perfect extension $K$ of $\bar{k}$ as well.

Note that if $r$ is even and $M$ is $\pi$-modular, then $M\simeq \pi H^{r/2}$. In this case we call an $O_{\breve{F}}$-basis $e_1,f_1,\dots,e_{r/2},f_{r/2}$ a \textit{standard normal basis} of $M$ if $\pi^{-1}e_1,\pi^{-1}f_1,\dots,\pi^{-1}e_{r/2},\pi^{-1}f_{r/2}$ is a standard normal basis of $H^{r/2}$ as in Definition \ref{standard normal basis}; if $r$ is odd and $M$ is almost $\pi$-modular, then $M\simeq \pi H^{(r-1)/2}\obot I_1^1$. In this case we call an $O_{\breve{F}}$-basis $e_1,f_1,\dots,e_{(r-1)/2},f_{(r-1)/2},\varphi$ a \textit{standard normal basis} of $M$ if $e_1,f_1,\dots,e_{(r-1)/2},f_{(r-1)/2}$ is a standard normal basis of $\pi H^{(r-1)/2}$ and $\varphi$ is an $O_{\breve{F}}$-basis of $I_1^1$ with hermitian norm $1$ and orthogonal to all other $e_i,f_j$.

Let $x\in\mathbb{V}_r$ and set $x^*$ to be the quasi-morphism
$$\mathbb{X}_r\stackrel{\lambda_{\mathbb{X}_r}}{\lrarr}\mathbb{X}_r^{\vee}\stackrel{x^\vee}{\lrarr}\ubE^{\vee}\stackrel{\lambda_{\ubE}^{-1}}{\lrarr}\ubE.$$
Then $x^*$ is the adjoint of $x$ with respect to the polarizations, and
$$\langle x\alpha,\beta\rangle=\langle \alpha,x^*\beta\rangle\text{ for any }\alpha\in\overline{\mathbb{N}}_1\text{ and }\beta\in\mathbb{N}_r.$$
From this we can relate the hermitian form on $\mathbb{N}_r$ and $\mathbb{V}_r$. If $x,y\in\mathbb{V}_r$ then
\begin{equation}\label{relation of herm forms}
    h(xe,ye)=\delta\cdot h(x,y).
\end{equation}

\begin{lemma}\label{h(vx,vy)}
    Let $x,y\in\dN_r$. Then we have
    \begin{itemize}
        \item [(1)] $h(Vx,Vy)=0$ if and only if $h(Fx,Fy)=0$, if and only if $h(x,y)=0$;
        \item [(2)] $h(Vx,Vy)\in \pi_0^a O_{\breve{F}}$ if and only if $h(Fx,Fy)\in\pi_0^a O_{\breve{F}}$, if and only if $h(x,y)\in \pi_0^{a-1}O_{\breve{F}}$.
    \end{itemize}
\end{lemma}
\begin{proof}
    We have
    \begin{align*}
        h(Vx,Vy)&=\langle \pi Vx,Vy\rangle+\pi\langle Vx,Vy\rangle
        =\langle F\pi Vx,y\rangle^{\sigma^{-1}}+\pi\langle FVx,y\rangle^{\sigma^{-1}}\\
        &=\langle \pi\pi_0x,y\rangle^{\sigma^{-1}}+\pi\langle \pi_0 x,y\rangle^{\sigma^{-1}}
        =\pi_0\langle \pi x,y\rangle^{\sigma^{-1}}+\pi_0 \pi\langle x,y\rangle^{\sigma^{-1}},\\
        h(Fx,Fy)&=\langle \pi Fx,Fy\rangle+\pi\langle Fx,Fy\rangle
        =\langle \pi x,VFy\rangle^{\sigma}+\pi\langle x,VFy\rangle^{\sigma}\\
        &=\langle \pi x,\pi_0 y\rangle^{\sigma}+\pi\langle x,\pi_0 y\rangle^{\sigma}
        =\pi_0\langle \pi x, y\rangle^{\sigma}+\pi_0\pi\langle x,y\rangle^{\sigma},\\
        h(x,y)&=\langle \pi x,y\rangle+\pi\langle x,y\rangle.
    \end{align*}

    For $(1)$, $h(Vx,Vy)=0$, $h(Fx,Fy)=0$ and $h(x,y)=0$ are all equivalent to $\langle \pi x,y\rangle=\langle x,y\rangle=0$.

    For $(2)$, $h(Vx,Vy)\in \pi_0^a O_{\breve{F}}$, $h(Fx,Fy)\in\pi_0^a O_{\breve{F}}$ and $h(x,y)\in\pi_0^{a-1}O_{\breve{F}}$ are all equivalent to $\langle \pi x,y\rangle$ and $\langle x,y\rangle$ lie in $\pi_0^{a-1}O_{\breve{F}}$.
\end{proof}

\begin{corollary}\label{V and dual}
    Let $L\subset\dN_r$ be a hermitian lattice. Then we have
    \begin{itemize}
        \item [(1)] $\pi L^{*}=(\pi^{-1}L)^*$;
        \item [(2)] $V L^*=\pi_0 (VL)^*$.
    \end{itemize}
\end{corollary}
\begin{proof}
    By definition,
    \begin{align*}
        (\pi^{-1}L)^*&=\{ x\in\dN_r \res h(x,\pi^{-1}y)\in O_{\breve{F}}\text{ for any }y\in L\}\\
        &=\{ x\in\dN_r \res h(\pi^{-1}x,y)\in O_{\breve{F}}\text{ for any }y\in L\}\\
        &=\{x\in\dN_r\res \pi^{-1}x\in L^*\}\\
        &=\pi L^*,
    \end{align*}
    \begin{align*}
        (VL)^*&=\{ x\in\dN_r\res h(x,Vy)\in O_{\breve{F}} \text{ for any }y\in L\}\\
        &=\{ x\in\dN_r\res h(V^{-1}x,y)\in \pi_0^{-1}O_{\breve{F}}\text{ for any }y\in L\}\quad\quad\text{ by Lemma \ref{h(vx,vy)}}\\
        &=\{x\in\dN_r\res h(\pi_0 V^{-1}x,y)\in O_{\breve{F}}\text{ for any }y\in L\}\\
        &=\pi_0^{-1} V L^*.
    \end{align*}
\end{proof}

We have the following detailed description of the Dieudonn\'e module and Hodge filtration of $\ubE$.

\begin{lemma}\label{D module of ubE}
    The Dieudonn\'e module $\overline{\dM}_0$ of $\ubE$ can be identified with $W_{O_{F_0}}(\bar{k})^2=O_{\breve{F}_0}^2$ endowed with the Frobenius and Verschiebung operator given in matrix form by
    \begin{equation}
        F=\begin{pmatrix}
        0& \pi_0\\ 1 & 0
    \end{pmatrix}\sigma\text{ , }
    V=\begin{pmatrix}
        0&\pi_0\\ 1 & 0
    \end{pmatrix}\sigma^{-1}.
    \end{equation}
    Moreover we can find an $O_{\breve{F}_0}$-basis $e,f$ of $\overline{\dM}_0$ such that
    \begin{itemize}
        \item [(1)] $f=-\pi e$ and $\overline{\dM}_0=O_{\breve{F}}e$;
        \item [(2)] $Ve=Fe=f, Vf=Ff=\pi_0 e$;
        \item [(3)] The Hodge filtration of $\ubE$ is given by
        \begin{equation}
        \xymatrix@R=.5em{0\ar[r]& \Fil^1\ar@{=}[d]\ar[r]&\dD(\ubE)\ar@{=}[d]\ar[r]&\Lie\ubE\ar@{=}[d]\ar[r]&0.\\
         & \langle f\rangle_{\bar{k}}&\langle e,f\rangle_{\bar{k}}&\langle e\rangle_{\bar{k}}&}
         \end{equation}
    \end{itemize}
\end{lemma}
\begin{proof}
    The description of the Dieudonn\'e module is the same as $\dE$, so we get the matrix form of $F$ and $V$ as in \eqref{Frob on E}. Take
    \begin{equation*}
        e=\begin{psmallmatrix}
        1\\0
    \end{psmallmatrix}\text{ and }f=\begin{psmallmatrix}
        0\\1
    \end{psmallmatrix}.
    \end{equation*}
    Note that the action of $\pi$ on $\overline{\dM}_0$ is given by $\begin{psmallmatrix}
        0&-\pi_0\\-1&0
    \end{psmallmatrix}$ as we twist the $O_F$-action by Galois conjugation. All the other results follow directly from computation.
\end{proof}


\begin{lemma}\label{existence of good basis}
    Let $z\in\cN_n(\bar{k})$ be any point and $M$ its Dieudonn\'e module. Then we can always find a standard normal basis $\sigma_1,\tau_1,\dots,\sigma_m,\tau_m$ of $M$ with $\sigma_1\in VM$.
\end{lemma}
\begin{proof}
    The Dieudonn\'e module $M$ is an $O_{\breve{F}}$-lattice in $\mathbb{N}_n$ such that $M^*=\pi^{-1}M$ and $\pi_0 M\subset VM\subset M$, $VM\subset^1 VM+\pi M$. Since the isocrystal $\mathbb{N}_n$ is supersingular, the lattices $M$ and $\pi^{-1}VM$ have the same co-length in any lattice that contains them both. The condition $VM\subset^1 VM+\pi M$ then implies that $VM\cap\pi M\subset^1 VM$.
	
	Pick any $e\in VM$ and $e\notin VM\cap \pi M$. Write $e=Ve'$ for some $e'\in M$. Write $n=2m$. As $M$ is $\pi$-modular, $M\simeq \pi H^m$. As we assume $e\notin \pi M$, we have $e'\notin \pi M$. Thus we can apply Lemma \ref{ortho in H^s}\footnote{To be precise, apply to the orthogonal complement of $\pi^{-1}e'$ in $\pi^{-1}$M first, then multiply the complement by $\pi$. The same remark applies to all following cases when applying Lemma \ref{ortho in H^s} or Corollary \ref{normal basis and ortho} to $\pi$-modular lattices.} to see that the orthogonal complement of $e'$ in $M$ is just
	$$M(e')=\pi H^{m-1}\obot I$$
	where $I=\langle e''\rangle_{O_{\breve{F}}}$ with $h(e'',e'')=-h(e',e')$ and $e''-e'\in \pi M$. Now consider the element $e+Ve''=Ve'+Ve''\in VM$. It is isotropic by Lemma \ref{h(vx,vy)} as
	$$h(e'+e'',e'+e'')=h(e',e')+h(e'',e'')=0.$$
	Moreover $Ve'+Ve''\notin \pi M$ as we already have $Ve''-Ve'\in V\pi M\subset \pi M$ and $e=Ve'\notin \pi M$. Thus we may just take $e$ at the beginning to be isotropic.

 As $e\notin \pi M$, by Lemma \ref{ortho in H^s} again we see that the orthogonal complement of $e$ in $M$ contains $\pi H^{m-1}$. Choose a standard normal basis $\sigma_2,\tau_2,\dots,\sigma_{m},\tau_{m}$ for this $\pi H^{m-1}$ and extend it to a standard normal basis $\sigma_1,\tau_1,\sigma_2,\tau_2,\dots,\sigma_m,\tau_m$ of $M$. Then $e=a \sigma_1+ b \tau_1$. As $e\notin \pi M$ we may assume $a\in O_{\breve{F}}^{\times}$ and by normalizing $\sigma_1, \tau_1$, we assume $a=1$. We may replace $\sigma_1$ by $e$ and still get a standard normal basis of $M$, with $\sigma_1=e\in VM$.
\end{proof}

Using the basis provided in Lemma \ref{existence of good basis}, we have the following detailed description of the Dieudonn\'e module and Hodge filtration of $\bar{k}$-points on $\cN_n$. This Lemma will be used repeatedly when showing surjectivity of morphisms between RZ spaces on $\bar{k}$-points.

\begin{lemma}\label{describe VM}
    Let $z=(X_z,\iota_z,\lambda_z,\rho_z)\in\cN_n(\bar{k})$ be any point. Let $M$ be the Dieudonn\'e module of $X_z$ and $\Fil^1_z\subset\dD_z$ be the Hodge Filtration of $X_z$. Suppose we have a standard normal basis $\sigma_1,\tau_1,\dots,\sigma_m,\tau_m$ of $M$ (or $M^*$) such that $\sigma_1\in VM$ (resp.\ $\sigma_1\in VM^*$). Then
    \begin{itemize}
        \item [(1)] $VM$ (resp.\ $VM^*$) is the $O_{\breve{F}}$-span of $$
	\sigma_1,\pi \sigma_2,\dots, \pi \sigma_m, \pi_0 \tau_1, \pi \tau_2,\pi \tau_3,\dots, \pi \tau_m;$$
    \item [(2)] $\Fil^1_z$ is the $\bar{k}$-span of
    $$\sigma_1,\pi \sigma_1,\pi \sigma_2,\dots,\pi \sigma_m,\pi \tau_2,\pi \tau_3,\dots, \pi \tau_m;$$
    \item [(3)] $\Lie X_z$ (resp.\ $\Lie X^{\vee}_z$) is the $\bar{k}$-span of
    $$\sigma_2,\sigma_3,\dots,\sigma_m,\tau_1,\pi \tau_1,\tau_2,\tau_3,\dots,\tau_m.$$
    \end{itemize}
\end{lemma}
\begin{proof}
Write $n=2m$. The Dieudonn\'e module $M$ is an $O_{\breve{F}}$-lattice in $\mathbb{N}_n$ such that $M^*=\pi^{-1}M$ and $\pi_0 M\subset VM\subset M$, $VM\subset^1 VM+\pi M$. We first assume $\sigma_1,\tau_1,\dots,\sigma_m,\tau_m$ is a standard normal basis of $M$.

	As $\lambda_z$ is $\pi$-modular, there exists a unique morphism $\sigma_z\colon X_z\to X_z^{\vee}$ such that $\lambda_z=\sigma_z \iota_z(\pi)$. The morphism $\sigma_z$ is an isomorphism and symmetric in the sense that $\sigma_z^{\vee}=\sigma_z$. This symmetrization gives a perfect symmetric pairing on the Dieudonn\'e module $M$ (viewed as a rank $2n=4m$ $O_{\breve{F}_0}$-lattice)
	$$(-,-)\colon M\times M\lrarr O_{\breve{F}_0}$$
	and under the $O_{\breve{F}_0}$-basis
	$$\sigma_1,\dots,\sigma_m,\tau_1,\dots,\tau_m,\pi \sigma_1,\dots,\pi \sigma_m,\pi \tau_1,\dots,\pi \tau_m$$
	the matrix of the symmetric pairing is of the form
	\begin{equation}\label{symmetric matrix}
	    \begin{pmatrix}
	0&0&0&-I_m\\
	0&0&I_m&0\\
	0&I_m&0&0\\
	-I_m&0&0&0
	\end{pmatrix}.
	\end{equation}
	
	Now the condition $\pi_0 M\subset VM$ implies
	$$\pi_0 \sigma_1,\dots,\pi_0 \sigma_m, \pi_0 \tau_1,\dots,\pi_0 \tau_m\in VM$$
	and by our assumption $\sigma_1\in VM$.
	
	Moreover the point $z$ also gives a Hodge filtration of its Dieudonn\'e crystal as $\bar{k}$-vector spaces
	$$0\lrarr \operatorname{Fil}^1_z\lrarr \mathbb{D}_z\lrarr \operatorname{Lie}X_z\lrarr 0$$
	where $\mathbb{D}_z=M/\pi_0 M$ and $\operatorname{Fil}^1_z=VM/\pi_0 M$. The induced perfect pairing
	$$(-,-)\colon\mathbb{D}_z\times\mathbb{D}_z\lrarr\bar{k}$$
	makes $\operatorname{Fil}^1_z$ a maximal totally isotropic subspace. From this we see that, as we already have $\sigma_1\in VM$ and $(\sigma_1,\pi \tau_1)=-1$, we have $\pi \tau_1\notin VM$. Now the condition
	$$VM\subset^1 VM+\pi M$$
	would imply that $\pi \sigma_1,\pi \sigma_2,\dots, \pi \sigma_m, \pi \tau_2,\pi \tau_3,\dots, \pi \tau_m\in VM$. In particular we have
	\begin{equation}\label{V_n}
	V_n\coloneqq \langle \sigma_1,\pi \sigma_2,\dots, \pi \sigma_m, \pi \tau_2,\pi \tau_3,\dots, \pi \tau_m, \pi_0 \tau_1\rangle_{O_{\breve{F}}}\subset VM.
    \end{equation}
	Since the isocrystal $\dN_n$ is supersingular, the co-length of $VM$ in $M$ is $n$. As the co-length of $V_n$ in $M$ is already $n$, the inclusion \eqref{V_n} is an equality. Consequently
	$$\operatorname{Fil}^1_z=\langle \sigma_1,\pi \sigma_1,\pi \sigma_2,\dots,\pi \sigma_m, \pi \tau_2,\pi \tau_3,\dots, \pi \tau_m\rangle_{\bar{k}}$$
	and
	$$\Lie X_z=M/VM=\langle \sigma_2,\sigma_3,\dots,\sigma_m,\tau_1,\pi\tau_1,\tau_2,\tau_3,\dots,\tau_m\rangle_{\bar{k}}.$$

    Now suppose $\sigma_1,\tau_1,\dots,\sigma_m,\tau_m$ is a standard normal basis of $M^*=\pi^{-1}M$. Then $\sigma'_i\coloneqq \pi\sigma_i,\tau'_j\coloneqq\pi\tau_j$ is a standard normal basis of $M$. Thus we have
    \begin{equation*}
    \begin{split}
        VM^*=V(\pi^{-1}M)=\pi^{-1}VM&=\pi^{-1}\langle \sigma'_1,\pi\sigma'_2,\dots,\pi\sigma'_m,\pi_0\tau'_1,\pi\tau'_2,\pi\tau'_3,\dots,\pi\tau'_m\rangle_{O_{\breve{F}}}\\
        &=\langle \sigma_1,\pi\sigma_2,\dots,\pi\sigma_m,\pi_0\tau_1,\pi\tau_2,\pi\tau_3,\dots,\pi\tau_m\rangle_{O_{\breve{F}}}
    \end{split}
    \end{equation*}
    and
    $$\Lie X_z^{\vee}=M^*/VM^*=\langle \sigma_2,\sigma_3,\dots,\sigma_m,\tau_1,\pi\tau_1,\tau_2,\tau_3,\dots,\tau_m\rangle_{\bar{k}}.$$
\end{proof}

Next we prove a lemma, which shows that $\cZ_n(u)$ is formally smooth of relative dimension $n-2$ over $\Spf O_{\breve{F}}$.

\begin{lemma}\label{tangent space of special divisor}
	Let $z\in\mathcal{N}_n(\bar{k})$ be any point. If there is nonzero $x\in\mathbb{V}_n$ such that $z\in \mathcal{Z}_n(x)(\bar{k})$ and $z\notin \mathcal{Z}_n(\pi^{-1}x)(\bar{k})$ then the tangent space of $\mathcal{Z}_n(x)$ at $z$ is of dimension $n-2$.
\end{lemma}
\begin{proof}
    We shall first recall the computations in \cite{PR09}*{\S5.c}. To be precise, let $V$ be an $F$-vector space of dimension $n$ and let $h\colon V\times V\to F$ be a split $F/F_0$-hermitian form. Choose a basis $e_1,\dots,e_n$ of $V$ such that
	$$h(e_i,e_{n+1-j})=\delta_{i,j}$$
	and for $i=0,\dots,n-1$ set
	$$\Lambda_i\coloneqq \langle \pi^{-1}e_1,\dots,\pi^{-1}e_i,e_{i+1},\dots,e_n\rangle_{O_F}.$$
	Attached to the hermitian form $h$ we also have two $F_0$-bilinear form
	$$(x,y)\coloneqq \frac{1}{2}\operatorname{Tr}_{F/F_0}(h(x,y))\quad\text{ and }\quad\langle x,y\rangle\coloneqq \frac{1}{2}\operatorname{Tr}_{F/F_0}(\pi^{-1}h(x,y))$$
	so the form $(\text{ },\text{ })$ is symmetric and $\langle\text{ },\text{ }\rangle$ is alternating.
	
	Now $\Lambda_m$ is equipped with the perfect symmetric pairing $(\text{ },\text{ })$. Let $f_1=-\pi^{-1}e_1, \dots, f_m=-\pi^{-1}e_m$ and $f_{m+1}=e_{m+1}, \dots, f_n=e_n$. Consider the $O_{F_0}$-basis of $\Lambda_m$ given by
	$$f_1, \dots, f_n, -\pi f_1, \dots, -\pi f_m, \pi f_{m+1}, \dots, \pi f_n.$$
	Reorder the basis, set
	\begin{align*}
	    \Lambda'&\coloneqq\langle f_1, \pi f_1, \pi f_2, \pi f_3, \dots, \pi f_{n-1}\rangle_{O_{F_0}},\\
	\Lambda''&\coloneqq \langle f_n, \pi f_n, f_2, f_3, \dots, f_{n-1}\rangle_{O_{F_0}}.
	\end{align*}
    Then $\Lambda_m=\Lambda'\oplus\Lambda''$.
 
    Let $M$ be the Dieudonn\'e module of $z=(X_z,\iota_z,\lambda_z,\rho_z)$. By Lemma \ref{existence of good basis} and Lemma \ref{describe VM} we see that there is a standard normal basis $\sigma_1,\tau_1,\dots,\sigma_m,\tau_m$ of $M$ with $\sigma_1\in VM$ and symmetric pairing matrix \eqref{symmetric matrix}. Then we have an isomorphism of $O_{\breve{F}}$-lattices preserving perfect symmetric pairing
    \begin{equation}
        \Lambda_m\otimes_{O_{F_0}}O_{\breve{F}_0}\simeq M
    \end{equation}
    sending $f_i\otimes 1$ to $\sigma_i$ for $1\leq i\leq m$ and $f_{n+1-j}\otimes 1$ to $-\tau_j$ for $1\leq j\leq m$. Under this isomorphism, the Hodge filtration
    \begin{equation*}
        \Fil^1_z=\langle \sigma_1,\pi\sigma_1,\pi\sigma_2,\dots,\pi\sigma_m,\pi\tau_2,\pi\tau_3,\dots,\pi\tau_m\rangle_{\bar{k}}\subset\dD_z
    \end{equation*}
    is identified with
	$$\cF_1\coloneqq \Lambda'\otimes_{O_{F_0}}\bar{k}=\langle f_1, \pi f_1, \pi f_2, \pi f_3, \dots, \pi f_{n-1}\rangle_{\bar{k}}\subset \Lambda_m\otimes_{O_{F_0}}\bar{k}.$$
	
	Now we want to look at the tangent space of $\mathcal{N}_n$ at $z$, which can be identified with $\mathcal{N}_{n,z}(\bar{k}[\varepsilon])$ where $\bar{k}[\varepsilon]=\bar{k}[x]/(x^2)$.
	
	Any lift $z'$ of $z$ to $\bar{k}[\varepsilon]$ gives rise to a Hodge filtration $\cF=\operatorname{Fil}^1_{z'}\subset\Lambda_m\otimes_{O_{F_0}}\bar{k}[\varepsilon]$ lifting $\cF_1$. We may find a linear map $X\colon \Lambda'_{\bar{k}[\varepsilon]}\to\Lambda''_{\bar{k}[\varepsilon]}$ such that $\cF=\{v+X\cdot v \res v\in \Lambda'_{\bar{k}[\varepsilon]}\}$. As we want $\cF$ to be a lift of $\cF_1$, $X\in\varepsilon M_n(\bar{k})$.
	
	The matrix of the symmetric form on $\Lambda_m$ under the basis
	$$f_1, \pi f_1, \pi f_2, \pi f_3, \dots, \pi f_{n-1}, f_n, \pi f_n, f_2, f_3, \dots, f_{n-1}$$
	is of the form
	$$\begin{pmatrix}
	0 & S\\
	S^{t} & 0
	\end{pmatrix}$$
	where $S$ is the skew matrix of size $n$,
	$$S=\begin{pmatrix}
	J_2^t & 0\\
	0 & J_{n-2}
	\end{pmatrix},
	J_{2s}=
	\begin{pmatrix}
	0 & -H_s\\
	H_s & 0
	\end{pmatrix}$$
	where $H_s$ is the unit antidiagonal matrix of size $s$.
	
	We may write the map $X$ as a matrix
	$$X=\begin{pmatrix}
	T & B \\
	C & Y
	\end{pmatrix}$$
	where each block has size $2\times 2$ for $T$, $2\times (n-2)$ for $B$, $(n-2)\times 2$ for $C$ and $(n-2)\times (n-2)$ for $Y$. Then the condition that $\cF$ is isotropic translates into
	$$SX^t=XS$$
	which is just $J_2\cdot T^t=T\cdot J_2$, $J_{n-2}\cdot Y^t=Y\cdot J_{n-2}$ and $C\cdot J_2^t=J_{n-2}\cdot B^t$. The first condition implies that $T$ is diagonal, $T=\operatorname{diag}(x,x)$, the last condition shows that $C$ is determined by $B$ via $C=J_{n-2}B^t J_2$.
	
	The condition that $\cF$ is $\pi$-stable translates to\footnote{The computation from \cite{PR09}*{\S5.c} says $B_1=B_2\cdot Y$, which is incorrect. The correct computation shows that it should be $B_2=-B_1\cdot Y$, as can be shown from the description below.}
	$$Y^2=\pi_0\cdot I_{n-2}\text{ , }B_2=-B_1\cdot Y.$$

	As we have signature $(1,n-1)$ in our setting, and $\pi=0$ in $\bar{k}[\varepsilon]$, the Wedge condition implies
	$$Y=\sqrt{\pi_0}\cdot I_{n-2}=0.$$
	So the matrix $X$ looks like
	$$X=\begin{pmatrix}
	x & 0 & b_1 & b_2 & \dots & b_{n-2} \\
	0 & x & 0 & 0 & \dots & 0\\
	0 & b_{n-2} & 0 & 0 &\dots &0\\
	&\dots&\dots&&\dots\\
	0 & b_{(n-2)/2} &0&0&\dots&0\\
	0& -b_{(n-2)/2-1}&0&0&\dots &0\\
	&\dots&\dots&&\dots\\
	0& -b_1 &0&0&\dots&0
	\end{pmatrix}.$$
	Thus $\cF$ is given by $\bar{k}[\varepsilon]$-span of
	$$\begin{aligned}
	&t_1=f_1+xf_n+b_1 f_2+\dots + b_{n-2}f_{n-1},\\
	&t_2=\pi f_1+x\pi f_n,\\
	&t_3=\pi f_2+b_{n-2}\pi f_n,\\
	&\quad\quad\quad\dots\\
	&t_n=\pi f_{n-1}-b_1\pi f_n
	\end{aligned}$$
	with all $x, b_1, \dots, b_{n-2}\in\varepsilon\cdot \bar{k}$. One checks easily that
	$$\pi t_1=t_2+b_1 t_3+ b_2 t_4+\dots +b_{n-2}t_n$$
	and $\pi t_i=0$ for $i=2,\dots, n$ (so indeed $\cF$ is $\pi$-stable). In particular this implies that the tangent space of $\mathcal{N}_n$ at $z$ is of $\bar{k}$-dimension $n-1$.
	
	Now suppose $x\in \mathbb{V}_n$ such that $z\in\mathcal{Z}_n(x)(\bar{k})$ but $z\notin\mathcal{Z}_n(\pi^{-1}x)(\bar{k})$. This means we have a morphism of $p$-divisible groups
	$$x\colon\ubE\lrarr X_z$$
	where $X_z$ is the $p$-divisible group corresponding to the $\bar{k}$-point $z$. Write the Hodge filtration of $\ubE$ as
	$$0\lrarr \operatorname{Fil}^1\lrarr \mathbb{D}(\ubE)\lrarr \operatorname{Lie}\ubE\lrarr0$$
	where $\operatorname{Fil}^1$ is a $1$-dimension $\bar{k}$-vector space. Choose a $\bar{k}$-basis $f$ of $\Fil^1$ as in Lemma \ref{D module of ubE} and consider its image in $\mathbb{D}(X_z)\simeq\Lambda_{m, \bar{k}}$ under the map $\mathbb{D}(x)$. As we assume $z\in\mathcal{Z}_n(x)(\bar{k})$, the image lies in $\cF_1$. The map $\mathbb{D}(x)$ also commutes with $\pi$-action and $\pi\cdot f=-\pi_0 e=0$. So we may write its image as
	$$a_1 \pi f_1 + a_2 \pi f_2 +\dots +a_{n-1}\pi f_{n-1}\in\cF_1.$$
	Similarly we can write the image of $f$ in $\mathbb{D}(X_z)_{\bar{k}[\varepsilon]}$ under the map $\mathbb{D}(x)_{\bar{k}[\varepsilon]}$ as
	$$f'\coloneqq a_1 \pi f_1+\dots+a_{n-1}\pi f_{n-1}+\varepsilon(\mu_1\pi f_1+\dots +\mu_n \pi f_n).$$
	
	Then a lift $z'$ of $z$ still lies in $\mathcal{Z}_n(x)(\bar{k}[\varepsilon])$ is equivalent to $f'\in\cF$. We can first use $\varepsilon\cdot t_i=\varepsilon \pi f_{i-1}$ for $i=2,\dots, n$ to get rid of $\varepsilon \mu_i \pi f_i$ for $i=1,\dots, n-1$, then the only obstruction for $f'\in\cF$ is to require
	$$a_1 x+a_2 b_{n-2}+\dots +a_{\frac{n+2}{2}}b_{\frac{n-2}{2}}-a_{\frac{n+2}{2}+1} b_{\frac{n-2}{2}-1}-\dots-a_{n-1}b_1=\mu_n.$$
	In particular if not all the $a_i$ are zero, then the tangent space of $\mathcal{Z}_n(x)$ at $z$ is of dimension $n-1-1=n-2$.
	
	Suppose all the $a_i$ are zero, which means the image of $f$ in $\cF_1$ is already zero. Follow the notation in Lemma \ref{D module of ubE} we let $\overline{\dM}_0$ be the Dieudonn\'e module of $\ubE$. The morphism $x\colon\ubE\to X_z$ gives rise to a morphism of Dieudonn\'e modules $x\colon \overline{\dM}_0=O_{\breve{F}}e\to M$. The vanishing of all the $a_i$ is equivalent to say that
	$$x V\overline{\dM}_0\subset VM\cap \pi_0 M=\pi_0 M.$$
	As $\overline{\dM}_0=\langle e,f\rangle_{O_{\breve{F}_0}}$ and explicitly $Ve=f$ and $Vf=\pi_0 e$, we have
	$$xf\in\pi_0 M\text{ and }x\pi_0 e\in \pi_0 M.$$
	But also we have $\pi e=-f$ and so from $xf\in \pi_0 M$ we get
	$$\pi^{-1}xf\in\pi M\subset M\text{ and }x\pi e\in \pi_0 M\text{ , }\pi^{-1}xe\in M.$$
	In particular we show that $\pi^{-1}x$ is also a morphism of Dieudonn\'e modules $\overline{\dM}_0\to M$ hence $z$ is in $\mathcal{Z}_n(\pi^{-1}x)(\bar{k})$, contradiction! Thus we can not have all the $a_i$ to be zero. The Lemma is proved.
\end{proof}

\begin{corollary}\label{N_n(u) is smooth}
	If $x\in\mathbb{V}_n$ has hermitian norm $h(x,x)\in \pi^2 O_{F_0}^{\times}$ then $\cZ_n(x)$ is formally smooth of relative dimension $n-2$ over $\Spf O_{\breve{F}}$.
\end{corollary}
\begin{proof}
    If $h(x,x)\in\pi^2 O_{F_0}^{\times}$ then $h(\pi^{-1}x,\pi^{-1}x)\in O_{F_0}^{\times}$ and $\cZ_n(\pi^{-1}x)$ is empty. By Lemma \ref{tangent space of special divisor} we see that the tangent space of $\cZ_n(x)$ at every geometric point has dimension $n-2$. As $\cZ_n(x)$ is a relative divisor in $\cN_n$ (\cite{LL22}*{Lemma 2.40}) we conclude that $\cZ_n(x)$ is formally smooth of relative dimension $n-2$ over $\Spf O_{\breve{F}}$.
\end{proof}

\subsection{Proof of Theorem \ref{identify lower RZ space with special divisor}}\label{proof of theorem 5.5}
Now let us prove Theorem \ref{identify lower RZ space with special divisor}. We shall consider $\delta^+$. All arguments apply to $\delta^-$ as well. The structure of the proof is as follows$\colon$
\begin{itemize}
	\item [(1)] We first show $\delta^+$ factors through $\cZ_n(u)^+$.
	\item [(2)] We then show that $\delta^+$ is surjective.
	\item [(3)] We finally show that $\delta^+$ is formally \'etale, and deduce that it is an isomorphism.
\end{itemize}

\begin{proposition}\label{factor through}
	The morphism $\delta^+$ factors through $\mathcal{Z}_n(u)^+$.
\end{proposition}

\begin{proof}
	For $S\in\operatorname{Nilp}_{O_{\breve{F}}}$ and $(X',\iota',\lambda',\rho')\in\mathcal{N}_{n-1}(S)$, the morphism $\delta^+$ provides a quadruple $(X,\iota,\lambda,\rho)\in\mathcal{N}_n^+(S)$ together with an $O_F$-isogeny of degree $q$ lifting $\phi_0$, which we denote by $\phi\colon X\to X'\times\ucE_S$.
	
	Since $\operatorname{ker}(\phi)$ is $\pi$-power torsion and of rank $q$, it is killed by $\pi$. Hence there exists a unique (necessarily $O_F$-linear) isogeny $\phi'\colon X'\times\ucE_S\to X$ such that the composite $\phi'\circ\phi\colon X\to X$ is $\iota(\pi)$.
	
	Let $\tilde{u}$ be the composite $$\ucE_S\stackrel{(0,\operatorname{id})}{\lrarr}X'\times\ucE_S\stackrel{\phi'}{\lrarr}X$$
	which is clearly a morphism. Claim that $\tilde{u}$ lifts $\rho^{-1}\circ u$ so that $(X,\iota,\lambda,\rho)\in\mathcal{Z}_n(u)^+$.
	
	Reduce $\tilde{u}$ from $S$ to its special fibre $\ovS$ we get
	$$\ubE_{\ovS}\stackrel{(0,\operatorname{id})}{\lrarr}X'_{\ovS}\times\ubE_{\ovS}\stackrel{\bar{\phi'}}{\lrarr}X_{\ovS}$$
	and by uniqueness of $\phi'$ we have a commutative diagram
	$$\xymatrix{
		X'_{\ovS} \times \ubE_{\ovS}  
		\ar[d]^-{\rho'\times\operatorname{id}}
		\ar[r]^-{\bar{\phi'}} & X_{\ovS} \ar[d]^-{\rho}\\
		\mathbb{X}_{n-1,\ovS}\times \ubE_{\ovS} \ar[r]^-{\phi_0'} & \mathbb{X}_{n,\ovS}}$$
	where $\phi_0'$ is given by
	$$\phi_0'\coloneqq \iota_{\mathbb{X}_{n-2}}(\pi)\times\begin{pmatrix}
	\frac{\iota_{\ubE}(\pi)}{2}&\frac{\iota_{\ubE}(\pi)}{2}\\1&-1
	\end{pmatrix}$$
    so that $\phi_0'\circ\phi_0=\iota_{\dX_{n,\ovS}}(\pi)$.
    
	Hence it suffices to show that $$\ubE_{\ovS}\stackrel{(0,\operatorname{id})}{\lrarr}X'_{\ovS}\times\ubE_{\ovS}\stackrel{\rho'\times\operatorname{id}}{\lrarr}\mathbb{X}_{n-1,\ovS}\times \ubE_{\ovS} \stackrel{\phi_0'}{\lrarr}\mathbb{X}_{n,\ovS}$$
	is just $u$, which is an easy computation$\colon$
	$$\left(\iota_{\mathbb{X}_{n-2}}(\pi)\times\begin{pmatrix}
	\frac{\iota_{\ubE}(\pi)}{2}&\frac{\iota_{\ubE}(\pi)}{2}\\1&-1
	\end{pmatrix}\right)\cdot\begin{pmatrix}0_{n-2}\\0\\1\end{pmatrix}=\begin{pmatrix}0_{n-2}\\\frac{\iota_{\ubE}(\pi)}{2}\\-1\end{pmatrix}.$$
\end{proof}

Now we have a morphism still denoted by $\delta^+\colon\mathcal{N}_{n-1}\to\mathcal{Z}_n(u)^+$, which is universally injective and formally unramified as it is a closed embedding. To show it is an isomorphism we need to show it is also surjective and formally \'etale.

\begin{proposition}\label{surj of delta}
	The morphism $\delta^+\colon\mathcal{N}_{n-1}\to\mathcal{Z}_n(u)^+$ is surjective.
\end{proposition}

\begin{proof}
	We will just show that $\delta^+$ is a surjection on $\bar{k}$-points and the argument for points in an arbitrary algebraically closed field is the same.
 
    The isogeny $\phi_0$ given in \eqref{phi0} allows us to identify $\mathbb{N}_n$ with $\mathbb{N}_{n-1}\obot\overline{\mathbb{N}}_1$. Follow Lemma \ref{D module of ubE} we let $O_{\breve{F}}e$ be the Dieudonn\'e module of $\ubE$. Then $e$ is also an $\breve{F}$-basis of $\overline{\dN}_1$. Using Dieudonn\'e theory, the $\bar{k}$-points of $\mathcal{N}_{n-1}$ and $\mathcal{N}_n^+$ are related as follows$\colon$
	let $(X',\iota',\lambda',\rho')\in\mathcal{N}_{n-1}(\bar{k})$ with Dieudonn\'e module $\dL\subset \mathbb{N}_{n-1}$ and suppose its image in $\mathcal{N}_n^+(\bar{k})$ has Dieudonn\'e module $\dM$, then
	\begin{equation}\label{L sub M}
	\pi \dL^*\obot \pi O_{\breve{F}}e\subset^1 \dM\subset^1 \dL\obot O_{\breve{F}}e.
	\end{equation}
	
	Suppose there is a point in $\mathcal{N}_n^+(\bar{k})$ with Dieudonn\'e module $\dM$. Then the point is in $\mathcal{Z}_n^+(u)$ is equivalent to that there is a morphism of Diendonn\'e modules
	$$u\colon O_{\breve{F}}e\lrarr \dM$$ such that when viewing $\dM$ in $\mathbb{N}_{n-1}\obot\overline{\mathbb{N}}_1$ via $\phi_0$, the morphism $u$ sends $e$ to $\pi e$ (as $\phi_0 u\colon\ubE\to\dX_n=\dX_{n-1}\times\ubE$ is just $(0_{n-1},\iota_{\ubE}(\pi))$). This turns out to be a simple condition that $\pi e\in \dM$.
	
	It is enough to show that given $\pi$-modular lattice $\dM\subset \mathbb{N}_n$ satisfying condition \eqref{VM} which contains $\pi e$, we can recover the almost $\pi$-modular lattice $\dL\subset \mathbb{N}_{n-1}$ satisfying condition \eqref{VModd} and \eqref{L sub M}. We set
	\begin{equation}\label{define L out of M}
	    (\pi e)^{\perp}\coloneqq\{a\in \dM \res h(a,\pi e)=0\}\text{ and }\dL\coloneqq (\pi^{-1}(\pi e)^{\perp})^{*}=\pi ((\pi e)^{\perp})^{*}.
	\end{equation}
	
	\textbf{Check condition \eqref{L sub M}}$\colon$ As $\dM$ is $\pi$-modular, $\dM\simeq \pi H^{n/2}$ and under a standard normal basis $\dM$ has moment matrix
	$$\begin{pmatrix}
	0 & -\pi\\ \pi & 0
	\end{pmatrix}^{\obot n/2}.$$
	By \eqref{relation of herm forms}, $\pi e\in \dM$ is of hermitian norm $\delta\cdot h(u,u)=\delta\pi^2$. If we divide everything by $\pi$ and apply Lemma \ref{ortho in H^s} and Corollary \ref{normal basis and ortho}, we see that $(\pi e)^{\perp}$ will have the moment matrix under a normal basis
	$$\begin{pmatrix}
	0 & -\pi\\ \pi & 0
	\end{pmatrix}^{\obot n/2 - 1}\obot (-\delta\pi^2)$$
	and $(\pi e)^{\perp}\obot \pi O_{\breve{F}}e\subset^1 \dM$. Correspondingly $\dL$ will have the moment matrix under a normal basis
	$$\begin{pmatrix}
	0 & -\pi\\ \pi & 0
	\end{pmatrix}^{\obot n/2 - 1}\obot (\delta).$$
	In particular $\dL$ is almost $\pi$-modular and $\pi \dL^{*}\obot \pi O_{\breve{F}}e\subset^1 \dM$. Then from the fact that $\dM^{*}=\pi^{-1}\dM$, $\dL$ is almost $\pi$-modular and $(O_{\breve{F}}e)^*=O_{\breve{F}}e$, we get $\dM\subset^1 \dL\obot O_{\breve{F}}e$.
	
	\textbf{Check condition \eqref{VModd}}$\colon$ We need to show $\dL$ satisfies condition \eqref{VModd}. We first claim that $F \dL^{*}\subset \dL^{*}$. This is equivalent to $F (\pi e)^{\perp}\subset (\pi e)^{\perp}$. By Lemma \ref{D module of ubE} we have $Ve=Fe=-\pi e$. Take $z\in (\pi e)^{\perp}$, we have
    \begin{equation*}
        h(VFz,V\pi e)=h(\pi_0 z,\pi Ve)=\pi_0 h(z,\pi(-\pi e))=\pi_0\pi h(z,\pi e)=0.
    \end{equation*}
    Thus by Lemma \ref{h(vx,vy)} we have $h(Fz,\pi e)=0$ and $F (\pi e)^{\perp}\subset (\pi e)^{\perp}$. The claim is proved. From $F\dL^*\subset\dL^*$ we get $V^{-1}\dL^*\subset\pi_0^{-1}\dL^*$.
 
    To show $V\dL\subset \dL$, we claim that
	$$h(x,y)\in O_{\breve{F}}\text{ for any }x\in \dL^{*}, y\in V\dL.$$
	We let $y=V y'\in V\dL$ and $x\in \dL^{*}$. Then $V^{-1}x\in V^{-1}\dL^*\subset\pi_0^{-1}\dL^*$. Now
    \begin{equation*}
        h(V^{-1}x,V^{-1}y)=h(V^{-1}x,y')\in h(\pi_0^{-1}\dL^*,\dL)\subset\pi_0^{-1}O_{\breve{F}}
    \end{equation*}
    and by Lemma \ref{h(vx,vy)} we have $h(x,y)\in O_{\breve{F}}$. The claim is proved and $V\dL\subset\dL$.
	
	Similarly to show $\pi_0 \dL\subset V\dL$, it is enough to show $F \dL\subset \dL$. and essentially to show $V \dL^{*}\subset \dL^{*}$, which can be proved the same way using $Fe =-\pi e$ and Lemma \ref{h(vx,vy)}.
	
	The only thing left in condition \eqref{VModd} is to show $V\dL\subset^{\leq 1} \pi \dL+V\dL$. Write $n=2m$ and pick a standard normal basis $e_1,f_1,\dots,e_m,f_m$ of $\dM$ such that $e_1\in V\dM$. This is possible by Lemma \ref{existence of good basis} and moreover by Lemma \ref{describe VM} we have
	$$V\dM=\langle e_1, \pi_0 f_1, \pi e_2,\pi f_2,\dots, \pi e_m,\pi f_m \rangle_{O_{\breve{F}}}.$$
	Consider $\pi e\in \dM$ under this basis and write
	$$\pi e=a_1 e_1 +b_1 f_1 +\dots + a_m e_m+b_m f_m.$$
	If $\pi e\in V\dM$ then $F \pi e\in FV \dM=\pi_0 \dM$, but $F \pi e=\pi Fe=\pi (-\pi e)=-\pi_0 e$ so $e\in \dM$ contradiction! Thus $\pi e\notin V\dM$. Also $\pi\cdot \pi e=\pi (-V e)=-V(\pi e)\in V\dM$. Hence we must have $b_1\in (\pi)$ and either one of $a_2,b_2,\dots,a_m,b_m$ is in $O_{\breve{F}}^{\times}$, or $a_2,b_2,\dots,a_m,b_m\in (\pi)$ and $b_1\notin (\pi_0)$.
	
	We then have two cases. We will see they correspond to the co-length of $V\dL$ in $\pi\dL+V\dL$ is $0$ or $1$, by finding a standard normal basis of $\dM$ and describing $\dL, V\dL$ in terms of this basis.
	
	\textbf{Case $1$}$\colon$ Suppose one of $a_2, b_2,\dots, a_m,b_m$ is in $O^{\times}_{\breve{F}}$. We may assume $a_m=1$. Then by the proof of Corollary \ref{normal basis and ortho} we can find a standard normal basis $\sigma_1,\tau_1,\dots,\sigma_m,\tau_m$ of $\dM$ such that $\pi e=\sigma_m+\frac{\pi\delta}{2}\tau_m$ and 
    $$\langle \pi e\rangle^{\perp}=\langle \sigma_1,\tau_1,\dots,\sigma_{m-1},\tau_{m-1},\sigma_m-\frac{\pi\delta}{2}\tau_m \rangle_{O_{\breve{F}}}.$$
    Moreover as $b_1\in (\pi)$ and $\pi f_m\in V\dM$, by construction \eqref{sigma and tau} the element $\sigma_1=e_1+\bar{b}_1f_m$ is still in $V\dM$. Thus by Lemma \ref{describe VM} again we have
	$$V\dM=\langle \sigma_1,\pi_0 \tau_1,\pi \sigma_2,\pi\tau_2,\dots,\pi\sigma_m,\pi\tau_m\rangle_{O_{\breve{F}}}.$$
	As $V(\pi e)=-\pi\pi e$ we have $V(\sigma_m+\frac{\pi\delta}{2}\tau_m)=-\pi(\sigma_m+\frac{\pi\delta}{2}\tau_m)$. As $V\dL^*\subset\dL^*$, we have $V\langle \pi e\rangle^{\perp}\subset \langle \pi e\rangle^{\perp}$. We may then write
	$$V(\sigma_m-\frac{\pi\delta}{2}\tau_m)=\alpha\cdot(\sigma_m-\frac{\pi\delta}{2}\tau_m)+\alpha_1 \sigma_1+\beta_1\tau_1+\dots+\alpha_{m-1}\sigma_{m-1}+\beta_{m-1}\tau_{m-1}\in V\dM.$$
	Hence $\alpha\in (\pi)$ as $\sigma_m\notin V\dM$ and $\pi\sigma_m\in V\dM$. Thus
	\begin{equation}\label{case 1 sigma m}
	    2V\sigma_m=V(\sigma_m-\frac{\pi\delta}{2}\tau_m)+V(\sigma_m+\frac{\pi\delta}{2}\tau_m)\in(\pi_0)\tau_m+\langle \sigma_1,\pi_0\tau_1,\dots,\pi\sigma_{m-1},\pi\tau_{m-1},\pi\sigma_m\rangle_{O_{\breve{F}}}.
	\end{equation}
	
	Now let $\gamma$ be one of $\sigma_1,\tau_1,\dots,\sigma_{m-1},\tau_{m-1}$. From $h(\gamma,\sigma_m+\frac{\pi\delta}{2}\tau_m)=0$ and Lemma \ref{h(vx,vy)} we get
	$$h(V\gamma,V(\sigma_m+\frac{\pi\delta}{2}\tau_m))=0.$$
	If we write $V\gamma\in a\pi\sigma_m+b\pi\tau_m+\langle \sigma_1,\pi_0\tau_1,\dots,\pi\sigma_{m-1},\pi\tau_{m-1}\rangle_{O_{\breve{F}}}$ then
	$$h(a\pi\sigma_m+b\pi\tau_m,-\pi(\sigma_m+\frac{\pi\delta}{2}\tau_m))=\frac{a\pi_0^2\delta}{2}+b\pi\pi_0=0.$$
	Hence $b\in(\pi)$ and 
    \begin{equation}\label{case 1 others}
        V\gamma\in(\pi_0)\tau_m+\langle \sigma_1,\pi_0\tau_1,\dots,\pi\sigma_{m-1},\pi\tau_{m-1},\pi\sigma_m\rangle_{O_{\breve{F}}}.
    \end{equation}
	
	Consider $N= \pi O_{\breve{F}} e \obot\langle \pi e\rangle^{\perp}\subset^1 \dM$. Then $N=\langle \sigma_1,\tau_1,\dots,\sigma_{m-1},\tau_{m-1},\sigma_m,\pi \tau_m\rangle_{O_{\breve{F}}}$ and $\pi N^{*}=\dL\obot O_{\breve{F}}e$. As $V\tau_m\in V\dM$, we have
	$$V(\pi\tau_m)\in(\pi_0)\tau_m+\langle \sigma_1,\pi_0\tau_1,\dots,\pi\sigma_{m-1},\pi\tau_{m-1},\pi\sigma_m\rangle_{O_{\breve{F}}}.$$
	Combine \eqref{case 1 sigma m} and \eqref{case 1 others} we have shown that
	$$VN\subset \langle \sigma_1,\pi_0\tau_1,\pi\sigma_2,\pi\tau_2,\dots,\pi\sigma_{m-1},\pi\tau_{m-1},\pi\sigma_m, \pi_0\tau_m\rangle_{O_{\breve{F}}}\subset^1 V\dM.$$
	As $VN\subset^1 V\dM$, we conclude that
	\begin{equation}\label{VN}
	    VN=\langle \sigma_1,\pi_0\tau_1,\pi\sigma_2,\pi\tau_2,\dots,\pi\sigma_{m-1},\pi\tau_{m-1},\pi\sigma_m, \pi_0\tau_m\rangle_{O_{\breve{F}}}.
	\end{equation}

    By construction \eqref{define L out of M}, we have
    \begin{equation}\label{L in case 1}
        \dL=\pi(\langle \pi e\rangle^{\perp})^*=\langle \sigma_1,\tau_1,\dots,\sigma_{m-1},\tau_{m-1},\pi^{-1}\sigma_m-\frac{\delta}{2}\tau_m\rangle_{O_{\breve{F}}}.
    \end{equation}
    We have shown that $V\dL\subset\dL$ and $V(\dL\obot O_{\breve{F}}e)=V(\pi N^*)=\pi V(N^*)=\pi\pi_0(VN)^*$ by Corollary \ref{V and dual}. From \eqref{VN} we get
    \begin{equation}\label{VL in case 1}
        \begin{split}
            V\dL\obot \pi O_{\breve{F}}e&=\langle \sigma_1,\pi_0\tau_1,\pi\sigma_2,\pi\tau_2,\dots,\pi\sigma_{m-1},\pi\tau_{m-1},\sigma_m,\pi\tau_m\rangle_{O_{\breve{F}}},\\
            V\dL&=\langle \sigma_1,\pi_0\tau_1,\pi\sigma_2,\pi\tau_2,\dots,\pi\sigma_{m-1},\pi\tau_{m-1},\sigma_m-\frac{\pi\delta}{2}\tau_m\rangle_{O_{\breve{F}}}.
        \end{split}
    \end{equation}
    It is then clear that $V\dL\subset^1\pi \dL+V\dL$.  
	
	\textbf{Case $2$}$\colon$ Suppose $b_1,a_2,b_2,\dots,a_m,b_m\in(\pi)$. As 
    $$h(\pi e,\pi e)=\delta\pi^2=\sum_{i=1}^{m}(-a_i\bar{b}_i+\bar{a}_ib_i)\pi\in \pi^2 O^{\times}_{\breve{F}}$$
    we must have $a_1\in O_{\breve{F}}^{\times}$ and we may assume $a_1=1$ after normalizing $e_1$ and $f_1$ by a unit. Again by Corollary \ref{normal basis and ortho} we can find a standard normal basis $\sigma_1,\tau_1,\dots,\sigma_m,\tau_m$ of $\dM$ such that $\pi e=\sigma_1+\frac{\pi\delta}{2}\tau_1$ and 
    $$\langle \pi e\rangle^{\perp}=\langle \sigma_1-\frac{\pi\delta}{2}\tau_1,\sigma_2,\tau_2,\dots,\sigma_m,\tau_m\rangle_{O_{\breve{F}}}.$$
    By construction \eqref{phi'} and \eqref{sigma 1}, $\sigma_1=e_1+(b_1-\frac{\pi\delta}{2})f_1+a_2 e_2+b_2 f_2+\dots+a_m e_m+b_m f_m$. As $b_1,a_2,b_2,\dots,a_m,b_m\in(\pi)$, we have $a_2 e_2,b_2f_2,\dots,a_m e_m, b_m f_m\in V\dM$. Note that
	$$\delta\pi^2=2\pi b_1-\sum_{i=2}^{m}(a_i\bar{b}_i-\bar{a}_ib_i)\pi$$
	where $\sum\limits_{i=2}^{m}(a_i\bar{b}_i-\bar{a}_ib_i)\pi\in \pi^4O_{\breve{F}}$. Hence
	$b_1-\frac{\pi\delta}{2}\in \pi^3O_{\breve{F}}$
	and $(b_1-\frac{\pi\delta}{2})f_1\in (\pi_0)f_1\in V\dM$. Thus $\sigma_1\in V\dM$ and again by Lemma \ref{describe VM} we have 
	$$V\dM=\langle \sigma_1,\pi_0 \tau_1,\pi \sigma_2,\pi\tau_2,\dots,\pi\sigma_m,\pi\tau_m\rangle_{O_{\breve{F}}}.$$
	Similarly as in Case $1$, using $V(\sigma_1+\frac{\pi\delta}{2}\tau_1)=-\pi(\sigma_1+\frac{\pi\delta}{2}\tau_1)$ we can show that if $\gamma$ is one of $\sigma_1,\pi\tau_1,\sigma_2,\tau_2,\dots,\sigma_m,\tau_m$ then
	$$V\gamma\in\langle \pi\sigma_1,\pi_0\tau_1,\pi\sigma_2,\pi\tau_2,\dots,\pi\sigma_m,\pi\tau_m\rangle_{O_{\breve{F}}}.$$
	Let $N=O_{\breve{F}}e\obot\langle \pi e\rangle^{\perp}\subset^1 \dM$, then $N=\langle \sigma_1,\pi\tau_1,\sigma_2,\tau_2,\dots,\sigma_m,\tau_m\rangle_{O_{\breve{F}}}$. We have shown that
	$$VN\subset \langle \pi\sigma_1,\pi_0\tau_1,\pi\sigma_2,\pi\tau_2,\dots,\pi\sigma_m,\pi\tau_m\rangle_{O_{\breve{F}}}\subset^1 V\dM$$
	and $VN\subset^1 V\dM$. Hence 
    \begin{equation}\label{case 2 VN}
        VN=\langle \pi\sigma_1,\pi_0\tau_1,\pi\sigma_2,\pi\tau_2,\dots,\pi\sigma_m,\pi\tau_m\rangle_{O_{\breve{F}}}=\pi N.
    \end{equation}

    By construction \eqref{define L out of M} we have
    \begin{equation}\label{L in case 2}
        \dL=\pi(\langle \pi e\rangle^{\perp})^*=\langle \pi^{-1}\sigma_1-\frac{\delta}{2}\tau_1,\sigma_2,\tau_2,\dots,\sigma_m,\tau_m\rangle_{O_{\breve{F}}}.
    \end{equation}
    From \eqref{case 2 VN} we get
    \begin{equation}\label{VL in case 2}
        V\dL=\langle \sigma_1-\frac{\pi\delta}{2}\tau_1,\pi\sigma_2,\pi\tau_2,\dots,\pi\sigma_m,\pi\tau_m\rangle_{O_{\breve{F}}}.
    \end{equation}
    It is clear that $V\dL=\pi\dL$ and $V\dL\subset^0 V\dL+\pi \dL$.
	
	Combine the two cases we see we always have $V\dL\subset^{\leq 1}V\dL+\pi \dL$ and thus we have shown that $\delta^+$ is surjective.
\end{proof}

The last step is to show that $\delta^{+}$ is formally \'etale.
\begin{proof}[Proof of Theorem \ref{identify lower RZ space with special divisor}]
	By Corollary \ref{N_n(u) is smooth}, we see that $\mathcal{Z}_n(u)^+$ is formally smooth of relative dimension $n-2$ over $\Spf O_{\breve{F}}$. For simplicity we shall write $\cN$ for $\mathcal{N}_{n-1}$ and $\cZ$ for $\mathcal{Z}_n(u)^+$. Also we write $\mathcal{O}$ for $\Spf O_{\breve{F}}$. Then we have a commutative diagram
	$$\xymatrix{
		\mathcal{N}\ar[rr]^-{\delta^+}\ar[rd]_-{f}&&\mathcal{Z}\ar[ld]^-{g}\\
		&\mathcal{O}=\operatorname{Spf}(O_{\breve{F}})
	}$$
	where both $f$ and $g$ are morphisms between locally Noetherian formal schemes, which are formally locally of finite type, separated, formally smooth of relative dimension $n-2$. Then locally both $f$ and $g$ are pseudo-finite in the sense of \cite{TLR05}*{Definition 1.3.1}. As $\delta^+$ is a closed embedding, we have a short exact sequence of coherent $\cO_{\cN}$-modules (\cite{TLR05}*{Proposition 2.2.8, Proposition 2.3.4})
	$$\cC_{\cN/\cZ}\lrarr(\delta^{+})^{*}\widehat{\Omega}^1_{\cZ/\cO}\lrarr\widehat{\Omega}^1_{\cN/\cO}\lrarr 0$$
	where both $(\delta^+)^*\widehat{\Omega}^1_{\cZ/\cO}$ and $\widehat{\Omega}^1_{\cN/\cO}$ are locally free $\cO_{\cN}$-modules of the same rank as $f$ and $g$ are smooth of same relative dimension (\cite{TLR05}*{Proposition 2.5.5}). Then the surjection
	$$(\delta^{+})^{*}\widehat{\Omega}^1_{\cZ/\cO}\lrarr\widehat{\Omega}^1_{\cN/\cO}$$
	must be an isomorphism. Hence $\delta^+$ is formally \'etale (\cite{TLR05}*{Corollary 2.5.10}). Combine with Proposition \ref{surj of delta} we conclude that the closed embedding $\delta^+$ is an isomorphism.
\end{proof}

\subsection{Pullback $\cZ$-cycles and $\cY$-cycles}

Now we want to relate intersection of special cycles on $\mathcal{N}_{n-1}$ to that on $\mathcal{N}_n$. We study this by pullbacking $\cZ$-cycles (resp.\ $\cY$-cycles) from $\cN_n$ to $\cN_{n-1}$ and identify them with $\cZ$-cycles (resp.\ $\cY$-cycles) on $\cN_{n-1}$. More precisely, let $x\in\mathbb{V}_{n-1}$ be a nonzero element. Then $x$ gives special cycles $\mathcal{Z}_{n-1}(x)$ and $\cY_{n-1}(x)$ on $\mathcal{N}_{n-1}$. Via the hermitian embedding \eqref{embed of dV}, we can also think of $x\in\dV_n$, which gives special cycles $\mathcal{Z}_n(x)$ and $\cY_n(x)$ on $\mathcal{N}_n$. We can pull back $\mathcal{Z}_n(x)$ (resp.\ $\cY_n(x)$) along the closed embedding $\delta^+$, which is $\mathcal{Z}_n(x)\cap\mathcal{Z}_n(u)^+\cong\mathcal{Z}_n(x)\cap\mathcal{N}_{n-1}$ as shown in the diagram below
\begin{equation}\label{diagram}
    \xymatrix{
	\mathcal{Z}_n(x)\cap\mathcal{N}_{n-1}\ar[d]\ar[r]^-{\sim}&\mathcal{Z}_n(x)\cap\mathcal{Z}_n(u)^+\ar[d]\ar[r]&\mathcal{Z}_n(x)\ar[d]\\
	\mathcal{N}_{n-1}\ar[r]^-{\delta^+}_-{\simeq}&\mathcal{Z}_n(u)^+\ar[r]&\mathcal{N}_n.
}
\end{equation}

\textbf{Pullback $\cZ$-cycle}$\colon$For a $\Spf O_{\breve{F}}$-scheme $S$, we have
\begin{align*}
    \cZ_n(x)\cap\cN_{n-1}(S)&=\left\{ \begin{aligned}(X',\iota',\lambda',\rho')\in\cN_{n-1}(S)\text{ such that if its image }\\\text{ under }\delta^+\text{ is }(X,\iota,\lambda,\rho)\in\cN_n^+(S)\text{ then there is }\\\text{ a homomorphism }
    \ucE_S\to X\text{ lifting }\rho^{-1}x\end{aligned}\right\},\\
    \cZ_{n-1}(x)(S)&=\left\{ \begin{aligned} (X',\iota',\lambda',\rho')\in\cN_{n-1}(S)\text{ such that there is }\\\text{ a homomorphism }\ucE_S\to X'\text{ lifting }(\rho')^{-1}x \end{aligned} \right\}.
\end{align*}

Then we have a morphism denoted by $\delta^+_x\colon$
\begin{align*}
    \cZ_n(x)\cap\cN_{n-1}&\lrarr\cZ_{n-1}(x)\\
    ( (X',\iota',\lambda',\rho'),\tilde{x}\colon\ucE_S\to X)&\longmapsto ( (X',\iota',\lambda',\rho'),\ucE_S\stackrel{\tilde{x}}{\lrarr}X\stackrel{\phi}{\lrarr}X'\times \ucE_S\stackrel{p_1}{\lrarr}X')
\end{align*}
where $\phi\colon X\to X'\times\ucE_S$ is the $O_F$-linear isogeny lifting $\phi_0$. The morphism is well-defined because we have the following commutative diagram over $\ovS$
$$\xymatrix{
\ubE_{\ovS}\ar[rr]^-{x}\ar[drr]_-{(x,0)}&&\dX_{n,\ovS}\ar[rr]^-{\rho^{-1}}\ar[d]^-{\phi_0}&&X_{\ovS}\ar[d]^-{\bar{\phi}}\\
&&\dX_{n-1,\ovS}\times\ubE_{\ovS}\ar[rr]^-{(\rho')^{-1}\times\operatorname{id}}&&X'_{\ovS}\times\ubE_{\ovS}.
}$$

Naively we would like the morphism $\delta_x^{+}$ to be an isomorphism for all $x\in \mathbb{V}_{n-1}$. However this is not possible. On one hand, $\mathcal{Z}_n(x)=\emptyset$ if $h(x,x)\notin \pi_0 O_{F_0}$. On the other hand, we have the following Proposition.

\begin{proposition}\label{N n-2=Z n-1 x}
If $x\in\dV_{n-1}$ has hermitian norm $h(x,x)=1$ then we have natural isomorphism
\begin{equation*}
    \cN_{n-2}\stackrel{\sim}{\lrarr}\cZ_{n-1}(x).
\end{equation*}  
\end{proposition}
\begin{proof}
By \eqref{action of unitary group}, after applying an action of $U(\dV_{n-1})$, we only need to prove this for
$$x_0=(0_{n-2},\operatorname{id}_{\ubE})\colon\ubE\to\dX_{n-1}=\dX_{n-2}\times\ubE.$$

We have a closed embedding of formal schemes
\begin{equation*}
    \delta_{\cN}\colon
   \xymatrix@R=0ex{
	   \cN_{n-2}\quad\quad \ar[r]  & \quad\quad {\cN}_{n-1}\\
		(X, \iota, \lambda, \rho) \ar@{|->}[r]  
		   &  \bigl(X \times \ucE, \iota \times \iota_{\ucE}, \lambda \times \lambda_{\ucE}, \rho \times \rho_{\ucE}\bigr)
	}.
\end{equation*}
It is direct to see that $\delta_{\cN}$ factors through $\cZ_{n-1}(x_0)$. By the same argument as in \cite{Ter13}*{Lemma 2}, we see that $\delta_{\cN}$ induces an isomorphism $\cN_{n-2}\simeq \cZ_{n-1}(x_0)$.
\end{proof}

If we choose $x\in\dV_{n-1}$ of hermitian norm $1$ then $\cZ_n(x)\cap\cN_{n-1}$ is empty while $\cZ_{n-1}(x)\simeq\cN_{n-2}$ is not. The morphism $\delta_x^+$ can not be an isomorphism in this case.

\begin{proposition}\label{identify special divisors}
	The morphism $\delta_x^+\colon\mathcal{Z}_n(x)\cap\mathcal{N}_{n-1}\to\mathcal{Z}_{n-1}(x)$ is an isomorphism if $h(x,x)\in \pi_0 O_{F_0}$.
\end{proposition}
\begin{proof}
	Write $n=2m$. From the diagram \eqref{diagram} we see that $\delta_x^+$ is \textit{a priori} a closed embedding, so again we need to show it is surjective and formally smooth.
	
	Follow Lemma \ref{D module of ubE}, we write $O_{\breve{F}}e$ for the Dieudonn\'e module of $\ubE$. From Dieudonn\'e theory the $\bar{k}$-points are described as follows$\colon$
	\begin{align*}
	\mathcal{Z}_n(x)\cap\mathcal{N}_{n-1}(\bar{k})&=\left\{
    \begin{array}{c}
	    \text{ almost }\pi\text{-modular }\dL\subset\dV_{n-1} \text{ satisfying condition }\eqref{VModd},\\
        \pi\text{-modular }\dM\subset\dV_n \text{ satisfying condition }\eqref{VM}\text{ and }\eqref{L sub M},\\
        \text{ and there exists a morphism }O_{\breve{F}}e\to \dM\text{ induced by }x
	\end{array} \right\},\\
    \mathcal{Z}_{n-1}(x)(\bar{k})&=\left\{
    \begin{array}{c}
        \text{ almost }\pi\text{-modular }\dL\subset\dV_{n-1}\text{ satisfying condition }\eqref{VModd},   \\
        \text{ and there exists a morphism }O_{\breve{F}}e\to\dL\text{ induced by }x
    \end{array}\right\}.
    \end{align*}
    
	Then the morphism $\delta_x^{+}$ just sends the pair $(\dL,\dM)$ to $\dL$ with the Dieudonn\'e morphism
	$$O_{\breve{F}}e\stackrel{x}{\lrarr} \dM\subset^1 \dL\obot O_{\breve{F}}e\stackrel{p_1}{\lrarr}\dL.$$
	
	To show surjectivity, take $\dL$ with the Dieudonn\'e morphism $x$, then the morphism $\delta^+$ provides us with the lattice $\dM$ satisfying conditions \eqref{VM} and \eqref{L sub M}. All we need to show is that the morphism $O_{\breve{F}}e\stackrel{(x,0)}{\lrarr} \dL\obot O_{\breve{F}}e$ factors through $\dM\subset^1 \dL\obot O_{\breve{F}}e$.
	
	As in \eqref{L in case 1} or \eqref{L in case 2}, we can find a standard normal basis $\sigma_i,\tau_i$ of $\dM$ such that $\pi e=\sigma_1+\frac{\pi\delta}{2}\tau_1$ and $\dL$ has normal basis $\pi^{-1}(\sigma_1-\frac{\pi\delta}{2}\tau_1),\sigma_i,\tau_i$ for $i\geq 2$. Let the image of $O_{\breve{F}}e\stackrel{x}{\lrarr}\dL$ be $a\pi^{-1}(\sigma_1-\frac{\pi\delta}{2}\tau_1)+a_2\sigma_2+b_2\tau_2+\dots+a_m \sigma_m+b_m\tau_m$, then its hermitian norm is
	$$\delta a\bar{a}-\sum_{i=2}^{m}(a_i\bar{b}_i-\bar{a}_ib_i)\pi.$$
	On the other hand, by \eqref{relation of herm forms} the hermitian norm should be $\delta\cdot h(x,x)\in\pi^2 O_{\breve{F}_0}$ as we assume $h(x,x)\in \pi_0 O_{F_0}$, so $a\in\pi O_{\breve{F}}$ and the image lies in $\dM$ as desired.
	
	To show it is formally smooth, we consider a Noetherian local $O_{\breve{F}}$-algebra $R$ on which $\pi$ is nilpotent, and a square zero ideal $I\subset R$. Let $S=R/I$ so that $R\to S$ is a thickening, equipped with the trivial nilpotent divided power structure on $I$. Suppose we have a commutative diagram consisting of solid arrows
	$$\xymatrix{
		\Spec S\ar[r]\ar[d] & \cZ_n(x)\cap\cN_{n-1}\ar[d] \\
		\Spec R\ar[r] \ar@{-->}[ur]& \cZ_{n-1}(x)
	}$$
	and we want to find a dotted arrow making the diagram commutative. In other words, we are given
	$(X',\iota',\lambda',\rho')\in \cN_{n-1}(S)$ such that if we write $\delta^+(X',\iota',\lambda',\rho')=(X,\iota,\lambda,\rho)$ then there exists a lift $x\colon\ucE_S\to X$ of $\ubE_{\ovS}\stackrel{x}{\lrarr}\mathbb{X}_{n,\ovS}\stackrel{\rho^{-1}}{\lrarr}X_{\ovS}$, together with a lift $(\widetilde{X'},\widetilde{\iota'},\widetilde{\lambda'},\widetilde{\rho'})\in\cZ_{n-1}(x)(R)$ of $(X',\iota',\lambda',\rho')$.
	
	Let $\delta^+(\widetilde{X'},\widetilde{\iota'},\widetilde{\lambda'},\widetilde{\rho'})=(\widetilde{X},\widetilde{\iota},\widetilde{\lambda},\widetilde{\rho})$. Then it suffices to show that the morphism $\ucE_S\to X$ lifts to $\ucE_R\to\widetilde{X}$. Let $\phi\colon X\to X'\times\ucE_S$ be the $O_F$-isogeny lifting $\phi_0$. Over $S$ we have a commutative diagram from Dieudonn\'e crystals$\colon$
	$$\xymatrix{
		\operatorname{Fil}^1(\ucE_S)\ar[d]\ar[r] & \mathbb{D}(\ucE_S)\ar[d]^-{\mathbb{D}(x)}\\
		\operatorname{Fil}^1(X)\ar[r]\ar[d] & \mathbb{D}(X)\ar[d]^-{\mathbb{D}(\phi)}\\
		\operatorname{Fil}^1(X')\obot\operatorname{Fil}^1(\ucE_S)\ar[r] & \mathbb{D}(X')\oplus\mathbb{D}(\ucE_S)
	}$$
	and it can be extended to a commutative diagram over $R\colon$
	$$\xymatrix{
		\operatorname{Fil}^1(\ucE_R)\ar[r]\ar@/_24pt/[dd] & \mathbb{D}(\ucE_R)\ar[d]^-{\mathbb{D}(x)_R}\\
		\operatorname{Fil}^1(\widetilde{X})\ar[r]\ar[d] & \mathbb{D}(X)_R\ar[d]^-{\mathbb{D}(\phi)_R}\ar[r] & \operatorname{Lie}(\widetilde{X})\ar[d]^-{\operatorname{Lie}(\phi)_R}\\
		\operatorname{Fil}^1(\widetilde{X'})\oplus\operatorname{Fil}^1(\ucE_R)\ar[r] & \mathbb{D}(X')_R\oplus\mathbb{D}(\ucE_R)\ar[r] & \operatorname{Lie}(\widetilde{X'})\oplus\operatorname{Lie}(\ucE_R)
	}$$
	By Grothendieck--Messing theory, it suffices to show $\mathbb{D}(x)_R$ preserves the Hodge filtration, i.e.
	$$\mathbb{D}(x)_R(\operatorname{Fil}^1(\ucE_R))\subset \operatorname{Fil}^1(\widetilde{X}).$$
	
	From \cite{LL22}*{Lemma 2.38} we know that there is a locally direct summand $R$-submodule of $\operatorname{Lie}(\widetilde{X})$ of rank $1$, denoted by $L_{\widetilde{X}}$, and by the same argument as \cite{How19}*{Proposition 4.1} we see that the image of the induced morphism
	$$\mathbb{D}(x)_R\colon \operatorname{Fil}^1(\ucE_R)\lrarr\operatorname{Lie}(\widetilde{X})$$
	lies in $L_{\widetilde{X}}$. Suppose $f$ is an $R$-basis of $\operatorname{Fil}^1(\ucE_R)$ and $e'$ is an $R$-basis of $L_{\widetilde{X}}$. Write $\mathbb{D}(x)_R(f)=r\cdot e'$ for some $r\in R$. Pick any $R$-basis $e_1,\dots,e_n$ of $\operatorname{Lie}(\widetilde{X'})\oplus\operatorname{Lie}(\ucE_R)$, and suppose the image of $e'$ under $\operatorname{Lie}(\phi)_R$ in $\operatorname{Lie}(\widetilde{X'})\oplus\operatorname{Lie}(\ucE_R)$ is given by
	$a_1 e_1+\dots +a_n e_n$ for $a_1,\dots,a_n\in R$. Then as $(\widetilde{X'},\widetilde{\iota'},\widetilde{\lambda'},\widetilde{\rho'})\in\cZ_{n-1}(x)(R)$, the image of $\operatorname{Fil}^1(\ucE_R)$ in $\operatorname{Lie}(\widetilde{X'})\oplus\operatorname{Lie}(\ucE_R)$ vanishes. Thus $r\cdot a_1=\dots=r\cdot a_n=0$.
	
	If one of $a_1,\dots,a_n$ is a unit, then we have $r=0$ so that $\mathbb{D}(x)_R$ preserves the Hodge filtration and we are done. Otherwise, let $\mathfrak{m}$ be the maximal ideal of $R$, and assume $a_1,\dots,a_n\in \mathfrak{m}$. We shall prove by contradiction that this is impossible.
	
	We make the base change to the residue field $\bar{k}$, then the image of $L_{\bar{k}}\coloneqq(L_{\widetilde{X}})_{\bar{k}}$ in $\operatorname{Lie}(\widetilde{X'})_{\bar{k}}\oplus\operatorname{Lie}(\ucE_R)_{\bar{k}}$ vanishes as we assume $a_1,\dots,a_n\in\fm$.
	
	For any standard normal basis $e_1,f_1,\dots,e_m,f_m$ of the Dieudonn\'e module $\dM$ of $X_{\bar{k}}$ such that $e_1\notin V\dM$, by Lemma \ref{describe VM} we have
	\begin{itemize}
		\item [(1)] $\operatorname{Fil}^1(X_{\bar{k}})=\langle e_1,\pi e_1,\pi e_2,\pi f_2,\dots,\pi e_m,\pi f_m\rangle_{\bar{k}}$,
		\item [(2)] $\operatorname{Lie}(X_{\bar{k}})=\langle f_1,\pi f_1, e_2, f_2,\dots, e_m,f_m\rangle_{\bar{k}}$.
	\end{itemize}
	By \cite{LL22}*{Lemma 2.42}, we see that $L_{\bar{k}}=\langle \pi f_1\rangle_{\bar{k}}$. 
 
    Now let $\dL$ be the Dieudonn\'e module of $X'_{\bar{k}}\times \ubE_{\bar{k}}$. Recall we have the relation \eqref{L sub M}
	$$N\subset^1 \dM\subset^1 \dL$$
	where $N=\pi \dL^*$. Then there are two cases, corresponding to the two cases in the proof of Proposition \ref{surj of delta}. In particular as in the proof, we can find a standard normal basis $\sigma_1,\tau_1,\dots,\sigma_m,\tau_m$ of $\dM$ with $\sigma_1\notin V\dM$ (thus $L_{\bar{k}}=\langle \pi\tau_1\rangle_{\bar{k}}$) and such that
    \begin{itemize}
        \item In the first case
    $$N=\langle \sigma_1, \tau_1,\dots, \sigma_{m-1}, \tau_{m-1}, \sigma_m, \pi \tau_m\rangle_{O_{\breve{F}}}\text{ and }\dL=\langle \sigma_1, \tau_1,\dots, \sigma_{m-1}, \tau_{m-1}, \pi^{-1}\sigma_m,\tau_m\rangle_{O_{\breve{F}}}.$$
    Then the inclusion $\dM\subset^1 \dL$ sends $\sigma_m$ to $\pi\cdot\pi^{-1}\sigma_m$, $\pi \sigma_m$ to $\pi_0 \cdot \pi^{-1}\sigma_m$ and all other bases to themselves. In particular the image of $\pi \tau_1$ is again $\pi \tau_1$, which does not vanish after base change to $\bar{k}$ (the only term that vanishes is $\pi \sigma_m$);
    \item In the second case
    $$N=\langle \sigma_1, \pi \tau_1, \sigma_2, \tau_2,\dots, \sigma_m, \tau_m\rangle_{O_{\breve{F}}} \text{ and } \dL=\langle \pi^{-1} \sigma_1, \tau_1, \sigma_2, \tau_2, \dots,\sigma_m,\tau_m\rangle_{O_{\breve{F}}}.$$
    Then the inclusion $\dM\subset^1 \dL$ sends $\sigma_1$ to $ \pi\cdot\pi^{-1}\sigma_1$ and $\pi \sigma_1$ to $\pi_0\cdot \pi^{-1 }\sigma_1$, and all other bases to themselves. In particular $\pi \tau_1$ maps to $\pi \tau_1$, which will not vanish after base change to $\bar{k}$ (the only term that vanishes is $\pi \sigma_1$). 
    \end{itemize}

    Thus it is a contradiction that the image of $L_{\bar{k}}$ will vanish. 
\end{proof}

\textbf{Pullback $\cY$-cycles}$\colon$ For a $\Spf O_{\breve{F}}$-scheme $S$, we have
$$\begin{aligned}
\cY_{n-1}(x)(S)&=\left\{\begin{array}{c}(X',\iota',\lambda',\rho')\in\cN_{n-1}(S) \text{ such that there exists a morphism }\\ \ucE_S\to (X')^{\vee}\text{ lifting }\ubE_{\ovS}\stackrel{x}{\lrarr}\dX_{n-1,\ovS}\stackrel{(\rho')^{-1}}{\lrarr} X'_{\ovS}\stackrel{\lambda'}{\lrarr}(X')^{\vee}_{\ovS}.\end{array}\right\},\\
\cN_{n-1}\cap\cY_n(x)(S)&=\left\{\begin{array}{c}(X',\iota',\lambda',\rho')\in\cN_{n-1}(S)\text{ such that if its image in }\cN_n^+(S)\\ \text{ is }(X,\iota,\lambda,\rho)\text{ then there exists a morphism }\ucE_S\to X^{\vee}\text{ lifting }\\\ubE_{\ovS}\stackrel{x}{\lrarr}\dX_{n,\ovS}\stackrel{\rho^{-1}}{\lrarr}X_{\ovS}\stackrel{\lambda}{\lrarr}X^{\vee}_{\ovS}.\end{array}\right\}.
\end{aligned}$$
Then we have a morphism denoted by $\gamma_x^+\colon$
$$\begin{aligned}
\cY_{n-1}(x)&\lrarr\cN_{n-1}\cap\cY_n(x)\\
((X',\iota',\lambda',\rho'),\widetilde{x}\colon\ucE_S\to(X')^{\vee})&\longmapsto((X',\iota',\lambda',\rho'),\ucE_S\stackrel{(\widetilde{x},0)}{\lrarr}(X')^{\vee}\times\ucE_S^{\vee}\stackrel{\phi^{\vee}}{\lrarr}X^{\vee})
\end{aligned}$$
where $\phi\colon X\to X'\times\ucE_{S}$ is the $O_F$-linear isogney lifting $\phi_0$. The morphism is well-defined because we have the following commutative diagram over $\ovS$
$$\xymatrix{
	\ubE_{\ovS} \ar[rr]^-{(x,0)}\ar[drr]^-{x}
    &&\dX_{n-1,\ovS}\times\ubE_{\ovS} \ar[rr]^-{(\rho')^{-1}\times \operatorname{id}}
    &&X'_{\ovS}\times\ubE_{\ovS}\ar[rr]^-{\lambda'\times\lambda_{\ubE}}
    &&(X'_{\ovS})^{\vee}\times\ubE_{\ovS}^{\vee}\ar[d]^-{\phi_{\ovS}^{\vee}}\\
	&&\dX_{n,\ovS}\ar[u]_-{\phi_0}\ar[rr]^-{\rho^{-1}}&&X_{\ovS}\ar[rr]^-{\lambda}\ar[u]^-{\phi_{\ovS}}
	&&X_{\ovS}^{\vee}.}$$

\begin{proposition}\label{identification of Y-cycles}
	The morphism $\gamma_x^+$ is an isomorphism for any nonzero $x\in\dV_{n-1}$.
\end{proposition}
\begin{proof}
	The morphism $\gamma_x^+$ completes the closed embedding diagram
	$$\xymatrix{\cY_{n-1}(x)\ar[rr]^-{\gamma_x^+}\ar[dr]&&\cN_{n-1}\cap\cY_n(x)\ar[ld]\\&\cN_{n-1}}$$
	so it is also a closed embedding. It remains to show it is surjective and formally smooth.
	
	We firstly show it is surjective. We shall consider the $\bar{k}$-points and the argument for arbitrary algebraically closed field is the same. Let $O_{\breve{F}}e$ denote the Dieudonn\'e module of $\ubE$ as in Lemma \ref{D module of ubE}. We have the following description in terms of Dieudonn\'e module$\colon$
	$$\begin{aligned}
	\cN_{n-1}\cap\cY_n(x)(\bar{k})&=\left\{\begin{array}{c}
	\text{ almost }\pi\text{-modular }\dL\subset\dN_{n-1}\text{ satisfying condition }\eqref{VModd},\\
	\pi\text{-modular }\dM\subset\dN_n\text{ satisfying condition }\eqref{VM}\text{ and }\eqref{L sub M},\\
    \text{ and there exists a morphism }O_{\breve{F}}e\to \dM^*\text{ induced by }x.
	\end{array}\right\},\\
	\cY_{n-1}(x)(\bar{k})&=\left\{\begin{array}{c}
	\text{ almost }\pi\text{-modular }\dL\subset\dN_{n-1}\text{ satisfying condition }\eqref{VModd},\\
	\text{and there exists a morphism }O_{\breve{F}}e\to \dL^*\text{ induced by }x.
	\end{array}\right\}.
	\end{aligned}$$
	The morphism $\gamma_x^+$ just sends $(\dL,x\colon O_{\breve{F}}e\to \dL^*)$ to $(\dL,\dM,O_{\breve{F}}e\stackrel{(x,0)}{\to}\dL^*\obot O_{\breve{F}}e \subset^1\pi^{-1}\dM=\dM^*)$ where $\dM=\delta^+(\dL)$. To show it is surjective, it suffices to show given a morphism $O_{\breve{F}}e\to \dM^*$ induced by $x$, it actually factors through $O_{\breve{F}}e\to \dL^*$.
	
	Recall any $x\in\dV_n$ induces a morphism of rational Dieudonn\'e modules $x\colon\breve{F}e\to\dN_n$, and we have the following relation from \eqref{relation of herm forms}
	$$h(xe,ye)=h(e,e)h(x,y)=\delta h(x,y).$$
	
	Consider the inclusion $\dL^*\obot O_{\breve{F}} e \subset^1 \dM^*$. The element $u\in\dV_n$ induces a morphism $\breve{F}e\to\dN_n$ sending $e$ to $u(e)=\pi e$. As we are taking $x\in\dV_{n-1}$, we have $h(x,u)=0$. This implies $h(x(e),u(e))=0$ so $x(e)$ lies in the orthogonal complement of $e\in \dM^*$, which is just $\dL^*$. Thus the morphism $x\colon O_{\breve{F}}e\to \dM^*$ naturally factors through $\dL^*$.
	
	Next we show the morphism $\gamma_x^+$ is formally smooth. Consider a Noetherian local $O_{\breve{F}}$-algebra $R$ on which $\pi$ is nilpotent, and a square zero ideal $I\subset R$. Let $S=R/I$ so that $R\to S$ is a thickening, equipped with the trivial nilpotent divided power structure on $I$. Suppose we have a commutative diagram consisting of solid arrows
	$$\xymatrix{
		\Spec S\ar[r]\ar[d]&\cY_{n-1}(x)\ar[d]\\
		\Spec R\ar[r]\ar@{-->}[ur]&\cN_{n-1}\cap\cY_n(x)
	}$$
	and we want to find a dotted arrow making the diagram commutative.
	
	Equivalently, we have $(X',\iota',\lambda',\rho')\in\cY_{n-1}(x)(S)$ and a lift $(\widetilde{X'},\widetilde{\iota'},\widetilde{\lambda'},\widetilde{\rho'})\in\cN_{n-1}(R)$ such that under $\delta^+$ its image $(\widetilde{X},\widetilde{\iota},\widetilde{\lambda},\widetilde{\rho})\in\cY_n(x)(R)$. We want to show that $(\widetilde{X'},\widetilde{\iota'},\widetilde{\lambda'},\widetilde{\rho'})\in\cY_{n-1}(x)(R)$.
	
	Write $\delta^+(X',\iota',\lambda',\rho')=(X,\iota,\lambda,\rho)$. Let $\phi\colon X\to X'\times \ucE_S$ and $\widetilde{\phi}\colon\widetilde{X}\to\widetilde{X'}\times\ucE_R$ be the $O_F$-linear isogenies lifting $\phi_0$. Over $S$ we have a commutative diagram from Dieudonn\'e crystals$\colon$
	$$\xymatrix{
		\operatorname{Fil}^1(\ucE_S)\ar[r]\ar[d]&\dD(\ucE_S)\ar[d]^-{\dD(x)\oplus 0}\\
		\operatorname{Fil}^1((X')^{\vee})\oplus\operatorname{Fil}^1(\ucE_S)\ar[r]\ar[d]&\dD((X')^{\vee})\oplus\dD(\ucE_S)\ar[d]^-{\dD(\phi^{\vee})}\\
		\operatorname{Fil}^1(X^{\vee})\ar[r]&\dD(X^{\vee})
	}$$
	and it can be extended to a commutative diagram over $R\colon$
	$$\xymatrix{
		\operatorname{Fil}^1(\ucE_R)\ar[r]\ar@/_60pt/[dd]&\dD(\ucE_R)\ar[d]^-{\dD(x)_R\oplus 0}\\
		\operatorname{Fil}^1(\widetilde{X'}^{\vee})\oplus\operatorname{Fil}^1(\ucE_R)\ar[d]\ar[r]&\dD((X')^{\vee})_R\oplus\dD(\ucE_R)\ar[r]\ar[d]^-{\dD(\widetilde{\phi}^{\vee})}&\operatorname{Lie}(\widetilde{X'}^{\vee})\oplus\operatorname{Lie}(\ucE_R)\ar[d]^-{\operatorname{Lie}(\widetilde{\phi}^{\vee})}\\
		\operatorname{Fil}^1(\widetilde{X}^{\vee})\ar[r]&\dD({X}^{\vee})_R\ar[r]&\operatorname{Lie}(\widetilde{X}^{\vee}).
	}$$
	Suppose $f$ is an $R$-basis of $\operatorname{Fil}^1(\ucE_R)$ and consider its image $f'$ under $\dD(x)_R$ in $\operatorname{Lie}(\widetilde{X'}^{\vee})$. By Grothendieck--Messing theory, it suffices to show the image is zero, as then we will have a lift of the morphism $\ucE_S\to (X')^{\vee}$ over $S$ to $\ucE_R\to\widetilde{X'}^{\vee}$ over $R$.
	
	
	Consider the reduction to residue field. Then there are two cases corresponding to the two cases in the proof of Proposition \ref{surj of delta}$\colon$
	\item [\textbf{Case 1}$\colon$] As in the proof, we can find a standard normal basis $\sigma_1,\tau_1,\dots,\sigma_m,\tau_m$ of the Dieudonn\'e module $M^*$ of $X^{\vee}_{\bar{k}}$ with $\sigma_1\notin VM^*$ such that element $e\in M^*$ is given by $\sigma_m+\frac{\pi\delta}{2}\tau_m$. Let $\varphi=\sigma_m-\frac{\pi\delta}{2}\tau_m$. By Lemma \ref{describe VM} we have
    $$\operatorname{Lie}(X^{\vee}_{\bar{k}})=\langle \tau_1,\pi\tau_1,\sigma_2,\tau_2,\dots,\sigma_m,\tau_m\rangle_{\bar{k}}.$$
    Denote the Dieudonn\'e module of $(X')^{\vee}_{\bar{k}}$ by $L^*$. Then by \eqref{define L out of M} we have
    $$L^*=\langle e\rangle^{\perp}=\langle \sigma_1,\tau_1,\dots,\sigma_{m-1},\tau_{m-1},\varphi\rangle_{O_{\breve{F}}},$$
    by \eqref{V and dual} and \eqref{VL in case 1} we have
    \begin{align*}VL^*=\pi_0(VL)^*&=\langle \sigma_1,\pi_0\tau_1,\pi\sigma_2,\dots,\pi\tau_{m-1},\pi\varphi\rangle_{O_{\breve{F}}},\\
	\operatorname{Lie}((X')^{\vee}_{\bar{k}})=L^*/VL^*&=\langle \tau_1,\pi\tau_1,\sigma_2,\dots,\tau_{m-1},\varphi\rangle_{\bar{k}}.\end{align*}
 
	The map $\dD(\phi^{\vee})_{\bar{k}}\colon \dD((X')^{\vee}_{\bar{k}})\oplus\dD(\ubE)\to\dD(X^{\vee}_{\bar{k}})$ then can be described as
	\begin{equation}\label{the map}
	\begin{aligned}
	\sigma_i&\mapsto \sigma_i, \quad\quad\quad\quad\quad\tau_i\mapsto \tau_i\text{ for }1\leq i\leq m-1,\\
	\varphi&\mapsto \sigma_m-\frac{\pi\delta}{2}\tau_m, \quad e\mapsto \sigma_m+\frac{\pi\delta}{2}\tau_m.
	\end{aligned}\end{equation}
	We may lift the $\bar{k}$-basis $\sigma_1,\tau_1,\dots,\sigma_m,\tau_m,\varphi,e$ to an $R$-basis, denoted by $\widetilde{\sigma}_1,\widetilde{\tau}_1,\dots,\widetilde{\sigma}_m,\widetilde{\tau}_m,\widetilde{\varphi},\widetilde{e}$ such that the map $\dD(\widetilde{\phi}^{\vee})$ has the same description as \eqref{the map} with every element replaced by its $R$-lift, and
    $$\Lie(\widetilde{X'}^{\vee})=\langle \widetilde{\tau}_1,\pi\widetilde{\tau}_1,\widetilde{\sigma}_2,\dots,\widetilde{\tau}_{m-1},\widetilde{\varphi}\rangle_R.$$
    Write $f'$ in terms of the basis as$$f'=a\widetilde{\tau}_1+a'\pi\widetilde{\tau}_1+a_2\widetilde{\sigma}_2+b_2\widetilde{\tau}_2+\dots+a_{m-1}\widetilde{\sigma}_{m-1}+b_{m-1}\widetilde{\tau}_{m-1}+\alpha\widetilde{\varphi}.$$
	Then under the map $\operatorname{Lie}(\widetilde{\phi}^{\vee})$, it maps to
	$$a\widetilde{\tau}_1+a'\pi\widetilde{\tau}_1+a_2\widetilde{\sigma}_2+b_2\widetilde{\tau}_2+\dots+a_{m-1}\widetilde{\sigma}_{m-1}+b_{m-1}\widetilde{\tau}_{m-1}+\alpha\widetilde{\sigma}_m-\alpha\frac{\pi\delta}{2}\widetilde{\tau}_m$$
	which should be zero as we have the morphism $\ucE_R\to\widetilde{X}^{\vee}$. This implies that $a=a'=a_2=b_2=\dots=a_{m-1}=b_{m-1}=\alpha=0$ and $f'=0$.
	
	\item [\textbf{Case 2}$\colon$] Again we can find a standard normal basis $\sigma_1,\tau_1,\dots,\sigma_m,\tau_m$ of $M^*$ with $\sigma_1\notin VM^*$ such that $e\in M^*$ is given by $\sigma_1+\frac{\pi\delta}{2}\tau_1$. Let $\varphi=\sigma_1-\frac{\pi\delta}{2}\tau_1$. Denote the Dieudonn\'e module of $(X')^{\vee}_{\bar{k}}$ by $L^*$. Then
	$$\begin{aligned}
	L^*&=\langle \varphi,\sigma_2,\tau_2,\dots,\sigma_m,\tau_m\rangle_{O_{\breve{F}}},\\
	VL^*&=\pi L^*=\langle \pi\varphi,\pi\sigma_2,\pi\tau_2,\dots,\pi\sigma_m,\pi\tau_m\rangle_{O_{\breve{F}}},\\
	\operatorname{Lie}((X')^{\vee}_{\bar{k}})&=\langle \varphi,\sigma_2,\tau_2,\dots,\sigma_m,\tau_m\rangle_{\bar{k}},\\
	\operatorname{Lie}(X^{\vee}_{\bar{k}})&=\langle \tau_1,\pi\tau_1,\sigma_2,\tau_2,\dots,\sigma_m,\tau_m\rangle_{\bar{k}}.
	\end{aligned}$$
	The map $\dD(\phi^{\vee})_{\bar{k}}\colon \dD((X')^{\vee}_{\bar{k}})\oplus\dD(\ubE)\to\dD(X^{\vee}_{\bar{k}})$ then can be described as
	\begin{equation}\label{the map in case 2}
	\begin{aligned}
	\varphi&\mapsto \sigma_1-\frac{\pi\delta}{2}\tau_1, \quad e\mapsto \sigma_1+\frac{\pi\delta}{2}\tau_1,\\
	\sigma_i&\mapsto \sigma_i, \quad\quad\quad\quad\tau_i\mapsto \tau_i\text{ for }2\leq i\leq m.
	\end{aligned}\end{equation}
	We may lift the $\bar{k}$-basis to an $R$-basis denoted by $\widetilde{\sigma}_1,\widetilde{\tau}_1,\dots,\widetilde{\sigma}_m,\widetilde{\tau}_m,\widetilde{\varphi},\widetilde{e}$ such that the map $\dD(\widetilde{\phi}^{\vee})$ has the same description as \eqref{the map in case 2} with every element replaced by its $R$-lift. Write $f'$ in terms of the basis as $$f'=\alpha\widetilde{\varphi}+a_2\widetilde{\sigma}_2+b_2\widetilde{\tau}_2+\dots+a_m\widetilde{\sigma}_m+b_m\widetilde{\tau}_m.$$
	Then under the map $\operatorname{Lie}(\widetilde{\phi}^{\vee})$, it maps to
	$$\alpha\widetilde{\sigma}_1-\alpha\frac{\pi\delta}{2}\widetilde{\tau}_1+a_2\widetilde{\sigma}_2+b_2\widetilde{\tau}_2+\dots+a_m\widetilde{\sigma}_m+b_m\widetilde{\tau}_m$$
	which should be zero as we have the morphism $\ucE_R\to\widetilde{X}^{\vee}$. This implies that $\alpha=a_2=b_2=\dots=a_m=b_m=0$ and $f'=0$.

	This finishes the proof that $\gamma_x^+$ is formally smooth and is an isomorphism.
\end{proof}

\begin{proposition}\label{geometric reduction}
    Let $L\subset\dV_{n-1}$ be an $O_F$-lattice. View $L$ as an $O_F$-lattice (of rank $n-1$) in $\dV_n$ via the hermitian embedding $\dV_{n-1}\to\dV_n$ and set $L^{\#}=L\obot\langle f\rangle_{O_F}$. Then
    $$\Int_{n-1,\cY}(L)=\frac{1}{2}\Int_{n,\cY}(L^{\#}).$$
    In particular, the intersection number $\Int_{n-1,\cY}(L)$ is independent of a choice of basis of $L$.
\end{proposition}
\begin{proof}
    From Lemma \ref{y=pi z even} and Theorem \ref{identify lower RZ space with special divisor} we have $$\cY_n(f)\simeq\cZ_n(\pi f)=\cZ_n(u)\text{ , }\cY_n(f)\cap\cN_n^+\simeq\cZ_n(u)^+\simeq\cN_{n-1}.$$
    From Proposition \ref{identification of Y-cycles} we have
    $$\cN_n^+\cap\cY_n(f)\cap\cY_n(x)\simeq\cY_{n-1}(x)$$
    for any nonzero $x\in\dV_{n-1}$. The same is true for $\cN_n^-$. The results then follow from definition.
\end{proof}

Consider nonzero $x\in\dV_{n-1}$ and assume $h(x,x)\in \pi_0 O_{F_0}$. By Lemma \ref{y=pi z even} and Proposition \ref{identification of Y-cycles} we have 
$$\cY_{n-1}(x)\simeq\cN_{n-1}\cap\cY_n(x)\simeq\cN_{n-1}\cap\cZ_n(\pi x).$$
By Proposition \ref{identify special divisors} we have $\cN_{n-1}\cap\cZ_n(\pi x)\simeq\cZ_{n-1}(\pi x)$. Combine the results we get

\begin{corollary}\label{y=pi z for odd}
	For any nonzero $x\in\dV_{n-1}$ such that $h(x,x)\in \pi_0 O_{F_0}$, we have identification
	$$\cY_{n-1}(x)\simeq\cZ_{n-1}(\pi x).$$
\end{corollary}

Explicitly the isomorphism is given by
$$(X',\iota',\lambda',\rho')\in\cN_{n-1}(S), x'\colon\ucE_S\to (X')^{\vee}\longmapsto (X',\iota',\lambda',\rho')\in\cN_{n-1}(S), \ucE_S\stackrel{x'}{\lrarr} (X')^{\vee}\stackrel{r}{\lrarr} X'$$
where $r\colon(X')^{\vee}\to X'$ is the unique morphism such that $r\circ\lambda'=\iota'(\pi)$.

\section{Proof of main Theorem}\label{proof of main theorem}
In this section we collect results in previous sections and prove Theorem \ref{main theorem}.

\begin{itemize}
    \item [--] When $n$ is even, the Theorem is proved in Corollary \ref{main theorem for even}.

    \item [--] When $n=1$, the condition \eqref{odd spin cond} is redundant and $\cN_1\simeq\Spf O_{\breve{F}}$. As the polarization is principal, $\cZ$-cycles and $\cY$-cycles are the same. Let $L=\langle \ell\rangle_{O_F}\subset\dV_1$ be a hermitian $O_F$-lattice with $\val_{\pi_0}h(\ell,\ell)=a$. Set $L^{\#}=L\obot I_1^{-1}$. Then on analytic side, by Proposition \ref{analytic reduction}, we have $\partial\Den(L)=\frac{1}{2}\partial\Den(L^{\#})$ and by \cite{LL22}*{Lemma 2.43}, $\partial\Den(L^{\#})=2(a+1)$ and so $\partial\Den(L)=a+1$. On the geometric side, by the theory of canonical lifting (\cite{Gr86}) we have $\Int_{\cY}(L)=a+1$. Hence $\Int_{\cY}(L)=\partial\Den(L)$.
    
    \item [--] Now assume $n\geq 3$ is odd. Let $L\subset\dV_n$ be any $O_F$-lattice. View it as in $\dV_{n+1}=\dV_n\obot\langle f\rangle_F$ via \eqref{embed of dV} and set $L^{\#}=L\obot\langle f\rangle_{O_F}$.
    
    On the analytic side, by Proposition \ref{analytic reduction} we have
    $$\partial\Den(L)=\frac{1}{2}\partial\Den(L^{\#}).$$

On the geometric side, by Proposition \ref{geometric reduction} we have
    $$\Int_{n,\cY}(L)=\frac{1}{2}\Int_{n+1,\cY}(L^{\#}).$$

By Corollary \ref{main theorem for even} we have $\partial\Den(L^{\#})=\Int_{n+1,\cY}(L^{\#})$. Thus we get $\partial\Den(L)=\Int_{n,\cY}(L)$. The Theorem is fully proved.
\end{itemize}

\end{document}